\documentclass[final,3p]{elsarticle}
\usepackage{amsmath,amssymb,amsthm,mathtools}
\usepackage{enumitem}
\usepackage[hidelinks]{hyperref}
\biboptions{sort&compress}
\hypersetup{
  pdftitle={Fractional Volterra-type operators from Bergman spaces with two-sided doubling weights to Hardy spaces},
  pdfauthor={Xiaolin Zhu and Feng Guo},
  pdfsubject={Boundedness and compactness of fractional Volterra-type operators},
  pdfkeywords={fractional Volterra-type operator, weighted Bergman space, Hardy space, two-sided doubling weight, Carleson measure, tent space}
}
\allowdisplaybreaks

\journal{Journal of Mathematical Analysis and Applications}

\numberwithin{equation}{section}

\newtheorem{theorem}{Theorem}[section]
\newtheorem{proposition}[theorem]{Proposition}
\newtheorem{lemma}[theorem]{Lemma}
\newtheorem{corollary}[theorem]{Corollary}
\newtheorem{remark}[theorem]{Remark}

\newcommand{\D}{\mathbb D}
\newcommand{\T}{\mathbb T}
\newcommand{\C}{\mathbb C}
\newcommand{\N}{\mathbb N}
\newcommand{\Z}{\mathbb Z}
\newcommand{\calD}{\mathcal D}
\newcommand{\calP}{\mathcal P}
\newcommand{\dA}{\,dA}
\newcommand{\dxi}{\,|d\xi|}
\newcommand{\RL}{\mathrm{RL}}
\newcommand{\IRL}{I^{\RL}}
\newcommand{\Ihat}{\widehat I^{\RL}}
\DeclareMathOperator{\Id}{Id}

\begin{document}
\begin{frontmatter}

\title{Fractional Volterra-type operators from Bergman spaces with
two-sided doubling weights to Hardy spaces}

\author[aff1]{Xiaolin Zhu}
\ead{xiaolinzhu@tynu.edu.cn}

\author[aff2]{Feng Guo\corref{cor1}}
\ead{70207994@nuaa.edu.cn}
\cortext[cor1]{Corresponding author}

\affiliation[aff1]{
  organization={Taiyuan Normal University, Taiyuan, P. R. China}
  }

\affiliation[aff2]{
  organization={Nanjing University of Aeronautics and Astronautics, Nanjing, P. R. China}}

\begin{abstract}
For \(\omega\in\calD\), we give necessary and sufficient symbol conditions
for the boundedness and compactness of the Riemann--Liouville family
\(V^\varphi_{\alpha,\beta}:A^p_\omega\to H^q\) for every
\(0<p,q<\infty\) and admissible \(\alpha,\beta>0\): without any additional
weight assumption when \(\alpha\ge\beta\), and, when \(\alpha<\beta\), under
the condition
\(p(\beta-\alpha)<d_-(\omega)\), where \(d_-(\omega)\) is the critical
reverse-doubling exponent.  We prove that this inequality is exactly
equivalent to the naturally shifted source weight retaining the two-sided
doubling and tail geometry required by the reduction; an integrable
logarithmic example shows that the endpoint fails.  The extension from power
weights is not formal, because the reduction replaces \(\omega\) by
\(\omega(z)(1-|z|^2)^{p(\alpha-\beta)}\), and a negative shift may destroy
even integrability.  Combining moment-induced Littlewood--Paley theory for
doubling Bergman spaces with an exact coefficient-multiplier comparison, we
establish the required Riemann--Liouville transfer, including the
finite-dimensional exceptional modes at positive integral shift orders.
A finite-codimensional range decomposition and a fractional \(g\)-function
reduction then reduce both questions to a single weighted area-map problem.
Within the admissible parameter set, the effects of \(\alpha\) are absorbed
by the shifted source weight and bounded model corrections, so the final
symbol conditions depend on \(\beta\).  The thresholds \(p=2\) and \(p=q\)
yield pointwise, Carleson, tent-integral, and non-tangential maximal criteria.
In the tent-integral range, boundedness already implies compactness.
\end{abstract}


\end{frontmatter}

\section{Introduction and main results}

For \(g\in H(\D)\), the classical Volterra-type integration operator and its
product identity are
\[
 J_gf(z)=\int_0^z f(\zeta)g'(\zeta)\,d\zeta,\qquad
 f(z)g(z)-f(0)g(0)=\int_0^z f'(\zeta)g(\zeta)\,d\zeta+J_gf(z).
\]
Thus \(J_g\) is an analytic paraproduct whose mapping properties reflect the
interaction between the boundary behavior of \(g\) and the underlying
function-space geometry \cite{AlemanCascanteFabregaPascuasPelaez2022}.
Pommerenke characterized its boundedness on \(H^2\) by \(g\in\mathrm{BMOA}\)
\cite{Pommerenke1977}; subsequent Hardy-space results include
\cite{AlemanSiskakis1995,AlemanCima2001}.  The weighted Bergman theory began
with \cite{AlemanSiskakis1997}; see also
\cite{Siskakis2006,PauPelaez2010,PelaezRattya2014}.

Mapping a Bergman space into a Hardy space sends an interior area norm to a
boundary norm, and the criteria change with the relative positions of \(p\),
\(q\), and the square-function exponent \(2\).  For power weights, boundedness
outside \(0<q<\min\{2,p\}\) was obtained in \cite{Wu2011}, and the remaining
range in \cite{MiihkinenPauPeralaWang2020}; compactness, essential norms, and
operator ideals were treated in \cite{ChenPauWang2021}.
The full boundedness and compactness classification for two-sided doubling
source weights is \cite[Theorems~1.1--1.2]{DuanWangWang2021}.

Fractional integration and differentiation have long been linked to weighted
Bergman norms \cite{BuckleyKoskelaVukotic1999}.  The two-parameter
Riemann--Liouville Volterra family used here was introduced in
\cite[Main Theorem~1.1 and formula~(1.18)]{FangGuoHouZhu2025Studia}; its
Hardy-to-Hardy boundedness for all \(0<p,q<\infty\), \(\alpha>0\), and
\(\alpha>\beta\) is
\cite[Main Theorem~1.6]{FangGuoHouZhu2025Studia}.  For the source weight
\((1-|z|^2)^\gamma\), \(\gamma>-1\), its Bergman-to-Hardy boundedness was
characterized for all \(0<p,q<\infty\) and admissible \(\alpha,\beta>0\)
under \(\gamma+p(\alpha-\beta)>-1\) in
\cite[Theorem~3]{FangGuoHouZhuBergmanHardy}.  A different moment-induced
fractional Hardy-space family appears in
\cite{BellavitaMorenoNikolaidisPelaez2026}.

Passing from power weights to \(\calD\) is not automatically compatible with
the Riemann--Liouville reduction, which forces the change of source weight
\[
        \omega\longmapsto
        \nu_{\alpha,\beta,p}
        =\omega(1-|z|^2)^{p(\alpha-\beta)}.                 \tag{1.1}
\]
For a power weight, this is again a power weight, so the change is hidden in
one exponent.  A general two-sided doubling weight has no prescribed
pointwise profile, and the stability of its tail geometry under (1.1) becomes
part of the operator problem.  This is precisely why \(\calD\) is the relevant
test class: it is the exact class of radial weights for the integer-order
Littlewood--Paley derivative norm equivalence
\cite[Theorem~5]{PelaezRattya2021}, while allowing weights far beyond the
standard power scale.  Littlewood--Paley equivalences for moment-induced
fractional derivatives were obtained in \cite{PelaezDeLaRosa2022}, and a
broader two-weight moment-ratio formulation was subsequently proved in
\cite{PeralaRattyaWang2025}.  The latter operators have nonvanishing positive
moment ratios.  They therefore do not reproduce the finite kernel of the
Riemann--Liouville coefficient operator at positive integral shift orders, nor
do they coincide with its finitely many low coefficients in general.  The
bridge required by the present operator problem is consequently an
exact-coefficient adaptation: a multiplier comparison on \(A^p_\rho\),
explicit treatment of the exceptional modes, and a finite-codimensional
decomposition that also transports compactness.

The stability issue has a genuine one-sided obstruction.  If
\(\alpha\ge\beta\), the twist in (1.1) has nonnegative power and preserves
\(\calD\), together with the required tail equivalence.  If
\(\alpha<\beta\), a power-weight example shows that the same twist can destroy
even integrability.  Stability under sufficiently small negative boundary
powers follows by specializing \cite[Lemma~2]{PelaezRattya2023BMO} to an
auxiliary constant weight, but that qualitative fact does not determine the
full range allowed by a given weight.
Consequently, the fractional problem on \(\calD\) is not a formal consequence
of either the classical weighted theory or the fractional power-weight
theory.  We identify the exact range in which the shifted-weight mechanism is
available: it is governed by a critical reverse-doubling exponent of
\(\omega\), with an integrable logarithmic example showing that the endpoint
must still be excluded.

We first characterize the negative weight shift by its critical
reverse-doubling exponent, thereby replacing the auxiliary shifted-weight
hypothesis by an explicit sharp threshold.  We then derive the exact
Riemann--Liouville transfer needed for the operator problem from the
moment-induced Littlewood--Paley theory, for every \(0<p<\infty\), retaining
the finite-dimensional exceptional modes at positive integral shift orders.
A finite-codimensional range decomposition and a fractional \(g\)-function
reduction reduce both boundedness and compactness to a single weighted
area-map problem.  Within the admissible parameter set, the factors involving
\(\alpha\) cancel from the four final symbol conditions.  The conditions
therefore depend on \(\beta\) and
\(G_\varphi=\IRL_{-\beta}\varphi\); the remaining effects of \(\alpha\) are
absorbed by the shifted source weight, bounded coefficient multipliers, and
finite-dimensional corrections.

The two thresholds in the resulting classification are forced by the reduced
area map.  The transition at \(p=2\) reflects its inner \(L^2\)-geometry, and
the transition at \(p=q\) separates the ranges \(q\ge p\) and \(q<p\).
Together they divide the exponent plane into four regimes governed by a
pointwise estimate, a Carleson measure, a finite-exponent tent integral, and a
non-tangential maximal condition.  We obtain these boundedness and compactness
criteria for all \(0<p,q<\infty\) and admissible \(\alpha,\beta>0\):
unconditionally when \(\alpha\ge\beta\), and under the explicit condition
\(p(\beta-\alpha)<d_-(\omega)\) when \(\alpha<\beta\).  In the
tent-integral range \(p>\max\{2,q\}\), finite tent integrability already forces
the boundary tails to vanish, so boundedness and compactness coincide.  The
remaining ranges require the corresponding pointwise, Carleson, or maximal
vanishing condition.

Let \(\D=\{z\in\C:|z|<1\}\), let \(\T=\partial\D\), and let
\(H(\D)\) denote the analytic functions on \(\D\).  We use
\(\N=\{1,2,\ldots\}\) and \(\Z=\{0,\pm1,\pm2,\ldots\}\).  We normalize area
measure by \(dA(z)=dx\,dy/\pi\), and \(|d\xi|\) denotes normalized arc length
on \(\T\).  If
\(f(z)=\sum_{n\ge0}\widehat f(n)z^n\), then \(\widehat f(n)\) is its
\(n\)-th Taylor coefficient.  We write \(X\lesssim Y\) when
\(X\le CY\) with a constant independent of the variables under consideration,
and \(X\asymp Y\) when both inequalities hold.
We also write
\[
        \delta(z)=1-|z|^2,
        \qquad z\in\D.
\]

A radial weight is the radial extension to \(\D\) of a nonnegative function
\(\omega\in L^1([0,1))\) whose tail
\[
        \widehat\omega(r)=\int_r^1\omega(s)\,ds,
        \qquad 0\le r<1,
\]
is positive for every \(r<1\).  We also write
\(\widehat\omega(z)=\widehat\omega(|z|)\).  The forward doubling class
\(\widehat{\calD}\) consists of those radial weights for which there is a
constant \(C\ge1\) such that
\[
        \widehat\omega(r)
        \le C\widehat\omega\!\left(\frac{1+r}{2}\right),
        \qquad 0\le r<1,
\]
and the reverse doubling class \(\check{\calD}\) consists of those for which
there are \(K>1\) and \(C>1\) such that
\[
        \widehat\omega(r)
        \ge C\widehat\omega\!\left(1-\frac{1-r}{K}\right),
        \qquad 0\le r<1.
\]
Throughout,
\[
        \calD=\widehat{\calD}\cap\check{\calD}
\]
is the class of two-sided doubling radial weights.
For \(\omega\in\calD\), define its critical reverse-doubling exponent by
\[
d_-(\omega)=\sup\left\{b>0:\ \exists\,C_b\ge1,\quad
\widehat\omega(s)\le C_b
\left(\frac{1-s}{1-r}\right)^b\widehat\omega(r),\
0\le r\le s<1\right\}.
\]
Here \(C_b\) is independent of \(r,s\).
By \cite[Lemma~B]{PelaezDeLaRosa2022}, \(d_-(\omega)>0\); the notation
\(d_-(\omega)\) is introduced here for the endpoint of the admissible
reverse-doubling exponents.

For \(0<p<\infty\), the weighted Bergman space induced by \(\omega\) is
\[
        A^p_\omega=L^p_a(\omega)
        =\left\{f\in H(\D):
        \|f\|_{A^p_\omega}^p=
        \int_{\D}|f(z)|^p\omega(z)\dA(z)<\infty\right\}.
\]
For \(0<q<\infty\), the Hardy space \(H^q\) consists of the functions
\(f\in H(\D)\) for which
\[
        \|f\|_{H^q}
        =\sup_{0<r<1}
        \left(\int_{\T}|f(r\xi)|^q\,|d\xi|\right)^{1/q}<\infty.
\]

We fix the Stolz region
\[
        \Gamma(\xi)=\{z\in\D:|z-\xi|<2(1-|z|)\},
        \qquad \xi\in\T.
\]
For an arc \(I\subset\T\), with normalized length \(|I|\), set
\[
        S(I)=\{re^{it}:e^{it}\in I,\ 1-|I|\le r<1\}.
\]
A positive Borel measure \(\mu\) on \(\D\) is a Carleson measure if
\[
        \sup_{I\subset\T}\frac{\mu(S(I))}{|I|}<\infty,
\]
and it is a vanishing Carleson measure if
\[
        \lim_{\eta\to0^+}
        \sup_{0<|I|\le\eta}\frac{\mu(S(I))}{|I|}=0.
\]

The operator is defined through the following Riemann--Liouville
realization.  Put \(\Z_- =\{-1,-2,\ldots\}\), choose
\(\operatorname{Arg}z\in(-\pi,\pi]\), and for \(c\in\C\setminus\Z_-\) set
\[
        p_c(z)=\exp\!\bigl(c(\log|z|+i\operatorname{Arg}z)\bigr),
        \qquad z\in\D\setminus\{0\}.
\]
Let
\[
        \mathcal F
        =\{p_cg:c\in\C\setminus\Z_-,\ g\in H(\D)\}.
\]
All elements of \(\mathcal F\) are regarded as single-valued functions on
\(\D\setminus\{0\}\); in this identification, functions in \(H(\D)\) are
restricted to the punctured disc.  Because the principal argument jumps across
the negative real axis, elements with nonintegral powers may be discontinuous;
moreover, \(\mathcal F\) is a class of functions of the form \(p_cg\), not a
vector space
\cite[Definition~3.1 and the following remarks]{FangGuoHouZhu2025Studia}.
If \(F(z)=p_c(z)\sum_{n\ge0}a_nz^n\in\mathcal F\) and \(t\in\C\), define
\[
\Ihat_tF(z)=
\begin{cases}
\displaystyle
p_{c+t}(z)\sum_{n=0}^{\infty}
\frac{\Gamma(n+1+c)}{\Gamma(n+1+c+t)}a_nz^n,
& c+t\notin\Z_-,\\[4mm]
\displaystyle
\sum_{n=-(c+t)}^{\infty}
\frac{\Gamma(n+1+c)}{\Gamma(n+1+c+t)}a_nz^{n+c+t},
& c+t\in\Z_-.
\end{cases}
\]
The resulting map \(\Ihat_t:\mathcal F\to\mathcal F\) is well defined; see
\cite[Definitions~3.1--3.2 and Lemma~3.3]{FangGuoHouZhu2025Studia}.  The second line is the
finite-kernel clause that applies when the shifted power is a negative
integer.

The admissible parameter set for the induced Volterra operators is
\[
        \calP=(((\C\setminus\Z)\cup\{0\})\times\C)
              \cup(\N\times\Z).
\]
For \((\alpha,\beta)\in\calP\) with \(\alpha,\beta>0\), and for
\(f,\varphi\in H(\D)\), define
\[
        V^\varphi_{\alpha,\beta}f
        =
        \Ihat_{\alpha}\left(
        \Ihat_{\beta-\alpha}f\,
        \Ihat_{-\beta}\varphi
        \right).
\]
Admissibility ensures that the middle product belongs to \(\mathcal F\) and
that the final output belongs to \(H(\D)\); see
\cite[Lemmas~3.4--3.5]{FangGuoHouZhu2025Studia}.  We state the symbol
conditions in terms of the coefficient operator
\[
        \IRL_t f(z)=\sum_{n=0}^{\infty}
        \frac{\Gamma(n+1)}{\Gamma(n+1+t)}\widehat f(n)z^n,
        \qquad t\in\C,
\]
where \(1/\Gamma(-m)=0\) for \(m=0,1,2,\ldots\), as justified by
Lemma~\ref{lem:background-tools}\textup{(v)}, and we set
\[
        G_\varphi=\IRL_{-\beta}\varphi.
\]
Proposition~\ref{prop:model-reduction} connects the two realizations by bounded
Taylor coefficient multipliers and finite-rank polynomial corrections;
boundedness and compactness are therefore preserved.

The fractional shift leads to the radial weight
\[
        \nu_{\alpha,\beta,p}(z)
        =\omega(z)(1-|z|^2)^{p(\alpha-\beta)}.
\]
We write \((H)\) for the shifted-weight hypothesis
\[
        \nu_{\alpha,\beta,p}\in\calD,
        \qquad
        \widehat\nu_{\alpha,\beta,p}(z)
        \asymp
        \widehat\omega(z)(1-|z|^2)^{p(\alpha-\beta)}.
        \tag{H}\label{hyp:H}
\]
The hypothesis \((H)\) is automatic when \(\alpha\ge\beta\), because the
shift then has a nonnegative power.  When \(\alpha<\beta\), the shifted
density may fail to be integrable.  Proposition~\ref{prop:negative-obstruction}
shows more precisely that, with \(\tau=p(\beta-\alpha)>0\),
\[
        (H)\quad\Longleftrightarrow\quad
        \tau<d_-(\omega).
\]
In particular, the two clauses in \((H)\) are not independent: once the
negative shift belongs to \(\calD\), the tail equivalence follows
automatically.
Theorems~\ref{thm:boundedness-main} and~\ref{thm:compactness-main}
classify boundedness and compactness for \(\alpha\ge\beta>0\), whereas
Theorem~\ref{thm:conditional-extension} gives the corresponding result for
\(\alpha<\beta\) below this critical threshold.

The superquadratic diagonal range \(2<p=q\) is handled by Carleson lifting.
The two ranges \(q<p\), separated by the threshold \(p=2\),
require lattice randomization and tent-space factorization.  These mechanisms
also account for the different compactness conditions.
\begin{theorem}\label{thm:boundedness-main}
Let \(0<p,q<\infty\), \(\omega\in\calD\), \(\varphi\in H(\D)\),
\((\alpha,\beta)\in\calP\), and \(\alpha\ge\beta>0\).  Then the following
statements hold.

\begin{enumerate}[label=\textup{(\Alph*)},leftmargin=2.5em]
\item If \(0<p\le \min\{2,q\}\) or \(2<p<q<\infty\), then
\(V^\varphi_{\alpha,\beta}:A^p_\omega\to H^q\) is bounded if and only if
\[
        \sup_{z\in\D}
        |G_\varphi(z)|
        \widehat\omega(z)^{-1/p}
        (1-|z|^2)^{\beta+1/q-1/p}<\infty.
        \tag{A}
\]

\item If \(2<p=q<\infty\), then
\(V^\varphi_{\alpha,\beta}:A^p_\omega\to H^p\) is bounded if and only if
\[
        d\mu_{\varphi,B}(z)=
        |G_\varphi(z)|^{\frac{2p}{p-2}}
        \frac{(1-|z|^2)^{\frac{p+2+2p(\beta-1)}{p-2}}}
             {\widehat\omega(z)^{\frac{2}{p-2}}}\,dA(z)
        \tag{B}
\]
is a Carleson measure.

\item If \(0<q<\infty\) and \(p>\max\{2,q\}\), then
\(V^\varphi_{\alpha,\beta}:A^p_\omega\to H^q\) is bounded if and only if
\[
        \xi\mapsto
        \left(
        \int_{\Gamma(\xi)}
        |G_\varphi(z)|^{\frac{2p}{p-2}}
        \frac{(1-|z|^2)^{\frac{4+2p(\beta-1)}{p-2}}}
             {\widehat\omega(z)^{\frac{2}{p-2}}}
        \dA(z)
        \right)^{\frac{p-2}{2p}}
        \in L^{\frac{pq}{p-q}}(\T).
        \tag{C}
\]

\item If \(0<q<p\le2\), then
\(V^\varphi_{\alpha,\beta}:A^p_\omega\to H^q\) is bounded if and only if
\[
        \xi\mapsto
        \sup_{z\in\Gamma(\xi)}
        |G_\varphi(z)|
        \frac{(1-|z|^2)^\beta}{\widehat\omega(z)^{1/p}}
        \in L^{\frac{pq}{p-q}}(\T).
        \tag{D}
\]
\end{enumerate}
\end{theorem}

\begin{theorem}\label{thm:compactness-main}
Under the assumptions of Theorem~\ref{thm:boundedness-main}, the following
compactness criteria hold.

\begin{enumerate}[label=\textup{(\Alph*)},leftmargin=2.5em]
\item In the range of Theorem~\ref{thm:boundedness-main}\textup{(A)},
\(V^\varphi_{\alpha,\beta}:A^p_\omega\to H^q\) is compact if and only if
\[
        \lim_{|z|\to1^-}
        |G_\varphi(z)|
        \widehat\omega(z)^{-1/p}
        (1-|z|^2)^{\beta+1/q-1/p}=0.
        \tag{A$_0$}
\]

\item In the range \(2<p=q<\infty\), compactness is equivalent to the measure
\(\mu_{\varphi,B}\) in Theorem~\ref{thm:boundedness-main}\textup{(B)} being a
vanishing Carleson measure.

\item In the range \(0<q<\infty\) and \(p>\max\{2,q\}\), compactness is
equivalent to boundedness, and hence to the condition in
Theorem~\ref{thm:boundedness-main}\textup{(C)}.

\item If \(0<q<p\le2\), compactness is equivalent to
\[
        \lim_{r\to1^-}
        \int_{\T}
        \sup_{z\in\Gamma(\xi)\setminus r\D}
        |G_\varphi(z)|^{\frac{pq}{p-q}}
        \frac{(1-|z|^2)^{\frac{pq\beta}{p-q}}}
             {\widehat\omega(z)^{\frac{q}{p-q}}}
        \dxi=0.
        \tag{D$_0$}
\]
\end{enumerate}
\end{theorem}

\begin{theorem}\label{thm:conditional-extension}
Let \(0<p,q<\infty\), \(\omega\in\calD\), \(\varphi\in H(\D)\),
\((\alpha,\beta)\in\calP\), and \(0<\alpha<\beta\).  Suppose that
\[
        p(\beta-\alpha)<d_-(\omega).
\]
Then the boundedness and compactness conclusions of
Theorems~\ref{thm:boundedness-main} and
\ref{thm:compactness-main} remain valid, with
\(G_\varphi=\IRL_{-\beta}\varphi\).
\end{theorem}

\begin{remark}\label{rem:exponent-bookkeeping}
Under \(u_\varphi=G_\varphi\delta^{\alpha-1}\) and
\(\widehat\nu\asymp\widehat\omega\delta^{p(\alpha-\beta)}\), the distinct
Carleson and tent densities in Cases B and C give, respectively, the powers
\[
 \frac{p+2+2p(\beta-1)}{p-2}
 \quad\text{and}\quad
 \frac{4+2p(\beta-1)}{p-2},
\]
with denominator \(\widehat\omega^{2/(p-2)}\) in both cases.
\end{remark}

\section{Fractional shifts and area-function reductions}

This section identifies the sharp threshold for negative weight shifts, develops
the Littlewood--Paley reduction, passes to the modified model and area map, and
records the resulting weight dictionary.
The regularization and local geometry of \(\calD\)-weights used below are
recorded in \cite[Lemmas~2.5--2.7]{DuanWangWang2021}.  The associated atomic
and Carleson measure theory appears in \cite{PelaezRattyaSierra2021}, with
related operator results for regular weights in
\cite{PelaezRattyaSierra2018}.

\subsection{Analytic and geometric preliminaries}

For \(z,w\in\D\), let
\[
        \varrho(z,w)=\left|\frac{z-w}{1-\overline zw}\right|
\]
be the pseudo-hyperbolic distance, and write
\(D(z,r)=\{w\in\D:\varrho(z,w)<r\}\).  For \(0<\eta<1\), an
\(\eta\)-lattice is a sequence \(Z=\{a_k\}\subset\D\) such that the discs
\(D(a_k,\eta)\) cover \(\D\) and the discs \(D(a_k,\eta/5)\) are pairwise
disjoint.  Every fixed pseudo-hyperbolic enlargement of the latter discs has
uniformly bounded overlap; see Lemma~\ref{lem:background-tools}\textup{(vi)}.
For an arc \(I\subset\T\), set
\[
        \wedge(I)=\D\setminus
        \bigcup_{\xi\in\T\setminus I}\Gamma(\xi).
\]
All continuous and discrete tent quasi-norms used below are equivalent under a
fixed change of aperture.  Moreover, \(S(I)\) and \(\wedge(I)\) are
interchangeable up to fixed enlargements of \(I\); see again
Lemma~\ref{lem:background-tools}\textup{(vi)}.

\begin{lemma}\label{lem:background-tools}
The following statements hold.
\begin{enumerate}[label=\textup{(\roman*)}]
\item On any measure space, generalized H\"older's inequality holds for a finite
family of exponents
\(0<r,r_j\le\infty\) satisfying \(\sum_j1/r_j=1/r\):
\[
        \left\|\prod_j f_j\right\|_{L^r}
        \le \prod_j\|f_j\|_{L^{r_j}}.
\]
Minkowski's inequality holds
for \(1\le r<\infty\), while
\[
        \left|\sum_j x_j\right|^r\le\sum_j|x_j|^r,
        \qquad 0<r\le1.
\]
For \(1<s<\infty\), the \(\ell^s\)--\(\ell^{s'}\) duality formula is valid.

\item Tonelli's theorem may be applied to nonnegative measurable functions, and
Fubini's theorem to complex-valued functions once absolute integrability has been
established.  Fatou's lemma applies to nonnegative measurable functions, while
monotone convergence applies to pointwise nondecreasing sequences of nonnegative
measurable functions.  Dominated convergence applies
to almost-everywhere convergent functions dominated in modulus by one integrable
function.

\item If \(F\) is analytic, then Cauchy's coefficient estimates hold; locally bounded
families of analytic functions are normal; and \(|F|^s\) is subharmonic for every
\(s>0\).  In particular,
\[
        |F(z)|\lesssim_p(1-|z|)^{-1/p}\|F\|_{H^p},
        \qquad 0<p<\infty.
        \tag{2.1a}
\]

\item If \(0<p<\infty\) and \(v\) is an integrable radial weight with
\(\widehat v(r)>0\) for \(r<1\), then point evaluations and Taylor coefficient
functionals are continuous on \(A^p_v\), bounded subsets of \(A^p_v\) are normal,
polynomials are dense in \(A^p_v\), and \(A^p_v\) is complete for its norm or
quasi-norm.

\item The reciprocal gamma function is entire and vanishes at the nonpositive
integers.  For fixed \(a,b\in\C\) and all sufficiently large integers \(n\),
\[
        \frac{\Gamma(n+a)}{\Gamma(n+b)}
        \sim n^{a-b}\sum_{k=0}^{\infty}c_k(a,b)n^{-k}
        \qquad(n\to\infty),
        \tag{2.1b}
\]
in the Poincar\'e sense.  Finite products of such quotients therefore have the
corresponding full inverse-power expansion; if their leading constant is nonzero,
their reciprocals do as well.

\item If \(v\in\calD\) and \(\kappa_v=\widehat v/\delta\), set
\(\kappa_v(E)=\int_E\kappa_v\,\dA\).  Then \(\kappa_v\) is
regular, \(\widehat\kappa_v\asymp\widehat v\), and
\[
        \int_\D |f|^p\kappa_v\dA\asymp\|f\|_{A^p_v}^p,
        \qquad
        \kappa_v(D(z,r))\asymp\widehat v(z)\delta(z)
        \tag{2.1d}
\]
for every fixed \(0<r<1\), and \(\delta\), \(\widehat v\), and \(\kappa_v\)
are locally comparable on such discs.  Maximal separated sets give
pseudo-hyperbolic lattices whose fixed enlargements have bounded overlap.
Boundary shadows have length comparable to \(\delta(z)\), and fixed changes of
cone aperture or tent enlargement change the associated
continuous and discrete tent quasi-norms only by constants.
\end{enumerate}
\end{lemma}

\begin{proof}
For \textup{(i)}, when \(r<\infty\), apply ordinary H\"older to
\(\prod_j|f_j|^r\) with exponents \(r_j/r\ge1\); an infinite \(r_j\) is
handled by the essential supremum, and \(r=\infty\) is immediate.  Ordinary
H\"older applied to
\[
 \int |f+g|^r\le
 \int |f|\,|f+g|^{r-1}+\int |g|\,|f+g|^{r-1}
\]
proves Minkowski for \(r>1\); \(r=1\) is the integral triangle inequality, and
the power inequality for \(0<r\le1\) follows from concavity of \(t^r\).
Sequence duality follows from H\"older and, for finitely supported nonzero
\(c\), the choice
\(d_k=\overline{c_k}|c_k|^{s-2}/\|c\|_{\ell^s}^{s-1}\), with \(d_k=0\)
when \(c_k=0\); truncation gives the general case.

Part \textup{(ii)} follows from the corresponding Tonelli, Fubini, Fatou,
monotone-convergence, and dominated-convergence theorems in
\cite[Theorems~1.7.18 and~1.7.21, Corollary~1.7.23,
Corollary~1.4.47, and Theorems~1.4.44 and~1.4.49]{TaoMeasureTheory}.
All applications below are to finite or \(\sigma\)-finite Borel spaces,
completed when necessary.

Cauchy's estimates, Montel's theorem, the Poisson representation, and the
maximum principle used in \textup{(iii)} are standard; see, for example,
\cite{ConwayComplexAnalysis}.  Off the zeros of
\(F\), direct differentiation gives
\(\Delta|F|^s=s^2|F|^{s-2}|F'|^2\ge0\); at a zero the sub-mean inequality is
automatic, so continuity makes \(|F|^s\) subharmonic.  For \(|z|<R<1\),
Poisson comparison, with \(P_w(\xi)=(1-|w|^2)/|\xi-w|^2\), gives
\[
 |F(z)|^p\le\int_\T P_{z/R}(\xi)|F(R\xi)|^p\dxi
 \le\frac{R+|z|}{R-|z|}\int_\T|F(R\xi)|^p\dxi .
\]
Letting \(R\uparrow1\) proves \((2.1a)\).

For \textup{(iv)}, fix a compact disc \(K\Subset\D\), choose
\(R>\sup_K|z|\), apply the preceding estimate on each circle of radius
\(t\ge R\), and integrate against \(2tv(t)\,dt\).  Since this measure has
positive mass on \([R,1)\),
\[
 \sup_K|f|^p\int_R^1 2tv(t)\,dt
 \le C_{K,R}\int_R^1M_p(t,f)^p2tv(t)\,dt
 \le C_{K,R}\|f\|_{A^p_v}^p. \tag{2.1c}
\]
Thus point evaluations are continuous, Cauchy's estimates give continuity of the
coefficient functionals, and Montel gives normality.  For
\(f_r(z)=f(rz)\), subharmonicity makes the integral means of \(|f|^p\)
nondecreasing.  Writing \(M_p(t,g)^p=\int_\T|g(t\xi)|^p\dxi\),
\[
 M_p(t,f_r-f)^p\le
 C_p\bigl(M_p(rt,f)^p+M_p(t,f)^p\bigr)\le2C_pM_p(t,f)^p.
\]
Dominated convergence gives \(\|f_r-f\|_{A^p_v}\to0\), and Taylor
polynomials approximate each fixed \(f_r\) uniformly on \(\overline\D\).
Thus polynomials are dense.  If \(f_n\) is Cauchy, \((2.1c)\) gives local
uniform convergence to some \(f\in H(\D)\), while Fatou gives
\[
 \|f_n-f\|_{A^p_v}^p\le
 \liminf_{m\to\infty}\|f_n-f_m\|_{A^p_v}^p,
\]
which proves completeness.

The first assertion in \textup{(v)} is given in the NIST DLMF
\cite[Sec.~5.2(i)]{NISTDLMF}; the quotient expansion is
\cite[Sec.~5.11(iii), Eq.~(5.11.13)]{NISTDLMF}, with \(c_0(a,b)=1\).
Multiplication gives the product expansion; for the gamma products below the
powers of \(n\) cancel.  Thus, with \(y=(n+1)^{-1}\) and
\(n^{-1}=y/(1-y)\), finite re-expansion gives
\[
 P_n=\sum_{j=0}^{N}d_jn^{-j}+O(n^{-N-1})
 =\sum_{j=0}^{N}\widetilde d_j(n+1)^{-j}
  +O((n+1)^{-N-1}).
\]
Thus \(P_n\) has Taylor form.  If \(\widetilde d_0\ne0\), recursive formal
inversion, with
\[
 e_0=\widetilde d_0^{-1},\qquad
 e_k=-\widetilde d_0^{-1}\sum_{j=1}^{k}\widetilde d_j e_{k-j},
\]
gives the same expansion for \(1/P_n\).

For \textup{(vi)}, regularization, norm equivalence, and disc comparability
are \cite[Lemmas~2.5--2.7]{DuanWangWang2021}, with
\(1-|z|\asymp\delta(z)\).  Local comparability and
\(A(D(z,r))\asymp\delta(z)^2\) give
\(\kappa_v(D(z,r))\asymp\widehat v(z)\delta(z)\).  A maximal
\(\eta/5\)-separated family covers \(\D\), and area comparison gives bounded
overlap of fixed enlargements.  The shadow and tent assertions follow from
\cite[formula~(2.6), the discussion preceding Lemma~2.11, Lemma~2.11, and the
following paragraph]{DuanWangWang2021}; aperture invariance for all exponents
used here is \cite[Section~2]{DuanWangWang2021}.  Applying it to
\(\sum_k\delta_{a_k}\) gives the discrete statement, including
\(\ell^\infty\).
\end{proof}

\begin{lemma}\label{lem:analytic-compactness-criterion}
Let \(0<p,q<\infty\), let \(v\) be an integrable radial weight with
\(\widehat v(r)>0\) for \(r<1\), and let
\(T:H(\D)\to H(\D)\) be linear and continuous for the compact-open topology.
If the restriction \(T:A^p_v\to H^q\) is bounded, then it is compact if and
only if
\[
        \|Tf_j\|_{H^q}\longrightarrow0
\]
for every bounded sequence \(\{f_j\}\subset A^p_v\) that converges to zero
locally uniformly on \(\D\).
\end{lemma}

\begin{proof}
Suppose first that the displayed sequential condition holds.  Every bounded
sequence in \(A^p_v\) has, by Lemma~\ref{lem:background-tools}\textup{(iv)}, a
subsequence \(\{f_{j_k}\}\) converging locally uniformly to some
\(f\in H(\D)\).  Fatou's lemma gives \(f\in A^p_v\), while
\(\{f_{j_k}-f\}\) is bounded in \(A^p_v\) and converges to zero locally
uniformly.  The sequential condition therefore implies
\(Tf_{j_k}\to Tf\) in \(H^q\), which proves relative compactness.

Conversely, suppose that \(T\) is compact and that \(f_j\to0\) locally
uniformly, with \(\{f_j\}\) bounded in \(A^p_v\).  If
\(\|Tf_j\|_{H^q}\not\to0\), then some \(\varepsilon>0\) and a subsequence
satisfy \(\|Tf_{j_k}\|_{H^q}\ge\varepsilon\).  Compactness provides a
further norm-convergent subsequence.  Its limit must be zero because norm
convergence in \(H^q\) implies local uniform convergence, whereas compact-open
continuity gives \(Tf_{j_k}\to0\) locally uniformly.  This contradiction
proves the claim.  The same argument applies in the quasi-Banach ranges with
the standard power metrics.
\end{proof}

\subsection{Nonnegative shifts of \texorpdfstring{\(\calD\)}{D}-weights}

For \(\sigma\in\mathbb R\), set
\[
        \omega_\sigma(r)=\omega(r)(1-r^2)^\sigma
\]
and regard \(\omega_\sigma\) as a radial weight whenever it is integrable on
\([0,1)\).  Integrability is automatic when \(\sigma\ge0\).

\begin{lemma}\label{lem:positive-shift-D}
Let \(\omega\in\calD\) and \(\sigma\ge0\).  Then
\[
        \omega_\sigma\in\calD,
        \qquad
        \widehat{\omega_\sigma}(r)
        \asymp
        (1-r^2)^\sigma\widehat\omega(r),
        \qquad 0\le r<1.
\]
\end{lemma}

\begin{proof}
Since \(1-r^2\asymp1-r\), it is enough to use \(1-r\).  Since
\((1-s^2)^\sigma\lesssim(1-r)^\sigma\) for \(r\le s<1\),
\[
 \widehat{\omega_\sigma}(r)
 =\int_r^1\omega(s)(1-s^2)^\sigma ds
 \lesssim(1-r)^\sigma\widehat\omega(r).
\]
Reverse doubling gives \(K>1\), \(0<\theta<1\), and
\(r_K=1-(1-r)/K\) such that
\(\widehat\omega(r_K)\le\theta\widehat\omega(r)\).  Since
\(1-s\asymp1-r\) for \(r\le s\le r_K\),
\[
\begin{aligned}
 \widehat{\omega_\sigma}(r)
 &\ge\int_r^{r_K}\omega(s)(1-s^2)^\sigma ds\\
 &\gtrsim(1-r)^\sigma
 \bigl(\widehat\omega(r)-\widehat\omega(r_K)\bigr)
 \gtrsim(1-r)^\sigma\widehat\omega(r).
\end{aligned}
\]
Thus the asserted tail equivalence holds.  With \(r_*=(1+r)/2\), it and
forward doubling for \(\omega\) give
\[
 \widehat{\omega_\sigma}(r)\lesssim
 (1-r)^\sigma\widehat\omega(r)\lesssim
 (1-r_*)^\sigma\widehat\omega(r_*)\lesssim
 \widehat{\omega_\sigma}(r_*),
\]
which proves forward doubling.  For reverse doubling, the same comparison at
\(r\) and \(r_K\) yields
\[
\widehat{\omega_\sigma}(r_K)\lesssim
K^{-\sigma}(1-r)^\sigma\widehat\omega(r_K)
\le C\theta K^{-\sigma}\widehat{\omega_\sigma}(r).
\]
Iteration of the base-weight estimate gives
\(\widehat\omega(r_{K^m})\le\theta^m\widehat\omega(r)\); applying the tail
comparison at \(r\) and \(r_{K^m}\) therefore yields
\[
\widehat{\omega_\sigma}(r_{K^m})
\le C(\theta K^{-\sigma})^m\widehat{\omega_\sigma}(r).
\]
Choose \(m\) so that the coefficient is less than one, and relabel \(K^m\).
This proves reverse doubling, hence \(\omega_\sigma\in\calD\).
\end{proof}

\begin{corollary}\label{cor:automatic-shift}
If \(0<p<\infty\), \(\omega\in\calD\), and \(\alpha\ge\beta\), then
\[
 \nu_{\alpha,\beta,p}=\omega\delta^{p(\alpha-\beta)}\in\calD,
 \qquad
 \widehat\nu_{\alpha,\beta,p}(z)\asymp
 \widehat\omega(z)\delta(z)^{p(\alpha-\beta)}.
\]
\end{corollary}

\begin{proof}
Apply Lemma~\ref{lem:positive-shift-D} with
\(\sigma=p(\alpha-\beta)\ge0\).
\end{proof}

\subsection{The critical negative-shift exponent}\label{subsec:negative-shift}

Stability under sufficiently small negative boundary powers, together with
the corresponding tail comparison, follows from
\cite[Lemma~2]{PelaezRattya2023BMO} after taking the auxiliary weight there
to be \(\nu\equiv1\), so that \(\widehat\nu(r)=1-r\), and using
\(1-r^2\asymp1-r\).  The next proposition identifies the full interval
relevant to the shifted weight in this paper.

\begin{proposition}\label{prop:negative-obstruction}
Let \(\omega\in\calD\) and \(\tau>0\), and put
\[
        \nu_\tau(r)=\omega(r)(1-r^2)^{-\tau}.
\]
The tail \(\widehat{\nu_\tau}\) is initially interpreted as an extended
nonnegative integral.  Then the following conditions are equivalent:
\begin{enumerate}[label=\textup{(\roman*)},leftmargin=2.5em]
\item \(\nu_\tau\in\calD\);
\item
\[
        \widehat{\nu_\tau}(r)
        \asymp
        (1-r^2)^{-\tau}\widehat\omega(r),
        \qquad 0\le r<1;
\]
\item \(\tau<d_-(\omega)\).
\end{enumerate}
Consequently, if \(\alpha<\beta\), then the shifted-weight hypothesis
\textup{(H)} is equivalent to
\[
        p(\beta-\alpha)<d_-(\omega).
\]
The strict inequality is sharp: for every \(\tau>0\), there is an
\(\omega\in\calD\) with \(d_-(\omega)=\tau\) such that \(\nu_\tau\) is
integrable but does not belong to \(\calD\).
\end{proposition}

\begin{proof}
Write
\[
        W(r)=\widehat\omega(r),
        \qquad
        N(r)=\widehat{\nu_\tau}(r).
\]
If \(\nu_\tau\in\calD\), then
\(\omega=\nu_\tau(1-|\cdot|^2)^\tau\), so
Lemma~\ref{lem:positive-shift-D}, applied with base weight \(\nu_\tau\),
gives
\[
        W(r)\asymp(1-r^2)^\tau N(r).
\]
Thus \textup{(i)} implies \textup{(ii)}.

Conversely, assume \textup{(ii)}.  Then \(N(r)<\infty\), so \(\nu_\tau\)
is a radial weight.  If \(r_*=(1+r)/2\), forward doubling for \(\omega\)
and \textup{(ii)} yield
\[
\begin{aligned}
        N(r)
        &\lesssim (1-r^2)^{-\tau}W(r)
        \lesssim (1-r^2)^{-\tau}W(r_*) \\
        &\lesssim
        \left(\frac{1-r_*^2}{1-r^2}\right)^\tau N(r_*)
        \lesssim N(r_*).
\end{aligned}
\]
Hence \(\nu_\tau\in\widehat{\calD}\).

For reverse doubling, choose \(c\in(0,1)\) such that
\[
        W(r)\ge c(1-r^2)^\tau N(r),
        \qquad 0\le r<1.
\]
Let \(r_K=1-(1-r)/K\), where \(K>1\) is so large that
\((2/K)^\tau<c\).  Since
\[
        \frac{1-r_K^2}{1-r^2}
        =\frac1K\frac{1+r_K}{1+r}\le\frac2K,
\]
splitting the integral defining \(W(r)\) at \(r_K\) gives
\[
\begin{aligned}
        W(r)
        &\le
        (1-r^2)^\tau\bigl(N(r)-N(r_K)\bigr)
        +(1-r_K^2)^\tau N(r_K).
\end{aligned}
\]
It follows that
\[
        N(r_K)\le
        \frac{1-c}{1-(2/K)^\tau}N(r)=qN(r),
        \qquad q<1.
\]
Thus \(\nu_\tau\in\check{\calD}\), proving
\textup{(ii)} implies \textup{(i)}.

Suppose that \(\tau<d_-(\omega)\).  There is an admissible exponent
\(b>\tau\), so
\[
        W(s)\lesssim
        \left(\frac{1-s}{1-r}\right)^bW(r),
        \qquad 0\le r\le s<1.
\]
For fixed \(r\), set \(r_j=1-2^{-j}(1-r)\).  Since
\[
        1-r_{j+1}^2\ge2^{-(j+1)}(1-r^2),
        \qquad
        W(r_j)\lesssim2^{-jb}W(r),
\]
we obtain
\[
\begin{aligned}
        N(r)
        &\le
        \sum_{j=0}^{\infty}
        (1-r_{j+1}^2)^{-\tau}W(r_j) \\
        &\lesssim
        (1-r^2)^{-\tau}W(r)
        \sum_{j=0}^{\infty}2^{-j(b-\tau)}
        \lesssim
        (1-r^2)^{-\tau}W(r).
\end{aligned}
\]
The reverse inequality follows from
\((1-s^2)^{-\tau}\ge(1-r^2)^{-\tau}\) for \(s\ge r\).
This proves \textup{(iii)} implies \textup{(ii)}.

Finally, assume \textup{(i)}.  By
\cite[Lemma~B]{PelaezDeLaRosa2022}, there is an \(\varepsilon>0\) such
that
\[
        N(s)\lesssim
        \left(\frac{1-s}{1-r}\right)^\varepsilon N(r),
        \qquad 0\le r\le s<1.
\]
The already proved tail comparison and
\[
        \frac{1-s^2}{1-r^2}
        \le2\frac{1-s}{1-r}
\]
therefore give
\[
        W(s)\lesssim
        \left(\frac{1-s}{1-r}\right)^{\tau+\varepsilon}W(r).
\]
Hence \(d_-(\omega)\ge\tau+\varepsilon>\tau\), proving
\textup{(i)} implies \textup{(iii)} and completing the equivalence.
Taking \(\tau=p(\beta-\alpha)\) gives the asserted characterization of
\textup{(H)}.

It remains to verify sharpness.  Fix \(\tau>0\), set
\[
        L(r)=\log\frac e{1-r^2},
        \qquad
        \omega(r)=(1-r^2)^{\tau-1}L(r)^{-2}.
\]
Elementary one-dimensional tail estimates give
\[
        \widehat\omega(r)
        \asymp(1-r^2)^\tau L(r)^{-2}.
\]
Thus \(\omega\in\calD\).  Since \(L(s)\ge L(r)\) and
\(1-t^2\asymp1-t\), one has
\(\widehat\omega(s)/\widehat\omega(r)\lesssim((1-s)/(1-r))^\tau\), so
\(b=\tau\) is admissible.  If \(b>\tau\), setting \(r=0\) gives
\(\widehat\omega(s)/[(1-s)^b\widehat\omega(0)]
\asymp(1-s)^{\tau-b}L(s)^{-2}\to\infty\).  Hence
\(d_-(\omega)=\tau\).  The critical shift
\[
        \nu_\tau(r)=(1-r^2)^{-1}L(r)^{-2}
\]
is integrable, but
\[
        \widehat{\nu_\tau}(r)\asymp L(r)^{-1}
        \not\asymp
        (1-r^2)^{-\tau}\widehat\omega(r)\asymp L(r)^{-2}.
\]
Moreover, for each fixed \(K>1\),
\[
        \frac{
        \widehat{\nu_\tau}\!\left(1-\frac{1-r}{K}\right)}
        {\widehat{\nu_\tau}(r)}
        \longrightarrow1,
        \qquad r\to1^-,
\]
so \(\nu_\tau\in\widehat{\calD}\setminus\check{\calD}\).
This proves the endpoint assertion.  Failure can be more drastic: the power
weight \(\omega_0(r)=(1-r^2)^{\tau/2-1}\) belongs to \(\calD\), whereas
\(\omega_0(r)(1-r^2)^{-\tau}=(1-r^2)^{-1-\tau/2}\) is not integrable.
\end{proof}

Thus \((H)\) is automatic when \(\alpha\ge\beta\), by
Corollary~\ref{cor:automatic-shift}; when \(\alpha<\beta\), it holds exactly
below the critical threshold in Proposition~\ref{prop:negative-obstruction}.

\subsection{The fractional Littlewood--Paley shift}

We work with the modified Riemann--Liouville coefficient operator
\[
        \IRL_t f(z)=\sum_{n=0}^{\infty}
        \frac{\Gamma(n+1)}{\Gamma(n+1+t)}\widehat f(n)z^n.
\]
Here \(1/\Gamma(-m)=0\) for \(m=0,1,2,\ldots\), as in
Lemma~\ref{lem:background-tools}\textup{(v)}.

\begin{lemma}\label{lem:finite-block-multipliers}
Let \(0<p<\infty\).  There is a number \(A_p>0\) with the following property.
If \(I\subset\{0,1,\ldots,M\}\) is a finite interval of integers, \(M\ge2\), and
\(\{b_n\}_{n\in I}\) is bounded, then
\[
        T_{I,b}\left(\sum_{n\in I}a_nz^n\right)
        =
        \sum_{n\in I}b_na_nz^n
\]
satisfies
\[
        \|T_{I,b}P\|_{H^p}
        \le
        C_p M^{A_p}\sup_{n\in I}|b_n|\,\|P\|_{H^p}
\]
for every polynomial \(P\) supported in \(I\).
\end{lemma}

\begin{proof}
Take \(A_p=1+1/p\) for \(p\ge1\) and \(A_p=2/p\) for \(0<p<1\).
Lemma~\ref{lem:background-tools}\textup{(iii)}, with
\(r=1-(M+1)^{-1}\), and Cauchy's estimate give
\[
 |P(re^{it})|\le C_p(1-r)^{-1/p}\|P\|_{H^p},
 \qquad |a_n|\le C_pM^{1/p}\|P\|_{H^p}\quad(0\le n\le M),
\]
because \(r^{-n}\le r^{-M}\le e\).  Minkowski's inequality when \(p\ge1\)
and the power inequality when \(0<p<1\), respectively, yield
\[
\begin{aligned}
 \|T_{I,b}P\|_{H^p}
 &\le\sum_{n\in I}|b_na_n|
 \le C_pM^{1+1/p}\sup_{n\in I}|b_n|\,\|P\|_{H^p},\\
 \|T_{I,b}P\|_{H^p}^p
 &\le\sum_{n\in I}|b_na_n|^p
 \le C_pM^2\sup_{n\in I}|b_n|^p\,\|P\|_{H^p}^p.
\end{aligned}
\]
\end{proof}

For \(P(z)=\sum p_nz^n\) and \(f(z)=\sum\widehat f(n)z^n\), write
\((P*f)(z)=\sum p_n\widehat f(n)z^n\).  Below \(P\) is a polynomial or a
finitely supported Ces\`aro block, so this Hadamard product is unambiguous.

\begin{lemma}\label{lem:D-Bergman-Cesaro-block-norm}
Let \(0<p<\infty\) and \(\rho\in\calD\).  There exist an integer \(K>1\) and a
universal Ces\`aro block system \(\{V_{j,K}\}_{j\ge0}\) such that, for every
\(f\in H(\D)\),
\[
        f=
        \sum_{j=0}^{\infty}V_{j,K}*f
\]
locally uniformly in \(\D\), each block \(V_{j,K}\) is supported on a fixed
multiplicative annulus of Taylor indices, and
\[
        \|f\|_{A^p_\rho}^p
        \asymp
        \sum_{j=0}^{\infty}
        \rho_{K^j}\,
        \|V_{j,K}*f\|_{H^p}^p,
        \qquad
        \rho_x=
        \int_0^1 r^x\rho(r)\,dr .
\]
The constants depend only on \(p\), \(\rho\), \(K\), and the chosen Ces\`aro
block system.
\end{lemma}

\begin{proof}
This is \cite[Proposition~9]{PelaezDeLaRosa2022}.  The decomposition, support
property, and uniform \(H^p\)-boundedness of the same universal Ces\`aro block
system are
given in \cite[Proposition~4 and Theorem~D]{PelaezDeLaRosa2022}.
\end{proof}

\begin{lemma}\label{lem:D-bergman-taylor-multipliers}
Let \(0<p<\infty\), \(\rho\in\calD\), and let
\[
        M_\lambda f(z)=\sum_{n=0}^{\infty}\lambda_n\widehat f(n)z^n .
\]
Assume that \(\lambda=\{\lambda_n\}\) has Taylor form: for every \(N\ge0\)
there are constants \(a_0,\ldots,a_N\) and \(C_N\) such that
\[
        \left|
        \lambda_n-\sum_{j=0}^{N}a_j(n+1)^{-j}
        \right|
        \le C_N(n+1)^{-N-1},
        \qquad n\ge1.
\]
Then \(M_\lambda\) is bounded on \(A^p_\rho\).  If, outside a finite set,
\(\lambda_n\ne0\) and \(1/\lambda_n\) also has Taylor form, then
\(M_\lambda\) is an isomorphism of \(A^p_\rho\) modulo a finite-dimensional
polynomial space.
\end{lemma}

\begin{proof}
Changing finitely many coefficients gives a bounded finite-rank operator, so
it suffices to consider the tail and, by the density and completeness in
Lemma~\ref{lem:background-tools}\textup{(iv)}, to prove the estimate for
polynomials.
Choose the integer \(K>1\) and the universal Ces\`aro block system
\(\{V_{j,K}\}_{j\ge0}\) from
Lemma~\ref{lem:D-Bergman-Cesaro-block-norm}.  Its locally uniform
decomposition, block quasi-norm, and finite overlap will be used below.  The
moment sequence also satisfies
\[
        \rho_{K^j}\asymp \rho_{K^m},
        \qquad |j-m|\le C_0,
\]
for every fixed \(C_0\).  This is the moment-comparability property for
\(\calD\)-weights; see \cite[Lemma~A(v)]{PelaezDeLaRosa2022}.

We first show that the model multipliers
\[
        S_\ell f(z)=\sum_{n=0}^{\infty}(n+1)^{-\ell}\widehat f(n)z^n,
        \qquad \ell=0,1,2,\ldots,
\]
are bounded on \(A^p_\rho\).  The case \(\ell=0\) is the identity.  For \(\ell\ge1\),
write, for \(j\ge1\),
\[
        V_{j,K}=W^{\psi}_{K^{j-1}}
\]
with \(\psi\in C_c^\infty((K^{-1},K^2))\).  The block multiplier
\(V_{j,K}*S_\ell\) has coefficients
\[
        \psi\!\left(\frac{n}{K^{j-1}}\right)(n+1)^{-\ell}.
\]
Put \(\Phi_{j,\ell}(x)=\psi(x)(K^{j-1}x+1)^{-\ell}\).
Choose an integer \(m\) such that \(mp>1\).  On the fixed compact support
of \(\psi\), the \(C^m\)-seminorms of \(\Phi_{j,\ell}\) are bounded
uniformly in \(j\), because differentiating
\((K^{j-1}x+1)^{-\ell}\) produces factors
\[
        K^{(j-1)k}(K^{j-1}x+1)^{-\ell-k},
        \qquad 0\le k\le m,
\]
and \(x\) stays away from zero.  Hence
\cite[Theorem~D\textup{(iii)}]{PelaezDeLaRosa2022} gives
\[
        \|W^{\Phi_{j,\ell}}_{K^{j-1}}*F\|_{H^p}
        \le C_{p,\ell}\|F\|_{H^p},
        \qquad j\ge1.
\]
The low blocks are bounded and finite rank by
Lemma~\ref{lem:background-tools}\textup{(iv)}.  Finite overlap therefore gives
\[
        \|V_{j,K}*(S_\ell f)\|_{H^p}
        \le C_{p,\ell}
        \sum_{|m-j|\le C_1}\|V_{m,K}*f\|_{H^p},
\]
where the finite overlap comes from the supports of the Ces\`aro blocks.  Multiplying
by \(\rho_{K^j}\), summing over \(j\), using the moment comparability above, and
applying Lemma~\ref{lem:background-tools}\textup{(i)} to the finite overlap yields
\[
        \|S_\ell f\|_{A^p_\rho}
        \lesssim
        \|f\|_{A^p_\rho}.
\]
Thus each \(S_\ell\), and hence every finite linear combination of them, is
bounded on \(A^p_\rho\).

It remains to treat the Taylor remainder.  Choose an integer \(L\) so large that
\(L+1>A_p+2\), where \(A_p\) is the exponent in
Lemma~\ref{lem:finite-block-multipliers}.  Write
\[
        \lambda_n=\sum_{\ell=0}^{L}a_\ell(n+1)^{-\ell}+\varepsilon_n,
        \qquad
        |\varepsilon_n|\le C_L(n+1)^{-L-1}.
\]
The Taylor polynomial part is bounded.  Since
\(f=\sum_m V_{m,K}*f\) and the supports of the blocks have finite overlap,
\[
        V_{j,K}*(M_\varepsilon f)
        =
        \sum_{|m-j|\le C_2}
        V_{j,K}*M_\varepsilon*(V_{m,K}*f).
\]
On the \(j\)-th block, \(n\asymp K^j\), so
\(\sup|\varepsilon_n|\lesssim K^{-j(L+1)}\), and every summand is supported
in \(\{0,\ldots,CK^j\}\).  Lemma~\ref{lem:finite-block-multipliers} gives
\[
        \|V_{j,K}*M_\varepsilon*(V_{m,K}*f)\|_{H^p}
        \lesssim
        K^{jA_p}K^{-j(L+1)}
        \|V_{m,K}*f\|_{H^p}
        \lesssim
        K^{-2j}\|V_{m,K}*f\|_{H^p}.
\]
Summing over the finitely many neighboring \(m\), the block quasi-norm,
moment comparability, and finite overlap give
\[
        \|M_\varepsilon f\|_{A^p_\rho}
        \lesssim
        \|f\|_{A^p_\rho}.
\]
Therefore \(M_\lambda\) is bounded on \(A^p_\rho\).

If \(1/\lambda_n\) has Taylor form outside a finite set, the same proof applies to the tail
multiplier \(M_{1/\lambda}\).  Altering finitely many coefficients only changes a finite-rank
polynomial operator.  Hence \(M_\lambda\) has a bounded inverse modulo a fixed
finite-dimensional polynomial space.
\end{proof}

\begin{theorem}\label{thm:fractional-LP-shift}
Let \(0<p<\infty\), \(s>0\), and \(\rho\in\calD\).  Set
\[
        \rho_s(z)=\rho(z)(1-|z|^2)^{sp},
\]
and assume \(\rho_s\in\calD\).  Then, for every \(f\in H(\D)\),
\[
        f\in A^p_\rho
        \quad\Longleftrightarrow\quad
        \IRL_{-s}f\in A^p_{\rho_s}.
        \tag{2.2}
\]
More precisely, if
\[
        E_s=\{n\ge0:n+1-s\in\{0,-1,-2,\ldots\}\},
\]
then
\[
        \|f\|_{A^p_\rho}^p
        \asymp
        \|\IRL_{-s}f\|_{A^p_{\rho_s}}^p
        +\sum_{n\in E_s}|\widehat f(n)|^p.
        \tag{2.3}
\]
\end{theorem}

\begin{proof}
Choose
\[
        \mu_s(r)=s(1-r^2)^{s-1}.
\]
For a radial weight \(\mu\) with positive moments
\(\mu_x=\int_0^1r^x\mu(r)\,dr\), write
\[
        D^\mu f(z)=
        \sum_{n=0}^{\infty}
        \frac{\widehat f(n)}{\mu_{2n+1}}z^n
\]
for the associated fractional derivative.  Then \(\mu_s\in\calD\) and
\[
        \widehat\mu_s(r)\asymp(1-r^2)^s.
\]
The fractional Littlewood--Paley theorem
\cite[Theorem~1]{PelaezDeLaRosa2022} gives
\[
        \|f\|_{A^p_\rho}^p
        \asymp
        \int_{\D}|D^{\mu_s}f(z)|^p\widehat\mu_s(z)^p\rho(z)\dA(z)
        \asymp
        \|D^{\mu_s}f\|_{A^p_{\rho_s}}^p.
        \tag{2.4}
\]
The moments of \(\mu_s\) are
\[
\begin{aligned}
        (\mu_s)_{2n+1}
        &=\int_0^1 r^{2n+1}s(1-r^2)^{s-1}\,dr  \\
        &=\frac{\Gamma(s+1)}{2}\frac{\Gamma(n+1)}{\Gamma(n+s+1)}.
\end{aligned}
\]
The last equality is Euler's beta integral
\cite[Sec.~5.12(i), Eq.~(5.12.1)]{NISTDLMF} after the substitution \(t=r^2\).
Thus
\[
        D^{\mu_s}f(z)=
        \frac{2}{\Gamma(s+1)}
        \sum_{n=0}^{\infty}
        \frac{\Gamma(n+s+1)}{\Gamma(n+1)}\widehat f(n)z^n.
\]
For \(n\notin E_s\), the ratio of this coefficient to the corresponding
coefficient of \(\IRL_{-s}f\) is
\[
        \lambda_n=
        \frac{2}{\Gamma(s+1)}
        \frac{\Gamma(n+s+1)\Gamma(n+1-s)}{\Gamma(n+1)^2}.
\]
The gamma-quotient expansion in
Lemma~\ref{lem:background-tools}\textup{(v)} shows that \(\lambda_n\) and
\(1/\lambda_n\) have Taylor form on the tail.
Lemma~\ref{lem:D-bergman-taylor-multipliers} therefore yields the quasi-norm
equivalence
\[
        \|D^{\mu_s}f\|_{A^p_{\rho_s}}^p
        \asymp
        \|\IRL_{-s}f\|_{A^p_{\rho_s}}^p
        +\sum_{n\in E_s}|\widehat f(n)|^p.
\]
Combining this with (2.4) proves (2.3), and hence (2.2).
\end{proof}

The two-weight moment-ratio theorem in
\cite{PeralaRattyaWang2025} gives a broader norm-transfer framework for
injective fractional derivatives.  The proof above uses the one-weight form
from \cite{PelaezDeLaRosa2022} because it exposes the comparison with the exact
Riemann--Liouville coefficients and identifies the exceptional set \(E_s\).
Those finite modes are needed in the range decomposition and compactness
transfer below.

\begin{corollary}\label{cor:LP-shift-for-volterra}
Let \(0<p<\infty\), \(\omega\in\calD\), \(\alpha,\beta>0\), and assume \((H)\).
Then
\[
        f\in A^p_\omega
        \quad\Longleftrightarrow\quad
        h=\IRL_{\beta-\alpha}f
        \in A^p_{\nu_{\alpha,\beta,p}}.
        \tag{2.5}
\]
More precisely, if \(\alpha>\beta\), \(s=\alpha-\beta\), and
\[
        E_s=\{n\ge0:n+1-s\in\{0,-1,-2,\ldots\}\},
\]
then
\[
        \|f\|_{A^p_\omega}^p
        \asymp
        \|h\|_{A^p_{\nu_{\alpha,\beta,p}}}^p
        +\sum_{n\in E_s}|\widehat f(n)|^p.
        \tag{2.5$'$}
\]
If \(\alpha\le\beta\), then
\(\|f\|_{A^p_\omega}\asymp
\|h\|_{A^p_{\nu_{\alpha,\beta,p}}}\).
\end{corollary}

\begin{proof}
If \(\alpha=\beta\), then \(\IRL_{\beta-\alpha}=\Id\) and
\(\nu_{\alpha,\beta,p}=\omega\), so the claim is immediate.  If
\(\alpha>\beta\), put
\(s=\alpha-\beta\).  Then \(\IRL_{\beta-\alpha}=\IRL_{-s}\) and
\(\nu_{\alpha,\beta,p}=\omega\delta^{sp}\), so Theorem~\ref{thm:fractional-LP-shift}
applies and gives \((2.5')\).

If \(\alpha<\beta\), put \(s=\beta-\alpha\).  Apply
Theorem~\ref{thm:fractional-LP-shift} to \(\rho=\nu_{\alpha,\beta,p}\).  Since
\(\rho_s=\omega\), we get
\[
        \|h\|_{A^p_{\nu_{\alpha,\beta,p}}}^p
        \asymp
        \|\IRL_{-s}h\|_{A^p_\omega}^p
        +\sum_{n\in E_s}|\widehat h(n)|^p,
        \qquad
        E_s=\{n\ge0:n+1-s\in\Z_{\le0}\}.
        \tag{2.5$''$}
\]
With \(h=\IRL_s f\), the composition \(\IRL_{-s}\IRL_s\) has multiplier
\(c_n(s)=\Gamma(n+1)^2/[\Gamma(n+1-s)\Gamma(n+1+s)]\).  Outside \(E_s\),
both \(c_n(s)\) and \(c_n(s)^{-1}\) have Taylor form, while for \(n\in E_s\),
\(c_n(s)=0\) but
\(\widehat h(n)=\Gamma(n+1)\widehat f(n)/\Gamma(n+1+s)\) with a nonzero factor.
Lemma~\ref{lem:D-bergman-taylor-multipliers} therefore shows that the
right-hand side of \((2.5'')\) is comparable to
\(\|f\|_{A^p_\omega}^p\), proving the claimed quasi-norm equivalence.
\end{proof}

\begin{lemma}\label{lem:LP-range-decomposition}
Let \(0<p<\infty\), \(\omega\in\calD\), and \(\alpha,\beta>0\).  Assume
\((H)\), and put \(\nu=\nu_{\alpha,\beta,p}\).  There exist a fixed
finite-dimensional polynomial space \(\mathcal P_{\alpha,\beta}\), a bounded projection
\[
 \Pi_{\alpha,\beta}:A^p_\nu\to\mathcal P_{\alpha,\beta},
 \qquad
 R_{\alpha,\beta}:(\Id-\Pi_{\alpha,\beta})A^p_\nu\to A^p_\omega
\]
such that
\[
        \IRL_{\beta-\alpha}R_{\alpha,\beta}h=h,
        \qquad h\in (\Id-\Pi_{\alpha,\beta})A^p_\nu .
        \tag{2.5a}
\]
Moreover, for every \(h\in A^p_\nu\),
\[
        \left\|R_{\alpha,\beta}
        \bigl(h-\Pi_{\alpha,\beta}h\bigr)\right\|_{A^p_\omega}
        +
        \left\|\Pi_{\alpha,\beta}h\right\|_{\mathrm{poly}}
        \lesssim
        \|h\|_{A^p_\nu},
        \tag{2.5b}
\]
where \(\|\cdot\|_{\mathrm{poly}}\) is any norm on the fixed finite-dimensional
space.  If \(h_j\) is bounded in \(A^p_\nu\) and \(h_j\to0\) locally
uniformly, then
\[
 \Pi_{\alpha,\beta}h_j\to0\quad\text{in the polynomial norm},\qquad
 R_{\alpha,\beta}(h_j-\Pi_{\alpha,\beta}h_j)\to0
\]
locally uniformly on \(\D\).
\end{lemma}

\begin{proof}
If \(\alpha=\beta\), then \(\nu=\omega\),
\(\IRL_{\beta-\alpha}=\Id\), and
\(\Pi_{\alpha,\beta}=0\), \(R_{\alpha,\beta}=\Id\) have all the required
properties.

Suppose next that \(\alpha>\beta\).  Put \(s=\alpha-\beta>0\), so that
\(\IRL_{\beta-\alpha}=\IRL_{-s}\) and \(\nu=\omega\delta^{sp}\), and let
\[
        E_s=\{n\ge0:n+1-s\in\{0,-1,-2,\ldots\}\}.
\]
Thus \(E_s=\varnothing\) unless \(s\in\N\), in which case
\(E_s=\{0,1,\ldots,s-1\}\).  Define
\[
        \Pi_{\alpha,\beta}h(z)=\sum_{n\in E_s}\widehat h(n)z^n.
\]
This is a bounded finite-rank projection on \(A^p_\nu\).  For
\(h\in(\Id-\Pi_{\alpha,\beta})A^p_\nu\), set
\[
        R_{\alpha,\beta}h(z)
        =
        \sum_{n\notin E_s}
        \frac{\Gamma(n+1-s)}{\Gamma(n+1)}\widehat h(n)z^n.
        \tag{2.5c}
\]
Then \(\IRL_{-s}R_{\alpha,\beta}h=h\).  Since the coefficients indexed by
\(E_s\) vanish, Theorem~\ref{thm:fractional-LP-shift} gives
\[
        \|R_{\alpha,\beta}h\|_{A^p_\omega}
        \lesssim
        \|h\|_{A^p_\nu},
        \qquad h\in(\Id-\Pi_{\alpha,\beta})A^p_\nu.
\]
Together with boundedness of the finite-rank projection, this is \((2.5b)\).

Finally, suppose that \(\alpha<\beta\), and put \(s=\beta-\alpha>0\).
Then \(\IRL_{\beta-\alpha}=\IRL_s\),
\(\nu=\omega\delta^{-sp}\), and \(\nu\delta^{sp}=\omega\).
There is no range obstruction for \(\IRL_s\).  We set
\(\Pi_{\alpha,\beta}=0\) and define
\[
        R_{\alpha,\beta}h(z)
        =
        \sum_{n=0}^\infty
        \frac{\Gamma(n+1+s)}{\Gamma(n+1)}\widehat h(n)z^n.
        \tag{2.5d}
\]
Then \(\IRL_sR_{\alpha,\beta}h=h\) coefficientwise.  Applying
Theorem~\ref{thm:fractional-LP-shift} to \(\rho=\nu\), for which
\(\rho_s=\omega\), gives
\[
        \|\IRL_{-s}h\|_{A^p_\omega}
        \lesssim
        \|h\|_{A^p_\nu},
        \qquad h\in A^p_\nu.
        \tag{2.5e}
\]
If \(s\notin\N\), then
\(R_{\alpha,\beta}h=M_s(\IRL_{-s}h)\), where
\[
        M_sF(z)=\sum_{n=0}^\infty m_{s,n}\widehat F(n)z^n,
        \qquad
        m_{s,n}
        =
        \frac{\Gamma(n+1+s)\Gamma(n+1-s)}{\Gamma(n+1)^2}.
\]
By Lemma~\ref{lem:background-tools}\textup{(v)}, the sequence
\(\{m_{s,n}\}\) has Taylor form on the tail.  Hence
Lemma~\ref{lem:D-bergman-taylor-multipliers}, applied to \(A^p_\omega\), gives
\[
        \|R_{\alpha,\beta}h\|_{A^p_\omega}
        \lesssim
        \|\IRL_{-s}h\|_{A^p_\omega}
        \lesssim
        \|h\|_{A^p_\nu}.
\]
If \(s=m\in\N\), the same identity holds on the tail \(n\ge m\).  The missing
low-order terms form the polynomial
\[
        Q_h(z)=
        \sum_{n=0}^{m-1}
        \frac{\Gamma(n+1+m)}{\Gamma(n+1)}\widehat h(n)z^n.
\]
Lemma~\ref{lem:background-tools}\textup{(iv)} gives
\(\|Q_h\|_{A^p_\omega}\lesssim\|h\|_{A^p_\nu}\).
On the tail, \(R_{\alpha,\beta}h-Q_h=M_s(\IRL_{-s}h)\), where now
\(M_s\) is understood after deleting the finitely many exceptional coefficients.
Lemma~\ref{lem:D-bergman-taylor-multipliers} and \((2.5e)\) therefore yield
\[
        \|R_{\alpha,\beta}h\|_{A^p_\omega}
        \lesssim
        \|\IRL_{-s}h\|_{A^p_\omega}+\|Q_h\|_{A^p_\omega}
        \lesssim
        \|h\|_{A^p_\nu}.
\]
This proves \((2.5b)\) in the negative-shift case as well, with
\(\Pi_{\alpha,\beta}=0\).

It remains to prove the local uniform convergence assertion.  If
\(h_j\to0\) locally uniformly, Cauchy's estimates imply convergence of every
Taylor coefficient to zero; hence \(\Pi_{\alpha,\beta}h_j\to0\) in the
polynomial norm.  By Lemma~\ref{lem:background-tools}\textup{(iv)}, the sequence
\[
        R_{\alpha,\beta}
        \bigl(h_j-\Pi_{\alpha,\beta}h_j\bigr)
\]
is locally bounded and therefore normal.  Every locally uniform subsequential
limit has all Taylor coefficients equal to zero, because each coefficient is a
fixed scalar multiple of a coefficient of
\(h_j-\Pi_{\alpha,\beta}h_j\).  Thus every convergent subsequence has limit
zero, and normality forces the full sequence to converge locally uniformly to
zero.
\end{proof}

\subsection{Reduction to the modified model}

\begin{lemma}\label{lem:taylor-multiplier-Hq}
Let \(0<q<\infty\).  Every Taylor-form coefficient multiplier is bounded on
\(H^q\).  If the multiplier sequence and its reciprocal have Taylor form
outside a finite set, then the multiplier is an isomorphism modulo
finite-dimensional polynomials.  Assigning nonzero values at the exceptional
indices yields a bounded automorphism of \(H^q\).
\end{lemma}

\begin{proof}
The Hardy-space multiplier assertion follows from
\cite[Lemmas~5--6]{FangGuoHouZhuBergmanHardy}.  Applying the same result
to \(1/\lambda_n\) on the tail gives the reciprocal multiplier.  For each
fixed \(n\), the Taylor coefficient functional
\(f\mapsto\widehat f(n)\) is continuous on \(H^q\); hence changing finitely
many coefficients produces a bounded finite-rank polynomial operator.
After assigning nonzero values at the exceptional indices, the two
coefficientwise multipliers are bounded inverses of one another on \(H^q\).
\end{proof}

Define the modified model
\[
        \widetilde V^\varphi_{\alpha,\beta}f
        =
        \IRL_{\alpha}
        \left(\IRL_{\beta-\alpha}f\,\IRL_{-\beta}\varphi\right).
        \tag{2.6}
\]

\begin{proposition}\label{prop:model-reduction}
Let \(0<p,q<\infty\), let \(\omega\) be a radial weight, let
\(\varphi\in H(\D)\), and let \((\alpha,\beta)\in\calP\) with
\(\alpha,\beta>0\).  Then
\(V^\varphi_{\alpha,\beta}:A^p_\omega\to H^q\) is bounded, respectively
compact, if and only if
\(\widetilde V^\varphi_{\alpha,\beta}:A^p_\omega\to H^q\) is bounded,
respectively compact.
\end{proposition}

\begin{proof}
We compare the coefficients of the original (bona fide) Riemann--Liouville
realization and the modified coefficient model.  The integer case requires
separate treatment because the formal factor \(p_{-\alpha}\) then belongs to
the exceptional set of the original calculus.

Write \(f(z)=\sum_{n\ge0}a_nz^n\) and
\(\varphi(z)=\sum_{n\ge0}b_nz^n\).

First suppose that \(\alpha\notin\N\).  Then \(-\alpha\notin\{-1,-2,\ldots\}\),
and the fractional-calculus representation of the analytic
paraproduct~\cite[Lemmas~3.4--3.5]{FangGuoHouZhu2025Studia} gives a
sequence \(\{c_n(f)\}_{n\ge0}\) such that, on \(\D\setminus\{0\}\),
\[
        \Ihat_{\beta-\alpha}f\,\Ihat_{-\beta}\varphi
        =p_{-\alpha}(z)\sum_{n=0}^{\infty}c_n(f)z^n .
        \tag{2.6a}
\]
Here the coefficients are explicitly
\[
        c_n(f)=
        \sum_{k+j=n}
        \frac{\Gamma(k+1)}{\Gamma(k+1+\beta-\alpha)}
        \frac{\Gamma(j+1)}{\Gamma(j+1-\beta)}
        a_k b_j,
        \qquad n\ge0,
\]
where \(1/\Gamma(-m)=0\) for \(m=0,1,2,\ldots\), as justified by
Lemma~\ref{lem:background-tools}\textup{(v)}.  This formula is obtained by
collecting the coefficient of \(p_{-\alpha}(z)z^n\) in the original product.
Admissibility of \((\alpha,\beta)\) ensures that this coefficient is well
defined.  The same finite convolution is the ordinary Taylor coefficient of
\(\IRL_{\beta-\alpha}f\,\IRL_{-\beta}\varphi\).  Applying \(\Ihat_\alpha\) to
\((2.6a)\) and using the coefficient model in \((2.6)\) give
\begin{align}
 V^\varphi_{\alpha,\beta}f(z)
 &=\sum_{n=0}^{\infty}
 \frac{\Gamma(n+1-\alpha)}{\Gamma(n+1)}c_n(f)z^n ,
 \tag{2.6b}\\
 \widetilde V^\varphi_{\alpha,\beta}f(z)
 &=\sum_{n=0}^{\infty}
 \frac{\Gamma(n+1)}{\Gamma(n+1+\alpha)}c_n(f)z^n .
 \tag{2.6c}
\end{align}
This coefficient comparison between the original and modified
Riemann--Liouville models is the explicit form used in the standard-weight
theory~\cite[Equations~(15)--(16) and the paragraph following
them]{FangGuoHouZhuBergmanHardy}.

Now suppose that \(\alpha=m\in\N\).  Since \((\alpha,\beta)\in\calP\) and
\(\beta>0\), we then have \(\beta=\ell\in\N\).  In this case we do not use the
formal symbol \(p_{-m}\).  Instead, define \(c_N(f)=0\) for \(0\le N<m\), and
for \(N\ge m\) let \(c_N(f)\) be the coefficient of \(z^{N-m}\) in the analytic
product
\[
        \Ihat_{\ell-m}f\,\Ihat_{-\ell}\varphi .
\]
This indexing aligns the final Taylor degree in the two models; the missing
degrees \(N<m\) form the finite-dimensional correction.
Equivalently, with the convention \(1/\Gamma(-j)=0\) for \(j=0,1,2,\ldots\)
from Lemma~\ref{lem:background-tools}\textup{(v)},
\[
        c_N(f)=
        \sum_{k+j=N}
        \frac{\Gamma(k+1)}{\Gamma(k+1+\ell-m)}
        \frac{\Gamma(j+1)}{\Gamma(j+1-\ell)}
        a_k b_j,
        \qquad N\ge0.
        \tag{2.6d}
\]
Formula \((2.6d)\) also gives \(c_N(f)=0\) for \(N<m\): a nonzero summand
requires \(j\ge\ell\) and, when \(\ell<m\), \(k\ge m-\ell\), so
\(N=k+j\ge m\).  The original \(m\)-fold integration and the modified model
therefore give
\begin{align}
 V^\varphi_{m,\ell}f(z)
 &=\sum_{N=m}^{\infty}
 \frac{\Gamma(N+1-m)}{\Gamma(N+1)}c_N(f)z^N ,
 \tag{2.6e}\\
 \widetilde V^\varphi_{m,\ell}f(z)
 &=\sum_{N=0}^{\infty}
 \frac{\Gamma(N+1)}{\Gamma(N+1+m)}c_N(f)z^N .
 \tag{2.6f}
\end{align}
Since \(c_N(f)=0\) for \(N<m\), equations \((2.6e)\)--\((2.6f)\) yield the
same coefficient ratio on every nonzero term, without invoking the undefined
formal factor \(p_{-m}\).

Thus in both cases there is an integer \(N_0=N_0(\alpha)\) such that, for
\(n\ge N_0\), the \(n\)-th coefficient of the original output equals
\[
        m_{\alpha,n}
\]
times the \(n\)-th coefficient of the modified output, where
\[
        m_{\alpha,n}
        =
        \frac{\Gamma(n+1-\alpha)\Gamma(n+1+\alpha)}{\Gamma(n+1)^2}.
        \tag{2.6g}
\]
Choose arbitrary nonzero values for \(m_{\alpha,n}\) when \(n<N_0\), and let
\(M_{m_\alpha}\) be the corresponding coefficient multiplier.  By
Lemma~\ref{lem:background-tools}\textup{(v)}, \(\{m_{\alpha,n}\}\) has Taylor form on the tail,
is bounded away from zero there, and has a Taylor-form reciprocal.  Hence
\(M_{m_\alpha}\) and the coefficientwise reciprocal multiplier are bounded on
\(H^q\) by Lemma~\ref{lem:taylor-multiplier-Hq}.  Thus \(M_{m_\alpha}\) is a
bounded automorphism of \(H^q\).

The comparison implies
\[
        V^\varphi_{\alpha,\beta}f
        =
        M_{m_\alpha}\widetilde V^\varphi_{\alpha,\beta}f
        +P_{\alpha,\beta,\varphi}f,
        \tag{2.6h}
\]
where \(P_{\alpha,\beta,\varphi}f\) is a polynomial of degree less than \(N_0\).
Its coefficients are finite linear combinations of \(c_0(f),\ldots,c_{N_0-1}(f)\),
and hence finite linear combinations of the Taylor coefficients
\(a_0,\ldots,a_{N_0-1}\) of \(f\).  These coefficient functionals are bounded on
\(A^p_\omega\) by Lemma~\ref{lem:background-tools}\textup{(iv)}.  Therefore
\[
        P_{\alpha,\beta,\varphi}:A^p_\omega\to H^q
\]
is bounded and has finite-dimensional range; in particular it is compact.

Thus boundedness or compactness of the modified model implies the same
property for the original model.  Conversely, applying \(M_{m_\alpha}^{-1}\)
to \((2.6h)\) gives
\[
        \widetilde V^\varphi_{\alpha,\beta}f
        =
        M_{m_\alpha}^{-1}V^\varphi_{\alpha,\beta}f
        -M_{m_\alpha}^{-1}P_{\alpha,\beta,\varphi}f .
\]
The second term is finite rank, so either property of the original model
implies the corresponding property of the modified one.
\end{proof}

\subsection{Fractional \texorpdfstring{\(g\)}{g}-function reduction}

The reduction to an area map requires the fractional \(g\)-function and the
analytic-tent multiplier theorem used in the standard-weight theory; the
ordinary Bergman Littlewood--Paley theorem alone is insufficient.

For \(a>0\) and \(0<q<\infty\), define
\[
        \|U\|_{\mathcal A_a^q}
        =
        \left(
        \int_{\T}
        \left(
        \int_{\Gamma(\xi)}|U(z)|^2\delta(z)^{2a-2}\dA(z)
        \right)^{q/2}\dxi
        \right)^{1/q}.
\]

\begin{lemma}\label{lem:fractional-primitive-area}
Let \(0<q<\infty\) and \(a>0\).  Then, for every \(U\in H(\D)\),
\[
        \|\IRL_a U\|_{H^q}
        \asymp
        \|U\|_{\mathcal A_a^q}.
\]
\end{lemma}

\begin{proof}
Since \((x^q+y^q)^{1/q}\asymp_q x+y\) for \(x,y\ge0\), the \(q\)-th-power
Flett formula \cite[Eq.~(5.3)]{FangGuoHouZhu2025Studia}, applied to
\(F-F(0)\) and combined with boundedness of the constant-term projection on
\(H^q\), is equivalent for every \(0<q<\infty\) to
\[
        \|F\|_{H^q}
        \asymp
        |\widehat F(0)|
        +\|I^F_{-a}(F-F(0))\|_{\mathcal A_a^q}.
        \tag{2.6i}
\]
Here
\[
        I^F_{-a}(F-F(0))(z)
        =\sum_{n=1}^{\infty}(n+1)^a\widehat F(n)z^n
\]
is the Flett coefficient operator; see also
\cite[Theorem~9 and the following remark]{FangGuoHouZhuBergmanHardy}.
For \(n\ge1\) outside the finite exceptional set \(E_a\) defined below, the
quotient
\[
        \frac{\Gamma(n+1)}
        {\Gamma(n+1-a)(n+1)^a}
\]
and its reciprocal have Taylor form by
Lemma~\ref{lem:background-tools}\textup{(v)}.  Since \(2a-2>-2\), the
analytic-tent multiplier theorem
\cite[Lemma~5.4]{FangGuoHouZhu2025Studia} applies to \(\mathcal A_a^q\) and
compares the two derivative operators on the tail; the remaining coefficients
are handled separately.  If \(a\notin\N\), the constant coefficient of
\(\IRL_{-a}F\) is \(\widehat F(0)/\Gamma(1-a)\).  Since
\(1/\Gamma(1-a)\ne0\), this coefficient controls \(|\widehat F(0)|\).  If
\(a\in\N\), then \(0\in E_a\).  Thus \((2.6i)\) becomes
\[
        \|F\|_{H^q}
        \asymp
        \|\IRL_{-a}F\|_{\mathcal A_a^q}
        +
        \sum_{n\in E_a}|\widehat F(n)|,
        \qquad
        E_a=\{n\ge0:n+1-a\in\Z_{\le0}\}.
        \tag{2.6j}
\]
The finite-dimensional terms remain controlled when \(0<q<1\).  Choose
\(0<r_0<1/3\).  Then \(r_0\D\subset\Gamma(\xi)\) for every
\(\xi\in\T\), and \(\delta^{2a-2}\asymp1\) on \(r_0\D\).  Consequently
\[
        \|U\|_{\mathcal A_a^q}
        \gtrsim
        \left(\int_{r_0\D}|U(z)|^2\dA(z)\right)^{1/2}.
        \tag{2.6k}
\]
Orthogonality of the monomials on \(r_0\D\) therefore gives, for each fixed
\(n\ge0\),
\[
        |\widehat U(n)|
        \lesssim_{n,r_0,a}
        \|U\|_{\mathcal A_a^q}.
\]
Conversely, every fixed polynomial has finite tent quasi-norm because
\(2a-2>-2\) and a Stolz region has cross-section comparable to \(\delta\) near
the boundary.  Hence the finite-dimensional correction is valid for every
\(0<q<\infty\) and does not require local convexity.

Apply this equivalence to \(F=\IRL_a U\).  Then
\[
        \IRL_{-a}F
        =
        \IRL_{-a}\IRL_aU
        =M_aU,
\]
where
\[
        M_aU(z)=\sum_{n=0}^{\infty}m_{a,n}\widehat U(n)z^n,
        \qquad
        m_{a,n}=\frac{\Gamma(n+1)^2}
        {\Gamma(n+1+a)\Gamma(n+1-a)}.
\]
Outside \(E_a\), the sequence \(m_{a,n}\) has Taylor form; for all
sufficiently large \(n\), it is bounded away from zero and its reciprocal also
has Taylor form.  Hence
\[
        \|M_aU\|_{\mathcal A_a^q}
        +
        \sum_{n\in E_a}|\widehat U(n)|
        \asymp
        \|U\|_{\mathcal A_a^q}.
\]
For each fixed index, the coefficient of \(\IRL_aU\) is a nonzero constant
multiple of the corresponding coefficient of \(U\).  Combining these facts
gives
\[
        \|\IRL_aU\|_{H^q}
        \asymp
        \|U\|_{\mathcal A_a^q}.
\]
\end{proof}

For the symbol \(\varphi\), set \(G_\varphi=\IRL_{-\beta}\varphi\).  For
analytic \(h\), define
\[
 A_{\alpha,\varphi}h(\xi)=
 \left(\int_{\Gamma(\xi)}
 |h(z)|^2|G_\varphi(z)|^2\delta(z)^{2\alpha-2}\dA(z)\right)^{1/2},
 \qquad
 \|A_{\alpha,\varphi}h\|_{L^q(\T)}
 =\|hG_\varphi\|_{\mathcal A_\alpha^q}.
\]

\begin{proposition}\label{prop:g-function-reduction}
Let \(0<p,q<\infty\), let \(\omega\in\calD\), let \(\alpha,\beta>0\), and let
\(\varphi\in H(\D)\).  Assume \((\alpha,\beta)\in\calP\) and \((H)\), and put
\(\nu=\nu_{\alpha,\beta,p}\).  Then
\(V^\varphi_{\alpha,\beta}:A^p_\omega\to H^q\) is bounded, respectively
compact, if and only if \(A_{\alpha,\varphi}:A^p_\nu\to L^q(\T)\) is
bounded, respectively compact.  Here compactness of
the positive sublinear area map \(A_{\alpha,\varphi}\) means boundedness together
with the following sequential condition: whenever \(\{h_j\}\) is bounded in \(A^p_\nu\) and
\(h_j\to0\) locally uniformly in \(\D\), then
\[
        \|A_{\alpha,\varphi}h_j\|_{L^q(\T)}\to0.
\]
\end{proposition}

\begin{proof}
By Proposition~\ref{prop:model-reduction}, it is enough to work with the
modified model \((2.6)\).  Put \(\nu=\nu_{\alpha,\beta,p}\),
\(h=\IRL_{\beta-\alpha}f\), and \(U=hG_\varphi\).
By Corollary~\ref{cor:LP-shift-for-volterra},
\(f\mapsto\IRL_{\beta-\alpha}f\) maps \(A^p_\omega\) onto a
finite-codimensional subspace of \(A^p_\nu\).  More precisely,
Lemma~\ref{lem:LP-range-decomposition} gives, for every \(h\in A^p_\nu\),
\[
        h=(h-P_h)+P_h,
        \qquad P_h=\Pi_{\alpha,\beta}h,
\]
where \(h-P_h\) lies in the exact range of \(\IRL_{\beta-\alpha}\), and there is
\(f=R_{\alpha,\beta}(h-P_h)\in A^p_\omega\) such that
\[
        \IRL_{\beta-\alpha}f=h-P_h,
        \qquad
        \|f\|_{A^p_\omega}+\|P_h\|_{\mathrm{poly}}
        \lesssim
        \|h\|_{A^p_\nu}.
\]

For functions in the exact range of the shift, Lemma~\ref{lem:fractional-primitive-area}
gives
\[
\begin{aligned}
        \|\widetilde V^\varphi_{\alpha,\beta}f\|_{H^q}
        &=
        \|\IRL_\alpha(hG_\varphi)\|_{H^q}       \\
        &\asymp
        \|hG_\varphi\|_{\mathcal A_\alpha^q}
        =
        \|A_{\alpha,\varphi}h\|_{L^q(\T)}.
\end{aligned}
\]
The area map is pointwise subadditive:
\[
        A_{\alpha,\varphi}(h_1+h_2)
        \le
        A_{\alpha,\varphi}h_1+A_{\alpha,\varphi}h_2.
        \tag{2.8a}
\]
We therefore use Minkowski's inequality when \(q\ge1\) and the power
inequality after taking \(q\)-th powers when \(0<q<1\); see
Lemma~\ref{lem:background-tools}\textup{(i)}.
If \(A_{\alpha,\varphi}:A^p_\nu\to L^q(\T)\) is bounded and \(f\in A^p_\omega\), then
\(h=\IRL_{\beta-\alpha}f\in A^p_\nu\) and
\(\|h\|_{A^p_\nu}\lesssim\|f\|_{A^p_\omega}\).  Hence
\[
        \|\widetilde V^\varphi_{\alpha,\beta}f\|_{H^q}
        \lesssim
        \|A_{\alpha,\varphi}h\|_{L^q}
        \lesssim
        \|h\|_{A^p_\nu}
        \lesssim
        \|f\|_{A^p_\omega}.
\]
Thus \(\widetilde V^\varphi_{\alpha,\beta}\) is bounded, and the original operator is
bounded by Proposition~\ref{prop:model-reduction}.

Conversely, suppose that \(\widetilde V^\varphi_{\alpha,\beta}\) is bounded and let
\(h\in A^p_\nu\).  Choose \(f\) and \(P_h\) as above.  Then
\[
        \|A_{\alpha,\varphi}(h-P_h)\|_{L^q}
        \asymp
        \|\widetilde V^\varphi_{\alpha,\beta}f\|_{H^q}
        \lesssim
        \|h\|_{A^p_\nu}.
\]
It remains to control the finite-dimensional complement.  It suffices to treat
one polynomial \(P\) from a fixed basis.  Choose an integer \(N\) so large that
\(z^NP\) belongs to the exact range of
\(\IRL_{\beta-\alpha}\).  The preceding estimate applied to \(z^NP\) gives
\(\|A_{\alpha,\varphi}(z^NP)\|_{L^q}<\infty\).  Fix \(0<r<1\).  Since
\(|z|^{-N}\le r^{-N}\) on \(\D\setminus r\D\), while
\[
        C_{P,r}^2=
        \int_{r\D}|P(z)|^2|G_\varphi(z)|^2
        \delta(z)^{2\alpha-2}\dA(z)<\infty,
\]
we have pointwise on \(\T\)
\[
        A_{\alpha,\varphi}P
        \le r^{-N}A_{\alpha,\varphi}(z^NP)+C_{P,r}.
        \tag{2.8b}
\]
Since \(\T\) has finite measure, \((2.8b)\) implies
\(\|A_{\alpha,\varphi}P\|_{L^q}<\infty\), with the power inequality used
when \(0<q<1\).  Applying this to a fixed basis and using finite
dimensionality and subadditivity shows that
\(P\mapsto A_{\alpha,\varphi}P\) is continuous from the polynomial space into
\(L^q(\T)\).  Consequently
\[
        \|A_{\alpha,\varphi}P_h\|_{L^q}
        \lesssim
        \|P_h\|_{\mathrm{poly}}
        \lesssim
        \|h\|_{A^p_\nu}.
\]
Using \((2.8a)\) and Lemma~\ref{lem:background-tools}\textup{(i)} gives
\(\|A_{\alpha,\varphi}h\|_{L^q}\lesssim \|h\|_{A^p_\nu}\).  Thus
\(A_{\alpha,\varphi}\) is bounded.

Compactness is handled by the same decomposition into the exact range and its
finite-dimensional complement.  Each fractional coefficient operator is
continuous in the compact-open topology by Cauchy's estimates and the
polynomial growth of its gamma quotient; multiplication by the fixed analytic
symbol is likewise continuous.  Hence
\(\widetilde V^\varphi_{\alpha,\beta}:H(\D)\to H(\D)\) is continuous in the
compact-open topology, and Lemma~\ref{lem:analytic-compactness-criterion}
applies to it.
If \(A_{\alpha,\varphi}\) satisfies the sequential compactness condition and
\(f_j\) is bounded in \(A^p_\omega\) with \(f_j\to0\) locally uniformly, then
\(h_j=\IRL_{\beta-\alpha}f_j\) is bounded in \(A^p_\nu\).  For
\(0<r<R<1\), Cauchy's estimate and the \(O((n+1)^K)\) growth of the
multiplier give
\[
 \sup_{|z|\le r}|h_j(z)|
 \le C M_j(R)\sum_{n\ge0}(n+1)^K(r/R)^n\longrightarrow0,
 \qquad M_j(R)=\max_{|z|=R}|f_j(z)|\to0,
\]
so \(h_j\to0\) locally uniformly.
Thus \(\|A_{\alpha,\varphi}h_j\|_{L^q}\to0\), and the preceding norm
equivalence gives \(\|\widetilde V^\varphi_{\alpha,\beta}f_j\|_{H^q}\to0\).
Lemma~\ref{lem:analytic-compactness-criterion} now shows that
\(\widetilde V^\varphi_{\alpha,\beta}\) is compact.

Conversely, assume that \(\widetilde V^\varphi_{\alpha,\beta}\) is compact.  Let
\(h_j\) be bounded in \(A^p_\nu\) and tend to zero locally uniformly.  Put
\[
        P_{h_j}=\Pi_{\alpha,\beta}h_j,
        \qquad
        f_j=R_{\alpha,\beta}(h_j-P_{h_j}).
\]
By Lemma~\ref{lem:LP-range-decomposition}, \(P_{h_j}\to0\) in the polynomial norm,
\(\{f_j\}\) is bounded in \(A^p_\omega\), and \(f_j\to0\) locally uniformly.
Lemma~\ref{lem:analytic-compactness-criterion} gives
\(\|\widetilde V^\varphi_{\alpha,\beta}f_j\|_{H^q}\to0\), hence
\(\|A_{\alpha,\varphi}(h_j-P_{h_j})\|_{L^q}\to0\).  The finite-dimensional estimate
above gives \(\|A_{\alpha,\varphi}P_{h_j}\|_{L^q}\to0\).  Since
\[
        A_{\alpha,\varphi}h_j
        \le
        A_{\alpha,\varphi}(h_j-P_{h_j})+A_{\alpha,\varphi}P_{h_j},
\]
Lemma~\ref{lem:background-tools}\textup{(i)} gives
\(\|A_{\alpha,\varphi}h_j\|_{L^q}\to0\).  Thus \(A_{\alpha,\varphi}\) satisfies the
sequential compactness condition.

Finally, Proposition~\ref{prop:model-reduction} transfers the boundedness and
compactness conclusions between the modified and original models.
\end{proof}

\subsection{The Volterra-weight dictionary}

Set
\[
        \nu=\nu_{\alpha,\beta,p}=\omega\delta^{p(\alpha-\beta)},
        \qquad
        u_\varphi(z)=G_\varphi(z)\delta(z)^{\alpha-1}.
        \tag{2.9}
\]
Under \((H)\),
\[
        \widehat\nu(z)
        \asymp
        \widehat\omega(z)\delta(z)^{p(\alpha-\beta)}.
        \tag{2.10}
\]
Then
\[
        A_{\alpha,\varphi}h(\xi)=
        \left(
        \int_{\Gamma(\xi)}|h(z)|^2|u_\varphi(z)|^2\dA(z)
        \right)^{1/2}.
        \tag{2.11}
\]
Thus \((2.11)\) is the weighted Volterra area-map problem under the
substitutions
\[
        \omega\rightsquigarrow\nu,
        \qquad
        g'(z)\rightsquigarrow u_\varphi(z),
        \qquad
        f\rightsquigarrow h=\IRL_{\beta-\alpha}f.
        \tag{2.12}
\]

\begin{lemma}\label{lem:radialized-symbol-local}
Let \(0<s<\infty\), \(0<r<1\), and
\(u(z)=G(z)\delta(z)^\tau\), where \(G\in H(\D)\) and \(\tau\in\mathbb R\).  Then
\[
        |u(z)|^s
        \lesssim
        \frac{1}{\delta(z)^2}
        \int_{D(z,r)}|u(\zeta)|^s\dA(\zeta).
        \tag{2.13}
\]
\end{lemma}

\begin{proof}
For \(\zeta\in D(z,r)\), Lemma~\ref{lem:background-tools}\textup{(vi)} gives
\(\delta(\zeta)\asymp\delta(z)\).  Since \(|G|^s\) is
subharmonic by Lemma~\ref{lem:background-tools}\textup{(iii)},
\[
        |G(z)|^s
        \lesssim
        \delta(z)^{-2}\int_{D(z,r)}|G(\zeta)|^s\dA(\zeta).
\]
Multiplying by \(\delta(z)^{s\tau}\) and using local comparability of \(\delta\)
gives (2.13).
\end{proof}

The following proposition records the resulting exponent identities.  Because
\(u_\varphi=G_\varphi\delta^{\alpha-1}\) need not be analytic, the local
submean estimates required below are supplied by
Lemma~\ref{lem:radialized-symbol-local}.

\begin{proposition}\label{prop:dictionary}
Assume \((H)\).  Substituting \((2.9)\)--\((2.10)\) into the pointwise,
Carleson, tent, and maximal conditions for the area map \((2.11)\) yields the
four conditions in Theorems~\ref{thm:boundedness-main} and
\ref{thm:compactness-main}.  The substitutions corresponding to conditions
\textup{(B)} and \textup{(C)} are asserted only when \(p>2\).
\end{proposition}

\begin{proof}
From \((2.9)\)--\((2.10)\),
\(|u_\varphi|=|G_\varphi|\delta^{\alpha-1}\) and
\(\widehat\nu^{-1/p}\asymp
\widehat\omega^{-1/p}\delta^{-(\alpha-\beta)}\).  Hence the four
substitutions are
\begin{align*}
 |u_\varphi|\widehat\nu^{-1/p}\delta^{1+1/q-1/p}
 &\asymp
 |G_\varphi|\widehat\omega^{-1/p}\delta^{\beta+1/q-1/p},\\
 |u_\varphi|^{\frac{2p}{p-2}}
 \frac{\delta^{\frac{p+2}{p-2}}}{\widehat\nu^{\frac{2}{p-2}}}
 &\asymp
 |G_\varphi|^{\frac{2p}{p-2}}
 \frac{\delta^{\frac{p+2+2p(\beta-1)}{p-2}}}
      {\widehat\omega^{\frac{2}{p-2}}},\\
 |u_\varphi|^{\frac{2p}{p-2}}
 \left(\frac{\delta^2}{\widehat\nu}\right)^{\frac{2}{p-2}}
 &\asymp
 |G_\varphi|^{\frac{2p}{p-2}}
 \frac{\delta^{\frac{4+2p(\beta-1)}{p-2}}}
      {\widehat\omega^{\frac{2}{p-2}}},\\
 |u_\varphi|\frac{\delta}{\widehat\nu^{1/p}}
 &\asymp
 |G_\varphi|\frac{\delta^\beta}{\widehat\omega^{1/p}}.
\end{align*}
For the middle two lines, the powers are obtained from
\[
 2p(\alpha-1)+p+2-2p(\alpha-\beta)
 =p+2+2p(\beta-1),
 \quad
 2p(\alpha-1)+4-2p(\alpha-\beta)
 =4+2p(\beta-1).
\]
These pointwise comparisons have constants independent of \(z\), so they
also preserve boundary limits, vanishing Carleson properties, and truncated
tent tails.
\end{proof}

\section{Area-map criteria in the four exponent ranges}\label{sec:four-cases}

The classical result~\cite{DuanWangWang2021} identifies the four geometric
regimes.  After the fractional reductions of Section~2, boundedness and
compactness are determined by the corresponding area-map criteria below.

\subsection{Case A: boundedness}

Assume
\[
        0<p\le \min\{2,q\}
        \qquad\hbox{or}\qquad
        2<p<q<\infty.
\]

\begin{lemma}\label{lem:caseA-fractional-embeddings}
Let \(\alpha>0\) and \(U\in H(\D)\).
\begin{enumerate}[label=\textup{(\roman*)},leftmargin=2em]
\item If \(0<s\le2\), then
\[
        \|\IRL_{\alpha}U\|_{H^s}
        \lesssim
        \left(\int_{\D}|U(z)|^s\delta(z)^{\alpha s-1}\dA(z)\right)^{1/s}.
        \tag{3.1}
\]
\item If \(0<s<t<\infty\), then
\[
        \|\IRL_{\alpha}U\|_{H^t}
        \lesssim
        \left(\int_{\D}|U(z)|^s\delta(z)^{s/t+s\alpha-2}\dA(z)\right)^{1/s}.
        \tag{3.2}
\]
\end{enumerate}
\end{lemma}

\begin{proof}
Both assertions follow from the modified Riemann--Liouville form of the
fractional Bergman-to-Hardy integration theorem
\cite[Theorem~7]{FangGuoHouZhuBergmanHardy}, which also treats the Hadamard
and Flett realizations.  For \textup{(i)} we apply this result with exponent
\(s\) and weight exponent
\(\alpha s-1>-1\), obtaining
\[
        \|\IRL_\alpha U\|_{H^s}
        \lesssim
        \left(\int_\D |U(z)|^s\delta(z)^{\alpha s-1}\dA(z)\right)^{1/s},
        \qquad 0<s\le2.
\]
For \textup{(ii)}, put
\[
        \gamma=s/t+s\alpha-2.
\]
If \(\gamma>-1\), then
\[
        \frac{\gamma+2}{s}-\frac1t
        =
        \alpha,
\]
and the same fractional integration theorem gives
\[
        \|\IRL_\alpha U\|_{H^t}
        \lesssim
        \left(\int_\D |U(z)|^s\delta(z)^\gamma\dA(z)\right)^{1/s}.
\]
If \(\gamma\le -1\), the right-hand side of
\((3.2)\) is finite only for \(U\equiv0\).  If \(U\not\equiv0\), then the
integral means
\[
        M_s(r,U)=
        \left(\int_\T |U(r\xi)|^s\dxi\right)^{1/s}
\]
are nondecreasing by Lemma~\ref{lem:background-tools}\textup{(iii)} and are
positive for all sufficiently large \(r\); hence they are eventually bounded
from below by a positive constant.  Using polar coordinates and normalized
area measure, for some \(r_0<1\),
\[
\begin{aligned}
        \int_\D |U(z)|^s\delta(z)^\gamma\dA(z)
        &=2\int_0^1 M_s(r,U)^s(1-r^2)^\gamma r\,dr\\
        &\gtrsim \int_{r_0}^1(1-r)^\gamma\,dr=\infty,
\end{aligned}
\]
because \(1-r^2\asymp1-r\) on \([r_0,1)\).
Thus, if \(U\not\equiv0\), the right-hand side of \((3.2)\) is infinite and
the estimate is vacuous; if \(U\equiv0\), both sides vanish.
\end{proof}

\begin{lemma}\label{lem:D-normalized-test-functions}
Let \(0<p<\infty\) and \(\nu\in\calD\).  There exists an integer \(N>0\) such
that, for \(a\in\D\),
\[
        h_a(z)=
        \frac{\delta(a)^{N-1/p}}
             {\widehat\nu(a)^{1/p}(1-\overline a z)^N}
\]
satisfies
\[
        \|h_a\|_{A^p_\nu}\lesssim1,
        \qquad
        |h_a(a)|=
        \frac{1}{\widehat\nu(a)^{1/p}\delta(a)^{1/p}},
        \tag{3.4}
\]
and \(h_a\to0\) uniformly on compact subsets as \(|a|\to1^-\).
\end{lemma}

\begin{proof}
Choose the integer \(N\) large enough so that the kernel estimate for
\(\calD\)-weights \cite[Lemma~A(vi)]{PelaezRattyaSierra2018} gives
\[
        \int_{\D}\frac{\nu(z)}{|1-\overline a z|^{Np}}\dA(z)
        \lesssim
        \frac{\widehat\nu(a)}{\delta(a)^{Np-1}}.
\]
Then
\[
        \|h_a\|_{A^p_\nu}^p
        =
        \frac{\delta(a)^{Np-1}}{\widehat\nu(a)}
        \int_{\D}\frac{\nu(z)}{|1-\overline a z|^{Np}}\dA(z)
        \lesssim1.
\]
The displayed value at \(a\) follows from \(1-\overline a a=\delta(a)\).

To prove local uniform convergence, note first that
\(\nu\in\calD\subset\widehat{\calD}\), the forward doubling condition gives a
polynomial lower bound for the tail: there is \(\lambda>0\) such that
\[
        \widehat\nu(a)\gtrsim \delta(a)^\lambda,
        \qquad a\in\D.
\]
Let \(r_n=1-2^{-n}\).  Iterating
\(\widehat\nu(r_n)\le C\widehat\nu(r_{n+1})\) gives
\[
        \widehat\nu(r_n)\ge C^{-n}\widehat\nu(0)
        =\widehat\nu(0)(1-r_n)^{\log_2 C}.
\]
Monotonicity of \(\widehat\nu\) extends this estimate from the dyadic radii to
all \(0\le r<1\), and \(1-r\asymp\delta(r)\) near the boundary.  Increasing
\(N\), if necessary, we may assume
\[
        N-\frac1p-\frac{\lambda}{p}>0.
\]
For each compact set \(K\Subset\D\), \(|1-\overline a z|\) is bounded below
uniformly for \(z\in K\) and \(|a|\) close to one.  Hence
\[
        \sup_{z\in K}|h_a(z)|
        \lesssim_K
        \frac{\delta(a)^{N-1/p}}{\widehat\nu(a)^{1/p}}
        \lesssim
        \delta(a)^{N-1/p-\lambda/p}
        \longrightarrow0,
        \qquad |a|\to1^-.
\]
\end{proof}

\begin{theorem}\label{thm:caseA-boundedness}
Let \(0<p,q<\infty\), \(\omega\in\calD\), \(\varphi\in H(\D)\),
\((\alpha,\beta)\in\calP\), and \(\alpha,\beta>0\).  Assume \((H)\) and
\[
        0<p\le\min\{2,q\}
        \qquad\text{or}\qquad
        2<p<q<\infty.
\]
Then
\(V^\varphi_{\alpha,\beta}:A^p_\omega\to H^q\) is bounded if and only if
\[
        \sup_{z\in\D}
        |G_\varphi(z)|\widehat\omega(z)^{-1/p}\delta(z)^{\beta+1/q-1/p}<\infty.
        \tag{3.5}
\]
\end{theorem}

\begin{proof}
By Propositions~\ref{prop:model-reduction} and
\ref{prop:g-function-reduction}, it suffices to work with
\(h=\IRL_{\beta-\alpha}f\in A^p_\nu\) and
\(u_\varphi=G_\varphi\delta^{\alpha-1}\).  Condition \((3.5)\) is equivalent to
\[
        \mathcal M_A:=
        \sup_{z\in\D}|G_\varphi(z)|\widehat\nu(z)^{-1/p}
        \delta(z)^{\alpha+1/q-1/p}<\infty.
        \tag{3.6}
\]

If \(p=q\le2\), apply
Lemma~\ref{lem:caseA-fractional-embeddings}\textup{(i)}; if \(p<q\), apply
part \textup{(ii)}.  In either case,
\[
\begin{aligned}
 \|\widetilde V^\varphi_{\alpha,\beta}f\|_{H^q}
 &\lesssim
 \left(\int_{\D}|h|^p|G_\varphi|^p
 \delta^{p/q+p\alpha-2}\dA\right)^{1/p}\\
 &\le\mathcal M_A
 \left(\int_{\D}|h|^p\frac{\widehat\nu}{\delta}\dA\right)^{1/p}
 \lesssim\mathcal M_A\|h\|_{A^p_\nu},
\end{aligned}
\]
where Lemma~\ref{lem:background-tools}\textup{(vi)} is used in the last
step.
For necessity, Proposition~\ref{prop:g-function-reduction} shows that the area map
\[
        A_{\alpha,\varphi}:A^p_\nu\to L^q(\T)
\]
is bounded.  Let \(h_a\) be the normalized test function from
Lemma~\ref{lem:D-normalized-test-functions}.  Then
\[
        \|A_{\alpha,\varphi}h_a\|_{L^q(\T)}\lesssim1.
\]
Let \(I(a)=\{\xi\in\T:a\in\Gamma(\xi)\}\), so that
\(|I(a)|\asymp\delta(a)\), and fix a small pseudo-hyperbolic radius \(r\).
After a fixed aperture enlargement,
Lemma~\ref{lem:background-tools}\textup{(vi)} implies that
\(D(a,r)\subset\Gamma(\xi)\) for every \(\xi\in I(a)\); see
also~\cite[Lemma~2.11]{DuanWangWang2021}.  On \(D(a,r)\),
\[
        |h_a(z)|\asymp
        \widehat\nu(a)^{-1/p}\delta(a)^{-1/p},
        \qquad
        \delta(z)\asymp\delta(a).
\]
Lemma~\ref{lem:radialized-symbol-local} also gives the local lower estimate
\[
        \int_{D(a,r)}|G_\varphi(z)|^2\delta(z)^{2\alpha-2}\dA(z)
        \gtrsim |G_\varphi(a)|^2\delta(a)^{2\alpha}.
\]
Consequently,
\[
        |G_\varphi(a)|\widehat\nu(a)^{-1/p}
        \delta(a)^{\alpha+1/q-1/p}
        \lesssim
        \|A_{\alpha,\varphi}h_a\|_{L^q(\T)}
        \lesssim1.
\]
Expanding \(\widehat\nu\) by \((H)\) gives (3.5).
\end{proof}

\subsection{Case A: compactness}

\begin{proposition}\label{prop:caseA-compactness}
Under the assumptions of Theorem~\ref{thm:caseA-boundedness},
\(V^\varphi_{\alpha,\beta}:A^p_\omega\to H^q\) is compact if and only if
\[
        \lim_{|z|\to1^-}
        |G_\varphi(z)|\widehat\omega(z)^{-1/p}\delta(z)^{\beta+1/q-1/p}=0.
        \tag{3.7}
\]
\end{proposition}

\begin{proof}
By Proposition~\ref{prop:dictionary}, (3.7) is equivalent to
\[
        \lim_{|z|\to1^-}
        |G_\varphi(z)|\widehat\nu(z)^{-1/p}
        \delta(z)^{\alpha+1/q-1/p}=0.
        \tag{3.8}
\]
For sufficiency, let \(f_j\) be bounded in \(A^p_\omega\) and converge to zero
uniformly on compact subsets, and set \(h_j=\IRL_{\beta-\alpha}f_j\).  Then
\(h_j\) is bounded in \(A^p_\nu\) and converges to zero uniformly on compact
subsets.  Given \(\varepsilon>0\), choose
\(R<1\) such that the shifted Case A quantity in (3.8) is less than \(\varepsilon\) on
\(\D\setminus R\D\).  The estimate used in
Theorem~\ref{thm:caseA-boundedness}, with
Lemma~\ref{lem:caseA-fractional-embeddings}\textup{(i)} or \textup{(ii)} as
appropriate, gives
\[
 \|\widetilde V^\varphi_{\alpha,\beta}f_j\|_{H^q}^p
 \lesssim
 \int_{R\D}|h_j|^p|G_\varphi|^p\delta^{p/q+p\alpha-2}\dA
 +\varepsilon^p\int_{\D}|h_j|^p\frac{\widehat\nu}{\delta}\dA,
\]
where the first term tends to zero by local uniform convergence and the
second is \(O(\varepsilon^p)\) by
Lemma~\ref{lem:background-tools}\textup{(vi)}.  Hence the limsup of the
left-hand side as \(j\to\infty\) is at most \(C\varepsilon^p\).  Letting first
\(j\to\infty\) and then
\(\varepsilon\to0\) proves compactness of the modified model, and
Proposition~\ref{prop:model-reduction} transfers compactness to the original
operator.

For necessity, compactness of \(V\) is equivalent to the sequential condition
for the area map in Proposition~\ref{prop:g-function-reduction}.
Use the same normalized kernels \(h_a\) as above.  Then \(h_a\to0\) uniformly
on compact subsets,
so
\[
        \|A_{\alpha,\varphi}h_a\|_{L^q}\to0,
        \qquad |a|\to1^-.
\]
The local disc estimate in the necessity proof of
Theorem~\ref{thm:caseA-boundedness} gives
\[
        |G_\varphi(a)|\widehat\nu(a)^{-1/p}\delta(a)^{\alpha+1/q-1/p}
        \lesssim
        \|A_{\alpha,\varphi}h_a\|_{L^q}.
\]
Letting \(|a|\to1\) proves (3.8), and hence (3.7).
\end{proof}

\subsection{Case B: \texorpdfstring{\(2<p=q<\infty\)}{2<p=q}}

The following lifting principle recovers the Case B Carleson measure from area
boundedness.  Put \(\kappa_\nu=\widehat\nu/\delta\), and, for measurable
\(E\subset\D\), write
\(\kappa_\nu(E)=\int_E\kappa_\nu(z)\dA(z)\).
A radial weight \(v\) is called regular if it is continuous and
\(\widehat v(r)\asymp(1-r)v(r)\).  For \(\nu\in\calD\), the weight
\(\kappa_\nu\) is regular, \(\widehat\kappa_\nu\asymp\widehat\nu\), and
\[
        \int_\D |h(z)|^p\kappa_\nu(z)\dA(z)
        \asymp
        \|h\|_{A^p_\nu}^p,
        \qquad h\in H(\D).
\]
These assertions follow from Lemma~\ref{lem:background-tools}\textup{(vi)}.
If \(I\subset\T\) is an arc, \(\wedge(I)\) denotes the tent over \(I\); replacing
\(\wedge(I)\) by the Carleson square \(S(I)\) only changes constants, again by
Lemma~\ref{lem:background-tools}\textup{(vi)}.

\begin{lemma}\label{lem:caseB-Hardy-Carleson}
Let \(2<p<\infty\), \(r=p/(p-2)\), and let \(\tau\) be a positive finite measure
on \(\D\).  Define
\[
        \widetilde\tau(\xi)=
        \int_{\Gamma(\xi)}\frac{d\tau(z)}{\delta(z)},
        \qquad \xi\in\T .
\]
Then
\[
        \|\widetilde\tau\|_{L^{p/2}(\T)}
        \asymp
        \|\Id:H^r\to L^1(\tau)\|.
\]
\end{lemma}

\begin{proof}
Apply Pau's tent-transform theorem \cite[Theorem~E]{Pau2016} to the embedding
\(\Id:H^r\to L^1(\tau)\).  Since \(r'=r/(r-1)=p/2\), that theorem yields
the displayed \(L^{p/2}\)-norm of \(\widetilde\tau\); replacing
\(1-|z|\) there by \(\delta(z)\) only changes constants.  This specialization
is also used in \cite[Proposition~2.13]{DuanWangWang2021}.
\end{proof}

\begin{lemma}\label{lem:caseB-lifting-proposition}
Let \(2<p<\infty\), \(\nu\in\calD\), and let \(\sigma\) be a positive locally
finite Borel measure on \(\D\).  Fix a pseudo-hyperbolic radius \(0<\rho<1\),
and set
\[
        d\Lambda_\sigma(z)=
        \left(
        \frac{\sigma(D(z,\rho))}{\kappa_\nu(D(z,\rho))}
        \right)^{\frac{p}{p-2}}
        \kappa_\nu(z)\dA(z).
\]
Then \(\Lambda_\sigma\) is a Carleson measure if and only if
\[
        \sup_{\|h\|_{A^p_\nu}\le1}
        \sup_{I\subset\T}
        \frac{\displaystyle\int_{\wedge(I)} |h(z)|^2\,d\sigma(z)}
             {|I|^{(p-2)/p}}
        <\infty.
\]
Moreover, \(\Lambda_\sigma\) is a vanishing Carleson measure if and only if
\[
        \lim_{\eta\to0^+}
        \sup_{0<|I|\le\eta}
        \sup_{\|h\|_{A^p_\nu}\le1}
        \frac{\displaystyle\int_{\wedge(I)} |h(z)|^2\,d\sigma(z)}
             {|I|^{(p-2)/p}}
        =0.
\]
\end{lemma}

The proof is adapted from the argument for the \(s=1\), \(\iota=2/p\) case of
\cite[Proposition~2.10]{DuanWangWang2021}.  We include the details because the
uniform unit-ball formulation used here is not stated there in precisely this
normalization.

\begin{proof}
Put \(\iota=2/p\), \(r=(1-\iota)^{-1}=p/(p-2)\), and
\[
Q_\sigma(I)=\sup_{\|h\|_{A^p_\nu}\le1}
\int_{\wedge(I)}|h(z)|^{p\iota}\,d\sigma(z).
\]
Since \(\nu\in\calD\), the weight \(\kappa_\nu=\widehat\nu/\delta\) is regular,
and, for every fixed \(0<s<1\),
\(\kappa_\nu(D(z,s))\asymp\kappa_\nu(z)\delta(z)^2\).
Both facts follow from Lemma~\ref{lem:background-tools}\textup{(vi)}.  For fixed radii
\(s_1,s_2\in(0,1)\), every arc \(I\) has a concentric enlargement \(I^*\),
with \(|I^*|\asymp|I|\), such that
\(D(a,s_1)\cap\wedge(I)\ne\varnothing\) implies
\(D(a,s_2)\subset\wedge(I^*)\).
For large arcs we take \(I^*=\T\).  This is the tent-enlargement statement in
Lemma~\ref{lem:background-tools}\textup{(vi)}; see also
\cite[Lemma~2.11]{DuanWangWang2021}.

Suppose first that \(\Lambda_\sigma\) is Carleson.  The sub-mean estimate and
regularity in Lemma~\ref{lem:background-tools}\textup{(iii), (vi)}, together
with \(p\iota=2\), give
\[
        |h(z)|^{p\iota}
        \lesssim
        \frac1{\kappa_\nu(D(z,\rho))}
        \int_{D(z,\rho)}
        |h(\zeta)|^{p\iota}\kappa_\nu(\zeta)\dA(\zeta).
\]
Here \(\zeta\in D(z,\rho)\) if and only if \(z\in D(\zeta,\rho)\), and for
such pairs
\[
        \kappa_\nu(D(z,\rho))
        \asymp \kappa_\nu(D(\zeta,\rho))
\]
by Lemma~\ref{lem:background-tools}\textup{(vi)}.  These comparisons, together
with the fixed tent enlargement, justify the change from \(z\)-centered to
\(\zeta\)-centered discs in the first line below.
Tonelli's theorem from Lemma~\ref{lem:background-tools}\textup{(ii)} and the
tent enlargement in part \textup{(vi)} therefore give
\[
\begin{aligned}
        \int_{\wedge(I)}|h(z)|^{p\iota}\,d\sigma(z)
        &\lesssim
        \int_{\wedge(I^*)}|h(\zeta)|^{p\iota}
        \frac{\sigma(D(\zeta,\rho))}
             {\kappa_\nu(D(\zeta,\rho))}
        \kappa_\nu(\zeta)\dA(\zeta)                                  \\
        &\le
        \left(\int_\D |h(\zeta)|^p
        \kappa_\nu(\zeta)\dA(\zeta)\right)^\iota
        \Lambda_\sigma(\wedge(I^*))^{1/r}                             \\
        &\lesssim
        \|h\|_{A^p_\nu}^{2}|I|^{1/r}.
\end{aligned}
\]
Thus \(Q_\sigma(I)\lesssim|I|^{1/r}\).  If \(\Lambda_\sigma\) is
vanishing, the same estimate yields
\[
        \frac{Q_\sigma(I)}{|I|^{1/r}}
        \lesssim
        \left(
        \frac{\Lambda_\sigma(\wedge(I^*))}{|I^*|}
        \right)^{1/r}
        \longrightarrow0.
\]

Conversely, let \(\{a_k\}\) be a sufficiently fine \(\eta\)-lattice as in
Lemma~\ref{lem:background-tools}\textup{(vi)}, and choose
\(R\in(0,1)\) so that \((\rho+\eta)/(1+\rho\eta)<R\).
The atomic decomposition for \(\calD\)-weighted Bergman spaces
\cite[Lemma~2.12]{DuanWangWang2021} gives, for sufficiently large \(M\),
\[
        f_k(z)=
        \frac{\delta(a_k)^{M-1/p}\widehat\nu(a_k)^{-1/p}}
             {(1-\overline{a_k}z)^M},
        \qquad
        \left\|\sum_k\lambda_k f_k\right\|_{A^p_\nu}
        \lesssim\|\lambda\|_{\ell^p}.
\]
On \(D(a_k,R)\), the same atoms satisfy
\[
        |f_k(z)|^p
        \asymp
        \frac1{\widehat\nu(a_k)\delta(a_k)}
        \asymp
        \frac1{\kappa_\nu(D(a_k,R))}.
\]
The last comparison is Lemma~\ref{lem:background-tools}\textup{(vi)}.

For an arc \(J\), let
\(\mathcal K_R(J)=\{k:D(a_k,R)\subset\wedge(J)\}\); for a finitely supported
\(\lambda\) on this set, put
\(F_t(z)=\sum_k\lambda_k r_k(t)f_k(z)\), where \(r_k\) are the Rademacher
functions.  The atomic estimate is uniform in \(t\).  Hence the definition
of \(Q_\sigma(J)\), Tonelli's theorem from
Lemma~\ref{lem:background-tools}\textup{(ii)}, Rademacher orthogonality
\cite[Lemma~2.8, with exponent~2]{DuanWangWang2021}, and the local estimate
above give
\[
\begin{aligned}
        Q_\sigma(J)\|\lambda\|_{\ell^p}^{2}
        &\gtrsim
        \int_0^1\int_{\wedge(J)}|F_t(z)|^2\,d\sigma(z)\,dt          \\
        &\gtrsim
        \sum_{k\in\mathcal K_R(J)}
        |\lambda_k|^2
        \frac{\sigma(D(a_k,R))}
             {\kappa_\nu(D(a_k,R))^\iota}.
\end{aligned}
\]
The sequence duality formula in Lemma~\ref{lem:background-tools}\textup{(i)}
therefore yields
\[
        \sum_{k\in\mathcal K_R(J)}
        \left(
        \frac{\sigma(D(a_k,R))}
             {\kappa_\nu(D(a_k,R))}
        \right)^r
        \kappa_\nu(D(a_k,R))
        \lesssim Q_\sigma(J)^r.
\]

It remains to include lattice discs meeting the boundary of the target tent.
If \(z\in D(a_k,\eta)\), direct substitution into the definition of \(\varrho\)
gives \(\varrho(w,a_k)\le(\rho+\eta)/(1+\rho\eta)<R\) for
\(w\in D(z,\rho)\).  Thus \(D(z,\rho)\subset D(a_k,R)\), and the regularity
assertion in Lemma~\ref{lem:background-tools}\textup{(vi)} gives
\[
\kappa_\nu(D(z,\rho))\asymp\widehat\nu(z)\delta(z)
\asymp\widehat\nu(a_k)\delta(a_k)\asymp\kappa_\nu(D(a_k,R)),
\qquad
\frac{\sigma(D(z,\rho))}{\kappa_\nu(D(z,\rho))}
\lesssim\frac{\sigma(D(a_k,R))}{\kappa_\nu(D(a_k,R))}.
\]
Choose \(I^*\) so that
\(D(a_k,\eta)\cap\wedge(I)\ne\varnothing\) implies
\(D(a_k,R)\subset\wedge(I^*)\).
The lattice covering in Lemma~\ref{lem:background-tools}\textup{(vi)} and the
last two estimates imply
\[
\begin{aligned}
        \Lambda_\sigma(\wedge(I))
        &\lesssim
        \sum_{k\in\mathcal K_R(I^*)}
        \left(
        \frac{\sigma(D(a_k,R))}
             {\kappa_\nu(D(a_k,R))}
        \right)^r
        \kappa_\nu(D(a_k,R))                                      \\
        &\lesssim Q_\sigma(I^*)^r.
\end{aligned}
\]
If \(Q_\sigma(J)\lesssim|J|^{1/r}\), this proves that
\(\Lambda_\sigma\) is Carleson.  If
\(Q_\sigma(J)/|J|^{1/r}\to0\), then
\[
        \frac{\Lambda_\sigma(\wedge(I))}{|I|}
        \lesssim
        \left(
        \frac{Q_\sigma(I^*)}{|I^*|^{1/r}}
        \right)^r
        \longrightarrow0,
\]
so \(\Lambda_\sigma\) is a vanishing Carleson measure.
\end{proof}

\begin{lemma}\label{lem:caseB-area-criterion}
Let \(2<p<\infty\), \(\nu\in\calD\), and
\[
        u(z)=G(z)\delta(z)^\tau,
        \qquad G\in H(\D),\quad \tau\in\mathbb R.
\]
Define
\[
        A_u h(\xi)=
        \left(\int_{\Gamma(\xi)}|h(z)|^2|u(z)|^2\dA(z)\right)^{1/2}.
\]
Then \(A_u:A^p_\nu\to L^p(\T)\) is bounded if and only if
\[
        d\mu_{u,B}(z)=
        |u(z)|^{\frac{2p}{p-2}}
        \frac{\delta(z)^{\frac{p+2}{p-2}}}
             {\widehat\nu(z)^{\frac{2}{p-2}}}\dA(z)
\]
is a Carleson measure.  Moreover, a bounded \(A_u\) satisfies the sequential
condition in Proposition~\ref{prop:g-function-reduction} if and only if
\(\mu_{u,B}\) is a vanishing Carleson measure.  For this positive
sublinear map, that condition is also equivalent to relative compactness of the
image of the unit ball, as proved below.
\end{lemma}

\begin{proof}
Set \(d\sigma(z)=|u(z)|^2\delta(z)\dA(z)\).  Then
\[
A_u h(\xi)^2=
\int_{\Gamma(\xi)}\frac{|h(z)|^2\,d\sigma(z)}{\delta(z)}.
\]

Assume first that \(\mu_{u,B}\) is Carleson.  Let \(F\in H^{p/(p-2)}\).  By
Lemma~\ref{lem:background-tools}\textup{(i)} and the identity
\[
        \left(
        |u|^{\frac{2p}{p-2}}
        \frac{\delta^{\frac{p+2}{p-2}}}
             {\widehat\nu^{\frac{2}{p-2}}}
        \right)^{\frac{p-2}{p}}
        \left(\frac{\widehat\nu}{\delta}\right)^{2/p}
        =|u|^2\delta,
\]
we get
\[
\begin{aligned}
        \int_\D |F(z)| |h(z)|^2\,d\sigma(z)
        &\le
        \left(\int_\D |F(z)|^{p/(p-2)}d\mu_{u,B}(z)\right)^{(p-2)/p}   \\
        &\quad\times
        \left(\int_\D |h(z)|^p\kappa_\nu(z)\dA(z)\right)^{2/p}        \\
        &\lesssim
        \|F\|_{H^{p/(p-2)}}\|h\|_{A^p_\nu}^2.
\end{aligned}
\]
Here the last inequality uses the Hardy--Carleson embedding
\cite[Lemma~21]{FangGuoHouZhuBergmanHardy} and the norm equivalence in
Lemma~\ref{lem:background-tools}\textup{(vi)}.
Taking \(F\equiv1\) in the preceding estimate shows that the measure
\(d\tau_h=|h|^2d\sigma\) is finite.  Lemma~\ref{lem:caseB-Hardy-Carleson},
applied to \(\tau_h\), gives
\(\|A_u h\|_{L^p}^2\lesssim\|h\|_{A^p_\nu}^2\), and boundedness follows.

Conversely, assume that \(A_u:A^p_\nu\to L^p(\T)\) is bounded.  For an arc \(I\) and
\(h\in A^p_\nu\), the identity above gives
\[
        \int_{\wedge(I)} |h(z)|^2\,d\sigma(z)
        \lesssim
        \int_I A_u h(\xi)^2|d\xi|.
\]
Lemma~\ref{lem:background-tools}\textup{(i)} on \(I\) yields
\[
        \int_{\wedge(I)} |h|^2\,d\sigma
        \lesssim
        |I|^{(p-2)/p}\|A_u h\|_{L^p}^2
        \lesssim
        |I|^{(p-2)/p}\|h\|_{A^p_\nu}^2.
\]
By Lemma~\ref{lem:caseB-lifting-proposition}, the lifted measure
\(\Lambda_\sigma\) defined there is Carleson.  To compare \(\mu_{u,B}\) with
\(\Lambda_\sigma\), use \(u=G\delta^\tau\).  Lemma~\ref{lem:radialized-symbol-local}
with exponent \(2\) gives
\[
        |u(z)|^2\lesssim
        \delta(z)^{-2}\int_{D(z,\rho)}|u(\zeta)|^2\dA(\zeta).
\]
Hence, using \(d\sigma=|u|^2\delta\dA\) and
\(\kappa_\nu(D(z,\rho))\asymp\widehat\nu(z)\delta(z)\) from
Lemma~\ref{lem:background-tools}\textup{(vi)},
\[
|u(z)|^2\frac{\delta(z)^2}{\widehat\nu(z)}
\lesssim\frac{\sigma(D(z,\rho))}{\kappa_\nu(D(z,\rho))}.
\]
Raising the local comparison to the power \(p/(p-2)\) and multiplying by
\(\kappa_\nu(z)\dA(z)=\widehat\nu(z)\delta(z)^{-1}\dA(z)\) gives
\(d\mu_{u,B}\lesssim d\Lambda_\sigma\).
Therefore \(\mu_{u,B}\) is Carleson.

We first relate the sequential condition to relative compactness for this
positive sublinear map.  Pointwise in \(\xi\),
\[
        |A_u h_1(\xi)-A_u h_2(\xi)|\le A_u(h_1-h_2)(\xi),
        \tag{3.9}
\]
so every bounded \(A_u\) satisfying this condition maps the unit ball into a
relatively compact subset of \(L^p(\T)\).  Indeed,
let \(\{h_j\}\) be any sequence in the unit ball of \(A^p_\nu\).  By
Lemma~\ref{lem:background-tools}\textup{(iii), (iv)}, a
subsequence, still denoted by \(h_j\), converges locally uniformly to some analytic
function \(h\).  Fatou's lemma in Lemma~\ref{lem:background-tools}\textup{(ii)}
gives \(h\in A^p_\nu\), and the sequence \(\{h_j-h\}\) is bounded in
\(A^p_\nu\) and converges to zero locally uniformly.  The sequential condition
then gives \(\|A_u(h_j-h)\|_{L^p}\to0\), and \((3.9)\) implies
\(A_u h_j\to A_u h\) in \(L^p(\T)\).

Suppose now that \(\mu_{u,B}\) is a vanishing Carleson measure and let \(h_j\)
be bounded in \(A^p_\nu\) with \(h_j\to0\) uniformly on compact subsets.  For
fixed \(R\in(0,1)\), let \(A_u^{\mathrm{in},R}\) and
\(A_{u,R}^{\mathrm{out}}\) denote the area maps obtained by restricting the
inner integral to \(R\D\) and \(\D\setminus R\D\), respectively.  Pointwise
subadditivity gives
\[
        A_u h_j\le A_u^{\mathrm{in},R}h_j+A_{u,R}^{\mathrm{out}}h_j.
\]
For the inner term, local uniform convergence gives
\(\sup_{z\in R\D}|h_j(z)|\to0\).  Moreover, for every \(\xi\in\T\),
\[
\begin{aligned}
        A_u^{\mathrm{in},R}h_j(\xi)^2
        &=
        \int_{\Gamma(\xi)\cap R\D}
        |h_j(z)|^2|u(z)|^2\,\dA(z)\\
        &\le
        \sup_{z\in R\D}|h_j(z)|^2
        \int_{R\D}|u(z)|^2\,\dA(z).
\end{aligned}
\]
Since \(u\) is bounded on \(R\overline\D\), the last integral is finite,
and hence
\(\|A_u^{\mathrm{in},R}h_j\|_{L^p(\T)}\to0\).
The outer term is estimated by the sufficiency argument above after replacing
\(d\sigma\) and \(d\mu_{u,B}\) by their restrictions to
\(\D\setminus R\D\).  The Carleson norm of the latter restricted measure tends
to zero as \(R\to1^-\): cover an arc longer than \(1-R\) by arcs whose lengths
lie between \(1-R\) and \(2(1-R)\), and apply the vanishing Carleson condition
to the shorter arcs.  Hence the outer term is uniformly small, and
\(A_u h_j\to0\) in \(L^p(\T)\).

Finally, assume that the image of the unit ball of \(A^p_\nu\) under \(A_u\)
is relatively compact in \(L^p(\T)\).  The \(L^p\)-norms on its compact
closure are uniformly absolutely continuous: take a finite
\(\varepsilon\)-net, use absolute continuity of the integral for its finitely many
members, and control the approximation error by
Lemma~\ref{lem:background-tools}\textup{(i)}.  Thus for every \(\varepsilon>0\)
there exists \(\eta>0\)
such that, whenever \(|E|<\eta\),
\[
        \sup_{\|h\|_{A^p_\nu}\le1}
        \int_E |A_u h(\xi)|^p|d\xi|<\varepsilon^p.
\]
Taking \(E=I\) and using Lemma~\ref{lem:background-tools}\textup{(i)} as above,
we obtain
\[
        \sup_{\|h\|_{A^p_\nu}\le1}
        \frac{\displaystyle\int_{\wedge(I)}|h|^2\,d\sigma}{|I|^{(p-2)/p}}
        \lesssim \varepsilon^2,
        \qquad |I|<\eta.
\]
Thus the uniform vanishing estimate holds.  Lemma~\ref{lem:caseB-lifting-proposition}
shows that \(\Lambda_\sigma\) is a vanishing Carleson measure, and the domination
\(d\mu_{u,B}\lesssim d\Lambda_\sigma\) proves that \(\mu_{u,B}\) is vanishing.
\end{proof}

\begin{theorem}\label{thm:caseB-boundedness}
Let \(2<p=q<\infty\), \(\omega\in\calD\), \(\varphi\in H(\D)\),
\((\alpha,\beta)\in\calP\), and \(\alpha,\beta>0\).  Assume \((H)\).  Then
\(V^\varphi_{\alpha,\beta}:A^p_\omega\to H^p\) is bounded if and only if the
measure \(\mu_{\varphi,B}\) in Theorem~\ref{thm:boundedness-main}\textup{(B)} is a
Carleson measure.
\end{theorem}

\begin{proof}
By Proposition~\ref{prop:g-function-reduction}, the problem is equivalent to
boundedness of \(A_{u_\varphi}:A^p_\nu\to L^p(\T)\), with
\(u_\varphi=G_\varphi\delta^{\alpha-1}\).  Lemma~\ref{lem:caseB-area-criterion}
gives its Carleson-measure criterion, which
Proposition~\ref{prop:dictionary} transforms into
\(\mu_{\varphi,B}\).
\end{proof}

\begin{proposition}\label{prop:caseB-compactness}
Under the assumptions of Theorem~\ref{thm:caseB-boundedness},
\(V^\varphi_{\alpha,\beta}:A^p_\omega\to H^p\) is compact if and only if
\(\mu_{\varphi,B}\) is a vanishing Carleson measure.
\end{proposition}

\begin{proof}
By Proposition~\ref{prop:g-function-reduction}, compactness of
\(V^\varphi_{\alpha,\beta}:A^p_\omega\to H^p\) is equivalent to the sequential
condition for \(A_{u_\varphi}:A^p_\nu\to L^p(\T)\), where
\(u_\varphi=G_\varphi\delta^{\alpha-1}\).  Lemma~\ref{lem:caseB-area-criterion}
identifies this with the vanishing form of the same measure criterion, and
Proposition~\ref{prop:dictionary} transforms that measure into
\(\mu_{\varphi,B}\).
\end{proof}

\subsection{Case C: the tent-integral range}

Assume
\[
        0<q<\infty,
        \qquad
        p>\max\{2,q\}.
\]
Set
\[
        s=\frac{2p}{p-2},
        \qquad
        r=\frac{pq}{p-q}.
        \tag{3.11}
\]

Fix a pseudo-hyperbolic lattice
\[
        Z=\{a_k\}_{k\ge1}\subset\D.
\]
For a sequence \(c=\{c_k\}\) and \(0<a<\infty\), \(0<b\le\infty\), we write
\(c\in T^a_b(Z)\) if
\[
        \|c\|_{T^a_b(Z)}^a
        =
        \int_\T
        \left(\sum_{a_k\in\Gamma(\xi)}|c_k|^b\right)^{a/b}\dxi<\infty,
\qquad 0<b<\infty.
\]
For \(b=\infty\) we define
\[
        \|c\|_{T^a_\infty(Z)}^a
        =
        \int_\T
        \left(\sup_{a_k\in\Gamma(\xi)}|c_k|\right)^a\dxi .
\]
Different fixed apertures yield equivalent discrete and continuous tent
quasi-norms.  We use the aperture invariance in
Lemma~\ref{lem:background-tools}\textup{(vi)} without further comment.

\begin{lemma}\label{lem:caseC-continuous-discrete}
Let \(0<q<\infty\), \(p>\max\{2,q\}\), \(\nu\in\calD\), and set
\[
        s=\frac{2p}{p-2},
        \qquad
        r=\frac{pq}{p-q}.
\]
Let
\[
        u(z)=G(z)\delta(z)^\tau,
        \qquad G\in H(\D),\quad \tau\in\mathbb R.
\]
Choose fixed pseudo-hyperbolic radii
\[
        0<\rho_0<\rho_1<\rho_2<1
\]
such that the discs \(D(a_k,\rho_0)\) cover \(\D\) and the enlarged discs
\(D(a_k,\rho_2)\) have bounded overlap, as in
Lemma~\ref{lem:background-tools}\textup{(vi)}.  Put
\[
        b_k
        =
        \frac{\delta(a_k)}{\widehat\nu(a_k)^{1/p}}
        \sup_{z\in D(a_k,\rho_0)}|u(z)|.
        \tag*{(3.12)}
\]
Then
\[
        \left\|
        \left(
        \int_{\Gamma(\xi)}
        |u(z)|^s
        \left(\frac{\delta(z)^2}{\widehat\nu(z)}\right)^{\frac{2}{p-2}}
        \dA(z)
        \right)^{1/s}
        \right\|_{L^r(\T)}
        \asymp
        \|\{b_k\}\|_{T^r_s(Z)} .
        \tag*{(3.13)}
\]
\end{lemma}

\begin{proof}
Although \(u\) need not be analytic, its representation \(u=G\delta^\tau\) and
Lemma~\ref{lem:background-tools}\textup{(iii), (vi)} give
subharmonicity of \(|G|^a\) for every \(a>0\) and local comparability of
\(\delta\); hence
\[
        |u(z)|^a
        \lesssim
        \frac1{\delta(z)^2}
        \int_{D(z,\rho)}|u(\zeta)|^a\dA(\zeta)
        \tag*{(3.14)}
\]
for every fixed \(0<\rho<1\).  The implicit constant depends on \(a,\rho,\tau\).

For the upper estimate in (3.13), cover \(\Gamma(\xi)\) by those discs
\(D(a_k,\rho_0)\) meeting a slightly enlarged cone.  On such a disc,
\(\delta(z)\asymp\delta(a_k)\) and
\(\widehat\nu(z)\asymp\widehat\nu(a_k)\).  Hence
\[
\begin{aligned}
        &\int_{\Gamma(\xi)}
        |u(z)|^s
        \left(\frac{\delta(z)^2}{\widehat\nu(z)}\right)^{\frac{2}{p-2}}
        \dA(z)                                                   \\
        &\qquad\lesssim
        \sum_{a_k\in\widetilde\Gamma(\xi)}
        \left(
        \frac{\delta(a_k)}{\widehat\nu(a_k)^{1/p}}
        \sup_{D(a_k,\rho_0)}|u|
        \right)^s
        =
        \sum_{a_k\in\widetilde\Gamma(\xi)}|b_k|^s .
\end{aligned}
\]
Lemma~\ref{lem:background-tools}\textup{(vi)} gives the upper bound for the
\(L^r\)-norm after restoring the original aperture.

For the reverse estimate, choose \(w_k\in D(a_k,\rho_0)\) with
\(|u(w_k)|\ge \frac12\sup_{D(a_k,\rho_0)}|u|\).  Applying (3.14) with exponent
\(s\) to a small disc centered at \(w_k\), contained in \(D(a_k,\rho_1)\), gives
\[
        \delta(a_k)^2\sup_{D(a_k,\rho_0)}|u|^s
        \lesssim
        \int_{D(a_k,\rho_1)}|u(z)|^s\dA(z).
\]
Using again the local comparability in
Lemma~\ref{lem:background-tools}\textup{(vi)},
\[
        |b_k|^s
        \lesssim
        \int_{D(a_k,\rho_1)}
        |u(z)|^s
        \left(\frac{\delta(z)^2}{\widehat\nu(z)}\right)^{\frac{2}{p-2}}
        \dA(z).
        \tag*{(3.15)}
\]
Summing (3.15) over \(a_k\in\Gamma(\xi)\), using finite overlap and enlarging the
cone once more, yields
\[
        \sum_{a_k\in\Gamma(\xi)}|b_k|^s
        \lesssim
        \int_{\widetilde\Gamma(\xi)}
        |u(z)|^s
        \left(\frac{\delta(z)^2}{\widehat\nu(z)}\right)^{\frac{2}{p-2}}
        \dA(z).
\]
Taking \(1/s\)-powers, then \(L^r\)-quasi-norms, and using
Lemma~\ref{lem:background-tools}\textup{(vi)} proves the lower bound in \((3.13)\).
\end{proof}

\begin{lemma}\label{lem:caseC-random-testing}
Let the assumptions and notation of Lemma~\ref{lem:caseC-continuous-discrete} hold.
If
\[
        A_u h(\xi)=
        \left(
        \int_{\Gamma(\xi)}|h(z)|^2|u(z)|^2\dA(z)
        \right)^{1/2}
\]
is bounded from \(A^p_\nu\) to \(L^q(\T)\), then, for every finitely supported
sequence \(\lambda=\{\lambda_k\}\),
\[
        \|\{\lambda_k b_k\}\|_{T^q_2(Z)}
        \lesssim
        \|A_u\|_{A^p_\nu\to L^q(\T)}
        \|\lambda\|_{T^p_p(Z)}.
        \tag*{(3.16)}
\]
\end{lemma}

\begin{proof}
We use the randomized atomic testing argument from
\cite[Proposition~2.18]{DuanWangWang2021}, including the details needed to fix
the normalization and derive the lower bound from a local integral rather than
a pointwise supremum.  Choose \(M\) large and set
\[
        e_k(z)=
        \frac{\delta(a_k)^M}
             {\widehat\nu(a_k)^{1/p}(1-\overline{a_k}z)^M}.
        \tag*{(3.17)}
\]
The \(\calD\)-weighted atomic decomposition of Pel\'aez--R\"atty\"a--Sierra
\cite[Theorem~1]{PelaezRattyaSierra2021}, in the formulation
\cite[Lemma~2.12]{DuanWangWang2021}, gives, for every finitely supported
\(\lambda\),
\[
        h_{\lambda,t}(z)=\sum_k r_k(t)\lambda_k e_k(z),
        \qquad 0<t<1,
\]
where \(\{r_k\}\) are the Rademacher functions, and
\[
        \|h_{\lambda,t}\|_{A^p_\nu}
        \lesssim
        \|\lambda\|_{T^p_p(Z)}
        \quad\hbox{uniformly in }t.
        \tag*{(3.18)}
\]
Indeed, \(\|\lambda\|_{T^p_p(Z)}^p\asymp\sum_k|\lambda_k|^p\delta(a_k)\), and
\(e_k\) is the normalized atom in that decomposition multiplied by
\(\delta(a_k)^{1/p}\).
On each lattice disc,
\[
        |e_k(z)|\asymp \widehat\nu(a_k)^{-1/p},
        \qquad z\in D(a_k,\rho_1).
        \tag*{(3.19)}
\]

Fix \(\xi\).  We apply the Hilbert-valued Khinchine inequality
\cite[Lemmas~2.8--2.9]{DuanWangWang2021}.  The Hilbert space is
\(L^2(\Gamma(\xi),|u|^2dA)\).  Tonelli's theorem, the boundedness of \(A_u\),
\((3.18)\), and the aperture invariance in
Lemma~\ref{lem:background-tools}\textup{(ii), (vi)} then give
\[
\begin{aligned}
        &\left\|
        \left(
        \int_{\Gamma(\xi)}
        \sum_k |\lambda_k|^2 |e_k(z)|^2 |u(z)|^2\dA(z)
        \right)^{1/2}
        \right\|_{L^q(\T)}                                      \\
        &\qquad\lesssim
        \left(\int_0^1\|A_u h_{\lambda,t}\|_{L^q(\T)}^qdt\right)^{1/q}
        \lesssim
        \|A_u\|_{A^p_\nu\to L^q}
        \|\lambda\|_{T^p_p(Z)} .
\end{aligned}
\tag*{(3.20)}
\]
Because the integrand in \((3.20)\) is a sum of nonnegative terms, for each
\(k\) with \(a_k\in\Gamma(\xi)\) we may retain only the contribution of
\(D(a_k,\rho_1)\) to the \(k\)-th summand, after a fixed aperture enlargement.
No disjointization is required.  Using \((3.19)\), and then \((3.14)\) with
exponent \(2\), gives
\[
\begin{aligned}
        \widehat\nu(a_k)^{-2/p}
        \int_{D(a_k,\rho_1)} |u(z)|^2\dA(z)
        &\gtrsim
        \widehat\nu(a_k)^{-2/p}
        \delta(a_k)^2
        \sup_{D(a_k,\rho_0)}|u|^2                         \\
        &= |b_k|^2 .
\end{aligned}
\tag*{(3.21)}
\]
Thus the left-hand side of (3.20) dominates
\(\|\{\lambda_k b_k\}\|_{T^q_2(Z)}\), with the fixed aperture change justified by
Lemma~\ref{lem:background-tools}\textup{(vi)}.  This proves
(3.16).
\end{proof}

\begin{lemma}\label{lem:caseC-multiplier}
Let \(0<q<\infty\), \(p>\max\{2,q\}\), and define \(s,r\) by
\[
        \frac1q=\frac1p+\frac1r,
        \qquad
        \frac12=\frac1p+\frac1s.
\]
For a sequence \(b=\{b_k\}\) on a pseudo-hyperbolic lattice \(Z\), multiplication by
\(b\) is bounded from \(T^p_p(Z)\) to \(T^q_2(Z)\) if and only if
\(b\in T^r_s(Z)\).  Moreover,
\[
        \|b\|_{T^r_s(Z)}
        \asymp
        \sup_{\lambda\ne0}
        \frac{\|\{\lambda_k b_k\}\|_{T^q_2(Z)}}
             {\|\lambda\|_{T^p_p(Z)}}.
        \tag*{(3.22)}
\]
\end{lemma}

\begin{proof}
Denote the supremum in \textup{(3.22)} by \(\mathfrak M_b\).  If
\(b\in T^r_s(Z)\), applying generalized H\"older's inequality first on
\(\Gamma(\xi)\cap Z\), with \(1/2=1/p+1/s\), and then on \(\T\), with
\(1/q=1/p+1/r\), gives
\[
\begin{aligned}
        \|\{\lambda_kb_k\}\|_{T^q_2(Z)}
        &\le
        \left\|
        \left(\sum_{a_k\in\Gamma(\xi)}|\lambda_k|^p\right)^{1/p}
        \left(\sum_{a_k\in\Gamma(\xi)}|b_k|^s\right)^{1/s}
        \right\|_{L^q(\T)}                                      \\
        &\le
        \|\lambda\|_{T^p_p(Z)}\|b\|_{T^r_s(Z)} .
\end{aligned}
\]
Thus \(\mathfrak M_b\lesssim\|b\|_{T^r_s(Z)}\).

Conversely, suppose that \(\mathfrak M_b<\infty\).  We give the power
argument underlying \cite[Proposition~2.18]{DuanWangWang2021}, using the
duality and factorization theorems
\cite[Lemmas~2.14--2.15]{DuanWangWang2021}.  Choose
\[
        \theta>\max\left\{\frac1q,\frac12\right\}.
\]
Then all exponents below exceed one, and
\[
        \frac1{(r\theta)'}
        =
        \frac1{(q\theta)'}+\frac1{p\theta},
        \qquad
        \frac1{(s\theta)'}
        =
        \frac1{(2\theta)'}+\frac1{p\theta}.
\]
Since \(r>q\) and \(s>2\), factorization gives
\[
        T^{(r\theta)'}_{(s\theta)'}(Z)
        =
        T^{(q\theta)'}_{(2\theta)'}(Z)
        \cdot T^{p\theta}_{p\theta}(Z),
\]
with the norm in the space on the left equivalent to the infimum of the
products of the two factor norms.

Let \(E\subset\N\) be finite and put \(b_k^E=b_k\) for \(k\in E\), and
\(b_k^E=0\) otherwise.  Its multiplier norm is at most \(\mathfrak M_b\).
By solidity, it suffices in the duality pairing to take a nonnegative finitely
supported \(e\in T^{(r\theta)'}_{(s\theta)'}(Z)\).  Factor
\[
        e_k=x_kv_k,
        \qquad
        x\in T^{(q\theta)'}_{(2\theta)'}(Z),
        \quad
        v\in T^{p\theta}_{p\theta}(Z),
\]
with nonnegative factors, and set \(\eta_k=v_k^\theta\).  Then
\[
        \|\eta\|_{T^p_p(Z)}^{1/\theta}
        =
        \|v\|_{T^{p\theta}_{p\theta}(Z)}.
\]
By Lemma~\ref{lem:background-tools}\textup{(vi)}, the boundary shadow of
\(a_k\) has length comparable to \(\delta(a_k)\).  Tonelli's theorem and
H\"older's inequality from Lemma~\ref{lem:background-tools}\textup{(ii), (i)},
applied with the conjugate pairs \(((2\theta)',2\theta)\) and
\(((q\theta)',q\theta)\), yield
\[
\begin{aligned}
        \sum_k e_k|b_k^E|^{1/\theta}\delta(a_k)
        &\asymp
        \int_\T\sum_{a_k\in\Gamma(\xi)}
        x_k\bigl(\eta_k|b_k^E|\bigr)^{1/\theta}\dxi                 \\
        &\le
        \|x\|_{T^{(q\theta)'}_{(2\theta)'}(Z)}
        \|\{\eta_kb_k^E\}\|_{T^q_2(Z)}^{1/\theta}                  \\
        &\le
        \mathfrak M_b^{1/\theta}
        \|x\|_{T^{(q\theta)'}_{(2\theta)'}(Z)}
        \|v\|_{T^{p\theta}_{p\theta}(Z)} .
\end{aligned}
\]
Taking the infimum over all factorizations and using
\[
        \bigl(T^{r\theta}_{s\theta}(Z)\bigr)^*
        =
        T^{(r\theta)'}_{(s\theta)'}(Z)
\]
under the pairing
\(\sum_k c_k\overline{e_k}\delta(a_k)\), we obtain
\[
        \bigl\|\{|b_k^E|^{1/\theta}\}\bigr\|
        _{T^{r\theta}_{s\theta}(Z)}
        \lesssim \mathfrak M_b^{1/\theta}.
\]
Raising to the power \(\theta\) gives
\(\|b^E\|_{T^r_s(Z)}\lesssim\mathfrak M_b\), uniformly in \(E\).
Letting \(E\uparrow\N\) and applying monotone convergence from
Lemma~\ref{lem:background-tools}\textup{(ii)} proves the reverse
estimate in \textup{(3.22)}.
\end{proof}

\begin{lemma}\label{lem:caseC-area-criterion}
Let \(0<q<\infty\), \(p>\max\{2,q\}\), and \(\nu\in\calD\).  Put
\[
        s=\frac{2p}{p-2},
        \qquad
        r=\frac{pq}{p-q}.
\]
Let
\[
        u(z)=G(z)\delta(z)^\tau,
        \qquad G\in H(\D),\quad \tau\in\mathbb R,
\]
and define
\[
        A_u h(\xi)
        =
        \left(
        \int_{\Gamma(\xi)}|h(z)|^2|u(z)|^2\dA(z)
        \right)^{1/2}.
\]
Then \(A_u:A^p_\nu\to L^q(\T)\) is bounded if and only if
\[
        \Psi_{\nu,u}(\xi)
        =
        \left(
        \int_{\Gamma(\xi)}
        |u(z)|^s
        \left(\frac{\delta(z)^2}{\widehat\nu(z)}\right)^{\frac{2}{p-2}}
        \dA(z)
        \right)^{1/s}
        \in L^r(\T).
        \tag*{(3.24)}
\]
Moreover, every bounded \(A_u\) satisfies the sequential condition in
Proposition~\ref{prop:g-function-reduction}.
\end{lemma}

\begin{proof}
For sufficiency, apply generalized H\"older's inequality on \(\Gamma(\xi)\)
with conjugate exponents \(p/2\) and \(p/(p-2)\).  This gives
\[
        A_u h(\xi)
        \le
        H_h(\xi)\Psi_{\nu,u}(\xi),
\]
where
\[
        H_h(\xi)=
        \left(
        \int_{\Gamma(\xi)}
        |h(z)|^p\frac{\widehat\nu(z)}{\delta(z)^2}\dA(z)
        \right)^{1/p}.
\]
Since \(1/q=1/p+1/r\), a second application of
Lemma~\ref{lem:background-tools}\textup{(i)} on \(\T\) gives
\[
        \|A_u h\|_{L^q}
        \le
        \|H_h\|_{L^p}\|\Psi_{\nu,u}\|_{L^r}.
\]
By Tonelli's theorem and the boundary-shadow estimate in
Lemma~\ref{lem:background-tools}\textup{(ii), (vi)},
\[
        \|H_h\|_{L^p}^p
        \asymp
        \int_\D |h(z)|^p\frac{\widehat\nu(z)}{\delta(z)}\dA(z).
\]
By Lemma~\ref{lem:background-tools}\textup{(vi)}, the regularized weight
\(\widehat\nu/\delta\) induces an equivalent Bergman norm on analytic
functions.  Hence
\[
        \|A_u h\|_{L^q}
        \lesssim
        \|\Psi_{\nu,u}\|_{L^r}\|h\|_{A^p_\nu}.
\]

To verify the sequential condition, assume \(\Psi_{\nu,u}\in L^r\) and define
the outer tail
\[
        \Psi_{\nu,u,R}(\xi)=
        \left(
        \int_{\Gamma(\xi)\setminus R\D}
        |u(z)|^s
        \left(\frac{\delta(z)^2}{\widehat\nu(z)}\right)^{\frac{2}{p-2}}\dA(z)
        \right)^{1/s}.
\]
Since \(\Psi_{\nu,u,R}^r\to0\) pointwise and is dominated by
\(\Psi_{\nu,u}^r\in L^1(\T)\), dominated convergence from
Lemma~\ref{lem:background-tools}\textup{(ii)} gives
\(\|\Psi_{\nu,u,R}\|_{L^r}\to0\) as \(R\to1^-\).  For a bounded sequence
\(h_j\in A^p_\nu\) converging to zero locally uniformly, let
\(A_u^{\mathrm{in},R}\) denote the area map restricted to \(R\D\).  Then
\[
        \|A_u^{\mathrm{in},R}h_j\|_{L^q}
        \le
        \sup_{|z|\le R}|h_j(z)|
        \left\|
        \left(\int_{\Gamma(\xi)\cap R\D}|u(z)|^2\dA(z)\right)^{1/2}
        \right\|_{L^q(\T)}
        \longrightarrow0.
        \tag{3.24a}
\]
The second factor is finite because \(u\) is bounded on \(R\overline\D\).
The corresponding outer area map is controlled by
\(\|\Psi_{\nu,u,R}\|_{L^r}\).  Pointwise subadditivity therefore gives
\(A_u h_j\to0\) in \(L^q\).

We prove necessity.  Assume \(A_u:A^p_\nu\to L^q(\T)\) is bounded.  Let \(Z=\{a_k\}\)
be a lattice as in Lemma~\ref{lem:caseC-continuous-discrete}, and let \(b=\{b_k\}\) be
defined by (3.12).  Lemma~\ref{lem:caseC-random-testing} gives the multiplier
estimate for finitely supported sequences.  Since finitely supported
sequences are dense in \(T^p_p(Z)\cong\ell^p(\delta(a_k))\) and
\(T^q_2(Z)\) is complete, the estimate extends uniquely to all of
\(T^p_p(Z)\).  By
Lemma~\ref{lem:caseC-multiplier}, \(b\in T^r_s(Z)\), and
\[
        \|b\|_{T^r_s(Z)}
        \lesssim
        \|A_u\|_{A^p_\nu\to L^q}.
\]
Finally, Lemma~\ref{lem:caseC-continuous-discrete} gives
\[
        \|\Psi_{\nu,u}\|_{L^r(\T)}
        \asymp
        \|b\|_{T^r_s(Z)}<\infty.
\]
Thus (3.24) is necessary.  Together with the sufficiency and sequential
estimates above, this proves both assertions.
\end{proof}

\begin{theorem}\label{thm:caseC-bounded-compact}
Let \(0<q<\infty\), \(p>\max\{2,q\}\), \(\omega\in\calD\),
\(\varphi\in H(\D)\), \((\alpha,\beta)\in\calP\), and
\(\alpha,\beta>0\).  Assume \((H)\).  Then the following are
equivalent:
\begin{enumerate}[label=\textup{(\roman*)},leftmargin=2em]
\item \(V^\varphi_{\alpha,\beta}:A^p_\omega\to H^q\) is bounded;
\item \(V^\varphi_{\alpha,\beta}:A^p_\omega\to H^q\) is compact;
\item the tent condition in Theorem~\ref{thm:boundedness-main}\textup{(C)} holds.
\end{enumerate}
\end{theorem}

\begin{proof}
By Proposition~\ref{prop:g-function-reduction}, it remains to apply
Lemma~\ref{lem:caseC-area-criterion} to
\(u_\varphi=G_\varphi\delta^{\alpha-1}\) and \(\nu=\omega\delta^{p(\alpha-\beta)}\).  The
tent density expands as
\[
\begin{aligned}
        &|u_\varphi|^{\frac{2p}{p-2}}
        \left(\frac{\delta^2}{\widehat\nu}\right)^{\frac{2}{p-2}}  \\
        &\qquad\asymp
        |G_\varphi|^{\frac{2p}{p-2}}
        \frac{\delta^{\frac{2p(\alpha-1)+4-2p(\alpha-\beta)}{p-2}}}
             {\widehat\omega^{\frac{2}{p-2}}}  \\
        &\qquad=
        |G_\varphi|^{\frac{2p}{p-2}}
        \frac{\delta^{\frac{4+2p(\beta-1)}{p-2}}}
             {\widehat\omega^{\frac{2}{p-2}}}.
\end{aligned}
\]
This is condition \textup{(C)} in Theorem~\ref{thm:boundedness-main}.
\end{proof}

\subsection{Case D: the non-tangential maximal range}

Assume throughout this subsection that
\[
        0<q<p\le2,
        \qquad
        r=\frac{pq}{p-q}.
        \tag{3.27}
\]
For \(\nu\in\calD\) and a measurable function \(u\), put
\[
        A_u h(\xi)=
        \left(
        \int_{\Gamma(\xi)}|h(z)|^2|u(z)|^2\dA(z)
        \right)^{1/2}
\]
and
\[
        B_{\nu,u}(\xi)=
        \sup_{z\in\Gamma(\xi)}
        |u(z)|\frac{\delta(z)}{\widehat\nu(z)^{1/p}}.
        \tag{3.28}
\]
For \(0<R<1\), define the outer versions
\[
        A_{u,R}^{\rm out}h(\xi)=
        \left(
        \int_{\Gamma(\xi)\setminus R\D}|h(z)|^2|u(z)|^2\dA(z)
        \right)^{1/2}
\]
and
\[
        B_{\nu,u,R}(\xi)=
        \sup_{z\in\Gamma(\xi)\setminus R\D}
        |u(z)|\frac{\delta(z)}{\widehat\nu(z)^{1/p}}.
        \tag{3.29}
\]

\begin{lemma}\label{lem:caseD-basic-tent-estimate}
Let \(0<p\le2\), \(\nu\in\calD\), and
\[
        H_h(\xi)=
        \left(
        \int_{\Gamma(\xi)}|h(z)|^2
        \frac{\widehat\nu(z)^{2/p}}{\delta(z)^2}\dA(z)
        \right)^{1/2}.
\]
Then
\[
        \|H_h\|_{L^p(\T)}\lesssim \|h\|_{A^p_\nu},
        \qquad h\in A^p_\nu .
        \tag{3.30}
\]
The same estimate holds with \(\Gamma(\xi)\) replaced by
\(\Gamma(\xi)\setminus R\D\), uniformly in \(R\).
\end{lemma}

\begin{proof}
Put
\[
        \kappa_\nu(z)=\frac{\widehat\nu(z)}{\delta(z)}.
\]
Lemma~\ref{lem:background-tools}\textup{(vi)} gives
\[
        \int_\D |h(z)|^p\kappa_\nu(z)\dA(z)
        \asymp
        \|h\|_{A^p_\nu}^p
        \tag{3.30a}
\]
for analytic \(h\).  Choose pseudo-hyperbolic radii
\(0<\rho_0<\rho_1<1\) and a lattice \(Z=\{a_k\}\) such that
\(\{D(a_k,\rho_0)\}\) covers \(\D\) and \(\{D(a_k,\rho_1)\}\) has bounded overlap,
as in Lemma~\ref{lem:background-tools}\textup{(vi)}.  We
write
\[
        \delta_k=\delta(a_k),\qquad
        \widehat\nu_k=\widehat\nu(a_k),\qquad
        D_k=D(a_k,\rho_0),\qquad
        D_k^*=D(a_k,\rho_1),
\]
and
\[
        J_k=\int_{D_k^*}|h(w)|^p\kappa_\nu(w)\dA(w).
\]
By Lemma~\ref{lem:background-tools}\textup{(vi)}, the quantities \(\delta\),
\(\widehat\nu\), and \(\kappa_\nu\) on \(D_k^*\) are comparable to their
values at \(a_k\).  By the sub-mean estimate in part \textup{(iii)}, for
\(z\in D_k\),
\[
        |h(z)|^p
        \lesssim
        \frac{1}{\delta_k^2}\int_{D_k^*}|h(w)|^p\dA(w)
        \lesssim
        \frac{J_k}{\delta_k\widehat\nu_k}.
        \tag{3.30b}
\]
Therefore, for
\[
        c_k=
        \int_{D_k}|h(z)|^2
        \frac{\widehat\nu(z)^{2/p}}{\delta(z)^2}\dA(z),
\]
we have, using local comparability and \(A(D_k)\asymp\delta_k^2\),
\[
\begin{aligned}
        c_k
        &\lesssim
        \left(\frac{J_k}{\delta_k\widehat\nu_k}\right)^{2/p}
        \frac{\widehat\nu_k^{2/p}}{\delta_k^2}A(D_k)       \\
        &\lesssim
        \delta_k^{-2/p}J_k^{2/p}.
\end{aligned}
        \tag{3.30c}
\]
After the fixed aperture enlargement in
Lemma~\ref{lem:background-tools}\textup{(vi)}, the covering by the discs \(D_k\)
gives
\[
        H_h(\xi)^2
        \lesssim
        \sum_{a_k\in\widetilde\Gamma(\xi)}c_k.
\]
Since \(0<p\le2\), the exponent \(p/2\le1\), and
Lemma~\ref{lem:background-tools}\textup{(i)} gives
\[
        H_h(\xi)^p
        \lesssim
        \sum_{a_k\in\widetilde\Gamma(\xi)}c_k^{p/2}.
\]
Integrating in \(\xi\) and using the boundary-shadow estimate in
Lemma~\ref{lem:background-tools}\textup{(vi)}, we obtain
\[
\begin{aligned}
        \|H_h\|_{L^p(\T)}^p
        &\lesssim
        \sum_k c_k^{p/2}\delta_k                                      \\
        &\lesssim
        \sum_k J_k.
\end{aligned}
\]
The bounded overlap in Lemma~\ref{lem:background-tools}\textup{(vi)}, together
with \((3.30a)\), gives
\[
        \sum_k J_k
        \lesssim
        \int_\D |h(w)|^p\kappa_\nu(w)\dA(w)
        \lesssim
        \|h\|_{A^p_\nu}^p.
\]
This proves \((3.30)\).  If \(\Gamma(\xi)\) is replaced by
\(\Gamma(\xi)\setminus R\D\), the inner integral only decreases, so the same estimate
holds uniformly in \(R\).
\end{proof}

\begin{lemma}\label{lem:caseD-continuous-discrete}
Let \(0<q<p\le2\), \(r=pq/(p-q)\), \(\nu\in\calD\), and let
\[
        u(z)=G(z)\delta(z)^\tau,
        \qquad G\in H(\D),\quad \tau\in\mathbb R.
\]
Choose fixed radii \(0<\rho_0<\rho_1<\rho_2<1\) and a pseudo-hyperbolic lattice
\(Z=\{a_k\}\) such that \(\{D(a_k,\rho_0)\}\) covers \(\D\), while
\(\{D(a_k,\rho_2)\}\) has bounded overlap, as guaranteed by
Lemma~\ref{lem:background-tools}\textup{(vi)}.  Set
\[
        b_k=
        \frac{\delta(a_k)}{\widehat\nu(a_k)^{1/p}}
        \sup_{z\in D(a_k,\rho_0)}|u(z)|.
        \tag{3.31}
\]
Then
\[
        \|B_{\nu,u}\|_{L^r(\T)}\asymp \|\{b_k\}\|_{T^r_\infty(Z)}.
        \tag{3.32}
\]
For \(0<R<1\), define the truncated discrete sequence
\[
        b_k^{[R]}=
        \frac{\delta(a_k)}{\widehat\nu(a_k)^{1/p}}
        \sup_{z\in D(a_k,\rho_0)\cap(\D\setminus R\D)}|u(z)|,
        \tag{3.33}
\]
where the supremum is understood to be zero if the set is empty.  Then
\[
        \|B_{\nu,u,R}\|_{L^r(\T)}\asymp \|\{b_k^{[R]}\}\|_{T^r_\infty(Z)},
        \qquad 0<R<1.
        \tag{3.34}
\]
The constants are independent of \(R\), by the aperture invariance in
Lemma~\ref{lem:background-tools}\textup{(vi)}.
\end{lemma}

\begin{proof}
If \(z\in\Gamma(\xi)\), choose \(a_k\) with \(z\in D(a_k,\rho_0)\).  Then
\(a_k\in\widetilde\Gamma(\xi)\), \(\delta(z)\asymp\delta(a_k)\), and
\(\widehat\nu(z)\asymp\widehat\nu(a_k)\), by
Lemma~\ref{lem:background-tools}\textup{(vi)}.  Therefore
\[
        |u(z)|\frac{\delta(z)}{\widehat\nu(z)^{1/p}}
        \lesssim
        \sup_{a_k\in\widetilde\Gamma(\xi)} b_k.
\]
This proves
\(\|B_{\nu,u}\|_{L^r}\lesssim\|\{b_k\}\|_{T^r_\infty}\).  Conversely, if
\(a_k\in\Gamma(\xi)\) and
\(w_k\in D(a_k,\rho_0)\) is chosen so that
\(|u(w_k)|\ge \frac12\sup_{D(a_k,\rho_0)}|u|\), then
\(w_k\in\widetilde\Gamma(\xi)\) and local comparability in
Lemma~\ref{lem:background-tools}\textup{(vi)} gives
\[
        b_k
        \lesssim
        |u(w_k)|\frac{\delta(w_k)}{\widehat\nu(w_k)^{1/p}}
        \le B_{\nu,u}^{\widetilde\Gamma}(\xi).
\]
Taking the supremum over \(a_k\in\Gamma(\xi)\) and applying
Lemma~\ref{lem:background-tools}\textup{(vi)} proves the reverse inequality.

For \((3.34)\), repeat the two preceding inequalities with
\(D(a_k,\rho_0)\) replaced by
\(D(a_k,\rho_0)\cap(\D\setminus R\D)\).  If
\(z\in\Gamma(\xi)\setminus R\D\), choose \(a_k\) with
\(z\in D(a_k,\rho_0)\).  Then
\[
        |u(z)|\frac{\delta(z)}{\widehat\nu(z)^{1/p}}
        \lesssim
        \sup_{a_k\in\widetilde\Gamma(\xi)} b_k^{[R]}.
\]
Conversely, if \(a_k\in\Gamma(\xi)\) and the set
\(D(a_k,\rho_0)\cap(\D\setminus R\D)\) is non-empty, choose
\(w_k\) in this set with
\[
        |u(w_k)|\ge \frac12
        \sup_{D(a_k,\rho_0)\cap(\D\setminus R\D)}|u|.
\]
Then \(w_k\in\widetilde\Gamma(\xi)\setminus R\D\), and hence
\[
        b_k^{[R]}
        \lesssim
        |u(w_k)|\frac{\delta(w_k)}{\widehat\nu(w_k)^{1/p}}
        \le B_{\nu,u,R}^{\widetilde\Gamma}(\xi).
\]
Taking suprema and applying Lemma~\ref{lem:background-tools}\textup{(vi)} proves
\((3.34)\).
\end{proof}

\begin{lemma}\label{lem:caseD-random-testing}
Let the assumptions and notation of Lemma~\ref{lem:caseD-continuous-discrete} hold.
If \(A_u:A^p_\nu\to L^q(\T)\) is bounded, then for every finitely supported sequence
\(\lambda=\{\lambda_k\}\),
\[
        \|\{\lambda_k b_k\}\|_{T^q_2(Z)}
        \lesssim
        \|A_u\|_{A^p_\nu\to L^q(\T)}\,
        \|\lambda\|_{T^p_p(Z)}.
        \tag{3.35}
\]
For the outer operator, one has uniformly in \(R\)
\[
        \|\{\lambda_k b_k^{[R]}\}\|_{T^q_2(Z)}
        \lesssim
        \|A_{u,R_*}^{\rm out}\|_{A^p_\nu\to L^q(\T)}\,
        \|\lambda\|_{T^p_p(Z)},
        \tag{3.36}
\]
where \(R_*=R_*(R)<R\) and \(R_*\to1^-\) as \(R\to1^-\).
\end{lemma}

\begin{proof}
Choose \(M\) large and put
\[
        e_k(z)=
        \frac{\delta(a_k)^M}
             {\widehat\nu(a_k)^{1/p}(1-\overline{a_k}z)^M}.
        \tag{3.37}
\]
The atomic decomposition for \(\calD\)-weighted Bergman spaces
\cite[Lemma~2.12]{DuanWangWang2021} implies that, for every finitely supported
\(\lambda\) and every choice of signs \(\varepsilon_k\),
\[
        h_\varepsilon(z)=\sum_k \varepsilon_k\lambda_k e_k(z)
\]
satisfies
\[
        \|h_\varepsilon\|_{A^p_\nu}\lesssim \|\lambda\|_{T^p_p(Z)},
        \tag{3.38}
\]
and on \(D(a_k,\rho_1)\)
\[
        |e_k(z)|\asymp \widehat\nu(a_k)^{-1/p}.
        \tag{3.39}
\]
Fix \(\xi\).  We apply the Hilbert-valued Khinchine inequality
\cite[Lemmas~2.8--2.9]{DuanWangWang2021}.  The Hilbert space is
\(L^2(\Gamma(\xi),|u|^2dA)\).  Aperture invariance is supplied by
Lemma~\ref{lem:background-tools}\textup{(ii), (vi)}.  Tonelli's theorem and
the boundedness of \(A_u\) then give
\[
\begin{aligned}
        &\left\|
        \left(
        \int_{\Gamma(\xi)}
        \sum_k |\lambda_k|^2|e_k(z)|^2|u(z)|^2\dA(z)
        \right)^{1/2}
        \right\|_{L^q(\T)}                                     \\
        &\qquad\lesssim
        \|A_u\|_{A^p_\nu\to L^q(\T)}\|\lambda\|_{T^p_p(Z)}.
\end{aligned}
        \tag{3.40}
\]
Because the integrand in \((3.40)\) is a sum of nonnegative terms, for each
\(k\) with \(a_k\in\Gamma(\xi)\) we may retain only the contribution of
\(D(a_k,\rho_1)\) to the \(k\)-th summand, after a fixed aperture enlargement.
Using \((3.39)\), the resulting contribution is bounded below by
\[
        |\lambda_k|^2\widehat\nu(a_k)^{-2/p}
        \int_{D(a_k,\rho_1)}|u(z)|^2\dA(z).
\]
To obtain the required lower bound from this local integral, choose
\(w_k\in D(a_k,\rho_0)\) so that
\[
        |u(w_k)|\ge\frac12\sup_{D(a_k,\rho_0)}|u|.
\]
Since \(u=G\delta^\tau\), Lemma~\ref{lem:background-tools}\textup{(iii), (vi)}
implies
\[
        \delta(a_k)^2\sup_{D(a_k,\rho_0)}|u|^2
        \lesssim
        \int_{D(a_k,\rho_1)}|u(z)|^2\dA(z).
        \tag{3.41}
\]
Thus the left-hand side of (3.40) dominates
\(\|\{\lambda_k b_k\}\|_{T^q_2(Z)}\), with the aperture enlargement justified
by Lemma~\ref{lem:background-tools}\textup{(vi)}.  This proves
(3.35).

For the outer estimate, repeat \((3.37)\)--\((3.41)\), restricting the
\(k\)-th summand to the corresponding truncated disc and the cone integral to
\(\Gamma(\xi)\setminus R_*\D\).  If
\(b_k^{[R]}\ne0\), choose
\(w_k\in D(a_k,\rho_0)\cap(\D\setminus R\D)\) so that
\[
        |u(w_k)|\ge\frac12
        \sup_{D(a_k,\rho_0)\cap(\D\setminus R\D)}|u|.
\]
Choose \(0<\sigma<1\) so that
\(D(w,\sigma)\subset D(a_k,\rho_1)\) whenever
\(w\in D(a_k,\rho_0)\), and set
\[
        R_*(R)=\max\left\{0,\frac{R-\sigma}{1-\sigma R}\right\}.
\]
If \(|w|\ge R\), then
\(D(w,\sigma)\subset\D\setminus R_*(R)\D\).
In particular, \(R_*(R)\to1^-\) as \(R\to1^-\), uniformly in \(k\).
Applying the local estimate inside
\(D(a_k,\rho_1)\cap(\D\setminus R_*\D)\) gives
\[
        \delta(a_k)^2
        \sup_{D(a_k,\rho_0)\cap(\D\setminus R\D)}|u|^2
        \lesssim
        \int_{D(a_k,\rho_1)\cap(\D\setminus R_*\D)}|u(z)|^2\dA(z).
        \tag{3.41a}
\]
The preceding randomized estimate therefore yields (3.36) with \(b_k^{[R]}\) in
place of \(b_k\).
\end{proof}

\begin{lemma}\label{lem:caseD-multiplier}
Let \(0<q<p\le2\) and \(r=pq/(p-q)\).  For a sequence
\(b=\{b_k\}\) on a pseudo-hyperbolic lattice \(Z\), multiplication by \(b\) maps
\(T^p_p(Z)\) boundedly into \(T^q_2(Z)\) if and only if
\(b\in T^r_\infty(Z)\).  Moreover,
\[
        \|b\|_{T^r_\infty(Z)}
        \asymp
        \sup_{\lambda\ne0}
        \frac{\|\{\lambda_kb_k\}\|_{T^q_2(Z)}}{\|\lambda\|_{T^p_p(Z)}}.
        \tag{3.42}
\]
\end{lemma}

\begin{proof}
If \(b\in T^r_\infty(Z)\), then for every \(\xi\in\T\), since \(p\le2\),
\[
\begin{aligned}
        \left(\sum_{a_k\in\Gamma(\xi)}|\lambda_kb_k|^2\right)^{1/2}
        &\le
        \left(\sup_{a_k\in\Gamma(\xi)}|b_k|\right)
        \left(\sum_{a_k\in\Gamma(\xi)}|\lambda_k|^2\right)^{1/2}       \\
        &\le
        \left(\sup_{a_k\in\Gamma(\xi)}|b_k|\right)
        \left(\sum_{a_k\in\Gamma(\xi)}|\lambda_k|^p\right)^{1/p}.
\end{aligned}
\]
Lemma~\ref{lem:background-tools}\textup{(i)} on \(\T\), with
\(1/q=1/p+1/r\), gives
\[
        \|\{\lambda_kb_k\}\|_{T^q_2(Z)}
        \lesssim
        \|b\|_{T^r_\infty(Z)}\|\lambda\|_{T^p_p(Z)}.
\]

Conversely, denote the supremum in \textup{(3.42)} by \(\mathfrak M_b\), and
suppose that it is finite.  We give the endpoint power argument used in the
necessity proof of \cite[Proposition~2.19]{DuanWangWang2021}.  Choose
\[
        \theta>\max\left\{\frac1q,\frac12\right\},
        \qquad
        \iota=
        \left(\frac1{(2\theta)'}+\frac1{p\theta}\right)^{-1}
        =
        \frac{2p\theta}{2p\theta+2-p}.
\]
Then \(0<\iota\le1\), with equality precisely when \(p=2\), and
\[
        \frac1{(r\theta)'}
        =
        \frac1{(q\theta)'}+\frac1{p\theta},
        \qquad
        \frac1\iota
        =
        \frac1{(2\theta)'}+\frac1{p\theta}.
\]
The discrete tent-space duality and factorization theorems
\cite[Lemmas~2.14--2.15]{DuanWangWang2021} therefore give
\[
        \bigl(T^{(r\theta)'}_\iota(Z)\bigr)^*
        =
        T^{r\theta}_\infty(Z),
        \qquad
        T^{(r\theta)'}_\iota(Z)
        =
        T^{(q\theta)'}_{(2\theta)'}(Z)
        \cdot T^{p\theta}_{p\theta}(Z),
        \tag{3.43}
\]
under the pairing
\(\sum_k x_k\overline{y_k}\delta(a_k)\), with equivalence of the dual and
factorization norms.

Let \(E\subset\N\) be finite and let \(b^E\) be the corresponding coordinate
truncation.  Its multiplier norm is at most \(\mathfrak M_b\).  By solidity, it
suffices to test against a nonnegative finitely supported
\(e\in T^{(r\theta)'}_\iota(Z)\).  Factor
\[
        e_k=x_kv_k,
        \qquad
        x\in T^{(q\theta)'}_{(2\theta)'}(Z),
        \quad
        v\in T^{p\theta}_{p\theta}(Z).
\]
We may take the factors nonnegative.  Set \(\eta_k=v_k^\theta\), so that
\[
        \|\eta\|_{T^p_p(Z)}^{1/\theta}
        =
        \|v\|_{T^{p\theta}_{p\theta}(Z)}.
\]
Using the boundary-shadow estimate from
Lemma~\ref{lem:background-tools}\textup{(vi)} and H\"older's inequality from
part \textup{(i)}, first with \(((2\theta)',2\theta)\) and then with
\(((q\theta)',q\theta)\), we obtain
\[
\begin{aligned}
        \sum_k e_k|b_k^E|^{1/\theta}\delta(a_k)
        &\asymp
        \int_\T\sum_{a_k\in\Gamma(\xi)}
        x_k\bigl(\eta_k|b_k^E|\bigr)^{1/\theta}\dxi                 \\
        &\le
        \|x\|_{T^{(q\theta)'}_{(2\theta)'}(Z)}
        \|\{\eta_kb_k^E\}\|_{T^q_2(Z)}^{1/\theta}                  \\
        &\le
        \mathfrak M_b^{1/\theta}
        \|x\|_{T^{(q\theta)'}_{(2\theta)'}(Z)}
        \|v\|_{T^{p\theta}_{p\theta}(Z)} .
\end{aligned}
\]
Taking the infimum over all factorizations and using the first identity in
\textup{(3.43)} yields
\[
        \bigl\|\{|b_k^E|^{1/\theta}\}\bigr\|
        _{T^{r\theta}_\infty(Z)}
        \lesssim \mathfrak M_b^{1/\theta}.
\]
Hence \(\|b^E\|_{T^r_\infty(Z)}\lesssim\mathfrak M_b\), uniformly in
\(E\).  Monotone convergence from Lemma~\ref{lem:background-tools}\textup{(ii)}
as \(E\uparrow\N\) proves
\(b\in T^r_\infty(Z)\) and the reverse estimate in \textup{(3.42)}.
\end{proof}

\begin{lemma}\label{lem:caseD-outer-norm}
Let \(\nu\in\calD\), \(0<q<p\le2\), and let \(u\) be measurable on \(\D\).
For \(h\in A^p_\nu\), define
\[
        A_u h(\xi)=
        \left(\int_{\Gamma(\xi)}|h(z)|^2|u(z)|^2\dA(z)\right)^{1/2}
\]
and
\[
        A_{u,R}^{\rm out}h(\xi)=
        \left(\int_{\Gamma(\xi)\setminus R\D}
        |h(z)|^2|u(z)|^2\dA(z)\right)^{1/2}.
\]
Assume that \(A_u:A^p_\nu\to L^q(\T)\) is bounded and satisfies the following
sequential condition:
\(\|A_u h_j\|_{L^q}\to0\) whenever \(\{h_j\}\) is bounded in \(A^p_\nu\) and
\(h_j\to0\) uniformly on compact subsets.  Then
\[
        \lim_{R\to1^-}
        \|A_{u,R}^{\rm out}\|_{A^p_\nu\to L^q(\T)}=0.
        \tag{3.44}
\]
\end{lemma}

\begin{proof}
Suppose not.  Then there are \(R_j\to1^-\), functions \(h_j\in A^p_\nu\) with
\(\|h_j\|_{A^p_\nu}\le1\), and \(\varepsilon_0>0\) such that
\[
        \|A_{u,R_j}^{\rm out}h_j\|_{L^q}\ge\varepsilon_0.
        \tag{3.45}
\]
By Lemma~\ref{lem:background-tools}\textup{(iii), (iv)}, a subsequence, still
called \(h_j\), converges locally uniformly to some analytic function \(h\).
Fatou's lemma from part \textup{(ii)} gives \(h\in A^p_\nu\), so \(h_j-h\) is
bounded in \(A^p_\nu\) and tends to zero locally uniformly; hence the
sequential condition gives
\[
        \|A_u(h_j-h)\|_{L^q}\to0.
\]
For the fixed function \(h\), \(A_{u,R_j}^{\rm out}h\to0\) pointwise and
\(A_{u,R_j}^{\rm out}h\le A_uh\in L^q(\T)\).  Dominated convergence from
Lemma~\ref{lem:background-tools}\textup{(ii)}, applied to the \(q\)-th powers,
therefore gives
\[
        \|A_{u,R_j}^{\rm out}h\|_{L^q}\to0.
\]
Since
\[
        A_{u,R_j}^{\rm out}h_j
        \le A_u(h_j-h)+A_{u,R_j}^{\rm out}h,
\]
the \(L^q\) quasi-triangle inequality from
Lemma~\ref{lem:background-tools}\textup{(i)} now contradicts \((3.45)\), proving
\((3.44)\).
\end{proof}

\begin{lemma}\label{lem:caseD-area-criterion}
Let \(\nu\in\calD\), \(0<q<p\le2\), put
\(r=pq/(p-q)\), and let
\[
        u(z)=G(z)\delta(z)^\tau,
        \qquad G\in H(\D),\quad \tau\in\mathbb R.
\]
Then \(A_u:A^p_\nu\to L^q(\T)\) is bounded if and only if
\(B_{\nu,u}\in L^r(\T)\).  Moreover, a bounded \(A_u\) satisfies the
sequential condition
\[
        \|A_uh_j\|_{L^q}\longrightarrow0
\]
for every bounded sequence \(\{h_j\}\subset A^p_\nu\) converging to zero
uniformly on compact subsets if and only if
\[
        \lim_{R\to1^-}\|B_{\nu,u,R}\|_{L^r(\T)}=0.
        \tag{3.46}
\]
\end{lemma}

\begin{proof}
Assume first \(B_{\nu,u}\in L^r(\T)\).  For \(z\in\Gamma(\xi)\),
\[
        |u(z)|\le B_{\nu,u}(\xi)\frac{\widehat\nu(z)^{1/p}}{\delta(z)}.
\]
Therefore
\[
        A_u h(\xi)
        \le B_{\nu,u}(\xi)H_h(\xi),
\]
where \(H_h\) is the function in Lemma~\ref{lem:caseD-basic-tent-estimate}.
Using Lemma~\ref{lem:background-tools}\textup{(i)} on \(\T\) and (3.30),
\[
        \|A_u h\|_{L^q}
        \lesssim
        \|B_{\nu,u}\|_{L^r}\|h\|_{A^p_\nu}.
\]
Assume now that the tail condition \((3.46)\) holds.  We first show that it
implies \(B_{\nu,u}\in L^r(\T)\).  Fix \(R_0<1\) so close to one that
\(\|B_{\nu,u,R_0}\|_{L^r}<\infty\).  On the compact disc \(R_0\D\), the function
\[
        |u(z)|\frac{\delta(z)}{\widehat\nu(z)^{1/p}}
\]
is bounded; hence the corresponding inner maximal function belongs to
\(L^r(\T)\).  Since
\[
        B_{\nu,u}\le B^{\rm in}_{\nu,u,R_0}+B_{\nu,u,R_0},
\]
we get \(B_{\nu,u}\in L^r(\T)\), and the boundedness part just proved applies.
Applying the preceding pointwise estimate with
\(B_{\nu,u}\) and \(A_u\) replaced by \(B_{\nu,u,R}\) and
\(A_{u,R}^{\rm out}\) gives
\[
        \|A_{u,R}^{\rm out}h\|_{L^q}
        \lesssim
        \|B_{\nu,u,R}\|_{L^r}\|h\|_{A^p_\nu}.
        \tag{3.47}
\]
For a bounded sequence \(h_j\to0\) locally uniformly, the inner part obeys
\[
        \|A_u^{\mathrm{in},R}h_j\|_{L^q}
        \le
        \sup_{|z|\le R}|h_j(z)|
        \left\|
        \left(\int_{\Gamma(\xi)\cap R\D}|u(z)|^2\dA(z)\right)^{1/2}
        \right\|_{L^q(\T)}
        \longrightarrow0,
        \tag{3.47a}
\]
where \(A_u^{\mathrm{in},R}\) denotes the area map restricted to \(R\D\), and
the second factor is finite by local boundedness of \(u\).  Combining this
inner limit with \((3.47)\), then letting first \(j\to\infty\) and
subsequently \(R\to1^-\), proves the sequential condition.

Conversely, suppose \(A_u\) is bounded.
Lemma~\ref{lem:caseD-random-testing} gives the multiplier estimate for
finitely supported sequences.  Since finitely supported sequences are dense
in \(T^p_p(Z)\cong\ell^p(\delta(a_k))\) and \(T^q_2(Z)\) is complete, the
estimate extends uniquely to all of \(T^p_p(Z)\).
Lemma~\ref{lem:caseD-multiplier} then gives \(b\in T^r_\infty(Z)\), and
Lemma~\ref{lem:caseD-continuous-discrete} yields
\(B_{\nu,u}\in L^r(\T)\), proving necessity.

Finally, assume that \(A_u\) is bounded and satisfies the sequential
condition.  In the outer randomized testing lemma,
\(R_*=R_*(R)\to1^-\) as \(R\to1^-\).  Hence
Lemma~\ref{lem:caseD-outer-norm} gives \(\|A_{u,R_*}^{\rm out}\|\to0\).
For each \(R\), the same density--completeness argument extends \((3.36)\) to
all of \(T^p_p(Z)\), with the same bound uniformly in \(R\).  Hence
Lemmas~\ref{lem:caseD-multiplier} and
\ref{lem:caseD-continuous-discrete} give
\[
        \|B_{\nu,u,R}\|_{L^r}
        \lesssim
        \|A_{u,R_*}^{\rm out}\|_{A^p_\nu\to L^q}
        \longrightarrow0.
\]
This proves \((3.46)\).
\end{proof}

\begin{theorem}\label{thm:caseD-boundedness}
Let \(0<q<p\le2\), \(\omega\in\calD\), \(\varphi\in H(\D)\),
\((\alpha,\beta)\in\calP\), and \(\alpha,\beta>0\).  Assume \((H)\).  Then
\(V^\varphi_{\alpha,\beta}:A^p_\omega\to H^q\) is bounded if and only if
condition \textup{(D)} in Theorem~\ref{thm:boundedness-main} is satisfied.
\end{theorem}

\begin{proof}
By Proposition~\ref{prop:g-function-reduction} and
Lemma~\ref{lem:caseD-area-criterion}, the corresponding area condition is
\[
        B_{\nu,u_\varphi}(\xi)=
        \sup_{z\in\Gamma(\xi)}
        |u_\varphi(z)|\frac{\delta(z)}{\widehat\nu(z)^{1/p}}
        \in L^{pq/(p-q)}(\T).
\]
Since \(u_\varphi=G_\varphi\delta^{\alpha-1}\) and
\(\widehat\nu^{1/p}\asymp\widehat\omega^{1/p}\delta^{\alpha-\beta}\), this is
condition \textup{(D)}:
\[
        \xi\mapsto
        \sup_{z\in\Gamma(\xi)}
        |G_\varphi(z)|\frac{\delta(z)^\beta}{\widehat\omega(z)^{1/p}}
        \in L^{pq/(p-q)}(\T).
\]
\end{proof}

\begin{proposition}\label{prop:caseD-compactness}
Under the assumptions of Theorem~\ref{thm:caseD-boundedness},
\(V^\varphi_{\alpha,\beta}:A^p_\omega\to H^q\) is compact if and only if
condition \((D_0)\) in Theorem~\ref{thm:compactness-main} holds.
\end{proposition}

\begin{proof}
By Proposition~\ref{prop:g-function-reduction} and
Lemma~\ref{lem:caseD-area-criterion}, compactness is equivalent to
\[
        \lim_{R\to1^-}
        \int_{\T}
        \sup_{z\in\Gamma(\xi)\setminus R\D}
        |u_\varphi(z)|^{\frac{pq}{p-q}}
        \frac{\delta(z)^{\frac{pq}{p-q}}}{\widehat\nu(z)^{\frac{q}{p-q}}}
        \dxi=0.
\]
Substituting \(u_\varphi=G_\varphi\delta^{\alpha-1}\) and using
\(\widehat\nu\asymp\widehat\omega\delta^{p(\alpha-\beta)}\), this becomes
\[
        \lim_{R\to1^-}
        \int_{\T}
        \sup_{z\in\Gamma(\xi)\setminus R\D}
        |G_\varphi(z)|^{\frac{pq}{p-q}}
        \frac{\delta(z)^{\frac{pq\beta}{p-q}}}{\widehat\omega(z)^{\frac{q}{p-q}}}
        \dxi=0,
\]
which is precisely condition \((D_0)\) in
Theorem~\ref{thm:compactness-main}.
\end{proof}

\subsection{Proofs of the main theorems}

\begin{proof}[Proof of Theorem~\ref{thm:boundedness-main}]
In the range \(\alpha\ge\beta\), Corollary~\ref{cor:automatic-shift} gives \((H)\).
Theorem~\ref{thm:caseA-boundedness} proves Case A,
Theorem~\ref{thm:caseB-boundedness} proves Case B,
Theorem~\ref{thm:caseC-bounded-compact} proves Case C, and
Theorem~\ref{thm:caseD-boundedness} proves Case D.  These four alternatives cover
all \(0<p,q<\infty\).
\end{proof}

\begin{proof}[Proof of Theorem~\ref{thm:compactness-main}]
The hypothesis \((H)\) is automatic when \(\alpha\ge\beta\).  The four compactness cases follow
from Proposition~\ref{prop:caseA-compactness},
Proposition~\ref{prop:caseB-compactness},
Theorem~\ref{thm:caseC-bounded-compact}, and
Proposition~\ref{prop:caseD-compactness}.
\end{proof}

\begin{proof}[Proof of Theorem~\ref{thm:conditional-extension}]
By Proposition~\ref{prop:negative-obstruction}, the condition
\(p(\beta-\alpha)<d_-(\omega)\) is equivalent to \((H)\).
The case theorems in Section~\ref{sec:four-cases} use \((H)\), but not the
inequality \(\alpha\ge\beta\).  The two preceding assembly arguments therefore
apply unchanged.
\end{proof}

\section{Consistency with the power-weight and classical theories}

\begin{proposition}\label{prop:standard-weight-check}
Let \(0<p,q<\infty\), \(\gamma>-1\), \(\varphi\in H(\D)\),
\((\alpha,\beta)\in\calP\), and \(\alpha,\beta>0\).  Assume
\(\gamma+p(\alpha-\beta)>-1\), and put
\(\omega(z)=\delta(z)^\gamma\).  Then the four boundedness criteria for
\(V^\varphi_{\alpha,\beta}:A^p_\omega\to H^q\) coincide with the
standard-weight fractional Bergman-to-Hardy criteria
in~\cite[Theorem~3]{FangGuoHouZhuBergmanHardy}.  The compactness criteria
specialize, respectively, to the boundary limit in Case A, the vanishing
Carleson condition in Case B, the same tent condition as for boundedness in
Case C, and the outer maximal-tail condition in Case D.
\end{proposition}

\begin{proof}
Here
\[
\widehat\omega(z)\asymp\delta(z)^{\gamma+1},\qquad
d_-(\omega)=\gamma+1,\qquad
\nu_{\alpha,\beta,p}=\delta^{\gamma+p(\alpha-\beta)}.
\]
Thus Corollary~\ref{cor:automatic-shift} and
Proposition~\ref{prop:negative-obstruction} give
\[
(H)\Longleftrightarrow p(\beta-\alpha)<\gamma+1
\Longleftrightarrow\gamma+p(\alpha-\beta)>-1.
\]
The conditions in Cases A and D become, respectively,
\[
\sup_{z\in\D}|G_\varphi(z)|
\delta(z)^{\beta+1/q-(\gamma+2)/p}<\infty,\qquad
\sup_{z\in\Gamma(\xi)}|G_\varphi(z)|
\delta(z)^{\beta-(\gamma+1)/p}\in L^{pq/(p-q)}(\T).
\]
The powers of \(\delta\) in Cases B and C are, respectively,
\[
\frac{p+2+2p(\beta-1)-2(\gamma+1)}{p-2}
=\frac{2p\beta-2\gamma-p}{p-2},\qquad
\frac{4+2p(\beta-1)-2(\gamma+1)}{p-2}
=\frac{2-2\gamma+2p(\beta-1)}{p-2};
\]
the latter is the standard-weight tent exponent.
To recover the zero-operator alternatives in
\cite[Theorem~3]{FangGuoHouZhuBergmanHardy}, we verify the four thresholds
directly.

In Case A, suppose that
\(\beta<(\gamma+2)/p-1/q\), and put
\(\varepsilon=(\gamma+2)/p-1/q-\beta>0\).
Condition \textup{(A)} gives
\[
|G_\varphi(z)|\lesssim\delta(z)^\varepsilon,\qquad
M_\infty(r,G_\varphi)
=\max_{|z|=r}|G_\varphi(z)|
\lesssim(1-r^2)^\varepsilon\longrightarrow0.
\]
Since \(M_\infty(r,G_\varphi)\) is nondecreasing by the maximum-modulus
principle, \(G_\varphi\equiv0\).  Thus Case A admits
\(G_\varphi\not\equiv0\) exactly when
\(\beta\ge(\gamma+2)/p-1/q\).

In Case B, put \(s=2p/(p-2)\) and
\(\lambda_B=(2p\beta-2\gamma-p)/(p-2)\).
If \(\beta\le(\gamma+1)/p\), then \(\lambda_B\le-1\).  A Carleson
measure has finite total mass, since one may take \(I=\T\) in its
definition.  If \(G_\varphi\not\equiv0\), then the integral means
\[
        M_s(r,G_\varphi)^s
        =\int_\T|G_\varphi(r\xi)|^s\dxi
\]
are nondecreasing and, for some \(r_0<1\), are bounded below by a positive
constant on \([r_0,1)\).  Hence
\[
\begin{aligned}
        \mu_{\varphi,B}(\D)
        &\asymp
        \int_0^1 M_s(r,G_\varphi)^s
        (1-r^2)^{\lambda_B}r\,dr  \\
        &\ge
        c\int_{r_0}^1(1-r^2)^{\lambda_B}r\,dr
        =\infty,
\end{aligned}
\]
a contradiction.  Therefore \(G_\varphi\equiv0\) whenever
\(\beta\le(\gamma+1)/p\), including the logarithmically divergent
endpoint \(\lambda_B=-1\).

For Case C, set \(s=2p/(p-2)\), \(t=pq/(p-q)\), and
\(\lambda_C=(2-2\gamma+2p(\beta-1))/(p-2)\).
If \(\beta\le(\gamma+1)/p\), then \(\lambda_C\le-2\).  Let
\(\Gamma^*(\xi)\) denote a cone with a fixed larger aperture than
\(\Gamma(\xi)\), and define
\[
        \Psi_R^*(\xi)=
        \left(
        \int_{\Gamma^*(\xi)\setminus R\D}
        |G_\varphi(z)|^s\delta(z)^{\lambda_C}\dA(z)
        \right)^{1/s}.
\]
The Case C condition and the aperture equivalence in
Lemma~\ref{lem:background-tools}\textup{(vi)} give
\(\Psi_0^*\in L^t(\T)\).  Moreover,
\(\Psi_R^*(\xi)\to0\) for almost every \(\xi\) as \(R\to1^-\), and
\(0\le\Psi_R^*\le\Psi_0^*\).  Dominated convergence applied to
\((\Psi_R^*)^t\), which is valid also for \(0<t<1\), yields
\[
        \|\Psi_R^*\|_{L^t(\T)}\longrightarrow0.
        \tag{4.1}
\]
Fix a sufficiently small pseudo-hyperbolic radius \(c>0\).  If
\(z=\rho\xi\) and \(\rho>c\), then \(D(z,c)\) is contained in
\(\Gamma^*(\xi)\setminus R(\rho)\D\), where
\(R(\rho)=(\rho-c)/(1-c\rho)\to1^-\).  Throughout this disc,
\[
        \delta(w)\asymp\delta(z),
        \qquad
        \int_{D(z,c)}\dA(w)\asymp\delta(z)^2.
\]
The sub-mean inequality for \(|G_\varphi|^s\) therefore yields
\[
\begin{aligned}
        |G_\varphi(\rho\xi)|^s
        &\lesssim
        \delta(\rho\xi)^{-2}
        \int_{D(\rho\xi,c)}|G_\varphi(w)|^s\dA(w) \\
        &\lesssim
        \delta(\rho\xi)^{-2-\lambda_C}
        \int_{D(\rho\xi,c)}
        |G_\varphi(w)|^s\delta(w)^{\lambda_C}\dA(w) \\
        &\le C\Psi_{R(\rho)}^*(\xi)^s,
\end{aligned}
\]
because \(-2-\lambda_C\ge0\) and \(0<\delta(\rho\xi)\le1\).
Together with \((4.1)\), this gives
\[
        M_t(\rho,G_\varphi)
        \lesssim
        \|\Psi_{R(\rho)}^*\|_{L^t(\T)}
        \longrightarrow0.
\]
Since \(|G_\varphi|^t\) is subharmonic, its integral means are
nondecreasing.  Hence \(G_\varphi\equiv0\).  Thus Case C admits
\(G_\varphi\not\equiv0\) exactly when \(\beta>(\gamma+1)/p\).

Finally, in Case D let \(t=pq/(p-q)\) and
\[
B(\xi)=\sup_{z\in\Gamma(\xi)}
|G_\varphi(z)|\delta(z)^{\beta-(\gamma+1)/p}.
\]
If \(\beta<(\gamma+1)/p\), put
\(\varepsilon=(\gamma+1)/p-\beta>0\).  Since
\(\rho\xi\in\Gamma(\xi)\) and \(B\in L^t(\T)\) by \textup{(D)},
\[
|G_\varphi(\rho\xi)|\le B(\xi)(1-\rho^2)^\varepsilon,\qquad
M_t(\rho,G_\varphi)
\le(1-\rho^2)^\varepsilon\|B\|_{L^t(\T)}\longrightarrow0.
\]
Monotonicity of the integral means again implies
\(G_\varphi\equiv0\).  At the endpoint
\(\beta=(\gamma+1)/p\) no decay factor occurs, which accounts for the
non-strict inequality in the standard theorem.

It remains to prove that \(G_\varphi\equiv0\) if and only if
\(V^\varphi_{\alpha,\beta}\equiv0\).  Write
\[
\varphi(z)=\sum_{j\ge0}b_jz^j,\qquad
G_\varphi(z)=\sum_{j\ge0}g_jz^j,\qquad
g_j=\frac{\Gamma(j+1)}{\Gamma(j+1-\beta)}b_j.
\]
If \(\beta\notin\N\), then on \(\D\setminus\{0\}\)
\(\Ihat_{-\beta}\varphi=p_{-\beta}G_\varphi\).
If \(\beta=m\in\N\), the finite-kernel clause in the definition of
\(\Ihat_{-\beta}\) gives
\(G_\varphi(z)=z^m\Ihat_{-m}\varphi(z)=z^m\varphi^{(m)}(z)\).
Thus \(G_\varphi\equiv0\) implies
\(\Ihat_{-\beta}\varphi\equiv0\), and the defining formula for
\(V^\varphi_{\alpha,\beta}\) gives
\(V^\varphi_{\alpha,\beta}\equiv0\).

Conversely, suppose that \(G_\varphi\not\equiv0\), and let \(j_0\) be the
smallest index such that \(g_{j_0}\ne0\).  Choose \(k\) sufficiently
large that
\(d_k=\Gamma(k+1)/\Gamma(k+1+\beta-\alpha)\ne0\) and \(N=k+j_0\)
lies beyond all exceptional indices in the coefficient comparison
\((2.6b)\)--\((2.6f)\).  For \(f(z)=z^k\), its coefficient there is
\(c_N(f)=d_kg_{j_0}\ne0\).
If \(\alpha\notin\N\), formula \((2.6b)\) shows that the \(N\)-th
Taylor coefficient of \(V^\varphi_{\alpha,\beta}f\) is
\(\Gamma(N+1-\alpha)c_N(f)/\Gamma(N+1)\ne0\),
after increasing \(k\) if necessary.  If \(\alpha=m\in\N\), admissibility
forces \(\beta\in\N\), and formula \((2.6e)\), with \(N\ge m\), gives
the same conclusion.  Hence
\(G_\varphi\equiv0\Longleftrightarrow
V^\varphi_{\alpha,\beta}\equiv0\).
This proves the zero-operator alternatives.  The compactness criteria follow
from the same substitutions and are automatic in the zero-operator regimes.
\end{proof}

\begin{proposition}\label{prop:classical-volterra-check}
Let \(0<p,q<\infty\), \(\omega\in\calD\), and
\(\varphi\in H(\D)\).  If \(\alpha=\beta=1\), then the boundedness and
compactness criteria above reduce to the weighted theorem for the classical
Volterra operator established in~\cite[Theorems~1.1--1.2]{DuanWangWang2021}:
\[
        J_\varphi f(z)=\int_0^z f(\lambda)\varphi'(\lambda)\,d\lambda.
\]
\end{proposition}

\begin{proof}
When \(\alpha=\beta=1\), \(\nu_{\alpha,\beta,p}=\omega\), and the original
realization coincides with the classical Volterra operator \(J_\varphi\).  In
the modified coefficient model, \(\IRL_{-1}\varphi(z)=z\varphi'(z)\).
The factor \(z\) is bounded by one on \(\D\) and bounded away from zero on
every outer annulus \(\{|z|>r\}\).  It therefore does not affect the Carleson,
tent-space, maximal, or boundary-vanishing conditions.  On a fixed inner
disc, local boundedness, the positive lower bounds for \(\widehat\omega\) and
\(\delta\), and finite area make each density integrable; hence the inner
contribution cannot affect a boundary supremum, vanishing condition, or outer
tail.  The four conditions therefore reduce to
\begin{align*}
&\sup_{z\in\D}|\varphi'(z)|\widehat\omega(z)^{-1/p}
  \delta(z)^{1+1/q-1/p}<\infty,\\
&|\varphi'(z)|^{\frac{2p}{p-2}}
  \frac{\delta(z)^{\frac{p+2}{p-2}}}
       {\widehat\omega(z)^{\frac{2}{p-2}}}\dA(z)
  \quad\hbox{is a Carleson measure},\\
&\left(\int_{\Gamma(\xi)}|\varphi'(z)|^{\frac{2p}{p-2}}
  \left(\frac{\delta(z)^2}{\widehat\omega(z)}\right)^{\frac{2}{p-2}}
  \dA(z)\right)^{\frac{p-2}{2p}}
  \in L^{pq/(p-q)}(\T),\\
&\sup_{z\in\Gamma(\xi)}|\varphi'(z)|
  \frac{\delta(z)}{\widehat\omega(z)^{1/p}}
  \in L^{pq/(p-q)}(\T).
\end{align*}
In their respective parameter ranges, these agree with the boundedness
conditions in~\cite[Theorem~1.1]{DuanWangWang2021}.  The compactness
conditions likewise agree with~\cite[Theorem~1.2]{DuanWangWang2021}: the
boundary limit in Case A, the vanishing Carleson condition in Case B, the same
tent condition as for boundedness in Case C, and the outer maximal-tail
condition in Case D.
\end{proof}

\section*{Declarations}

\noindent\textbf{Funding.}
X. Zhu is supported by National Natural Science Foundation (NNSF Tianyuan) of China (Grant No. 12526614).\par
\smallskip
\noindent\textbf{Declaration of competing interest.}
The authors declare that they have no known competing financial interests or
personal relationships that could have appeared to influence the work
reported in this paper.\par
\smallskip
\noindent\textbf{Data availability.}
No data were used for the research described in this article.\par
\smallskip
\noindent\textbf{AI Statement.}
The authors used artificial-intelligence tools for language editing,
\LaTeX\ formatting, and limited assistance with local mathematical
reasoning.  The proof strategy and mathematical development are the
authors' own; all mathematical content was independently written and checked by the authors, who take full responsibility for it.

\end{document}